\documentclass[10pt]{amsart}

\usepackage{txfonts}
\usepackage[T1]{fontenc}
\usepackage{fancyhdr}

\usepackage[utf8]{inputenc} 
\usepackage{geometry} 
\usepackage{booktabs} 
\usepackage{array} 
\usepackage{paralist} 
\usepackage{verbatim} 
\usepackage{tikz}
\usetikzlibrary{arrows}
\usetikzlibrary{decorations.markings}
\usepackage{subfig}
\usepackage{tabularx}
\usepackage{enumitem}
\usepackage{amsmath,amsthm,amscd}
\usepackage{amssymb,amsfonts}
\usepackage{fancyhdr} 
\usepackage{pgf,tikz}
\usepackage{caption}
\usepackage{tikz-cd}
\usepackage{tikzscale}
\usetikzlibrary{patterns}
\usepackage{rotating}
\usepackage{xcolor}
\usepackage{lscape}
\usepackage{xspace}
\usepackage{xfrac}
\usepackage{blindtext}
\usepackage{bm}
\usepackage[ruled,vlined]{algorithm2e}

\usepackage{float}

\usepackage{adjustbox}

\usetikzlibrary{decorations.markings}
\usetikzlibrary{patterns}
\usetikzlibrary{calc}
\usepackage{soul}

\usepackage{tikz}
        \usetikzlibrary{patterns,hobby}
        \usepackage{pgfplots}
        \pgfplotsset{compat=1.6}

\usepackage{faktor} 
\usepackage{pgf,tikz}
\usepackage{svg}
\usetikzlibrary{arrows.meta}
\usetikzlibrary{decorations.markings}
\usetikzlibrary{patterns}

\usepackage{color}   
\usepackage{hyperref}
\hypersetup{
    colorlinks=true, 
    linktoc=all,     
    linkcolor=blue,  
    citecolor=blue,
    filecolor=blue,
    urlcolor=blue
}

\usepackage[T1]{fontenc}
\usepackage{hyperref}
\usepackage{nameref}
\usepackage{lscape}

\usepackage{blkarray}
\newcommand{\matindex}[1]{\mbox{\scriptsize#1}}

\def\p#1{\ensuremath{\bar {#1}}}
\newtheorem{thm}{Theorem}[section]
\newtheorem{cor}[thm]{Corollary}
\newtheorem{prop}[thm]{Proposition}
\newtheorem{lem}[thm]{Lemma}
\theoremstyle{definition}
\newtheorem{defn}[thm]{Definition}
\theoremstyle{remark}
\newtheorem{rmk}[thm]{Remark}
\theoremstyle{definition}

\theoremstyle{definition}
\newtheorem{ex}[thm]{Example}
\theoremstyle{definition}

\numberwithin{equation}{section}
\title{A lower bound for the genus of a knot using the Links-Gould invariant}

\author[B.-M. Kohli]{Ben-Michael Kohli}
\address[Ben-Michael Kohli]{Section de Math\'ematiques, Universit\'e de Gen\`eve \\
rue du Conseil-G\'en\'eral 7-9, 1205 Gen\`eve, Switzerland}
\email{bm.kohli@protonmail.ch}

\author{Guillaume Tahar}
\address[Guillaume Tahar]{Beijing Institute of Mathematical Sciences and Applications, Huairou District, Beijing, China}
\email{guillaume.tahar@bimsa.cn}

\date{\today}

\begin{document}
\begin{abstract}
The Links-Gould invariant of links $LG^{2,1}$ is a two-variable generalization of the Alexander-Conway polynomial. Using representation theory of $U_{q}\mathfrak{gl}(2 \vert 1)$, we prove that the degree of the Links-Gould polynomial provides a lower bound on the Seifert genus of any knot, therefore improving the bound known as the Seifert inequality in the case of the Alexander polynomial. As an example, unlike some classical tools such as the Alexander polynomial and Levine-Tristram signature, this new genus bound detects the fact that the Kinoshita-Terasaka and Conway knots have genus greater or equal to 2.
\end{abstract}
\maketitle
\setcounter{tocdepth}{1}
\tableofcontents

\section{Introduction}

The Links-Gould invariants of oriented links $LG^{m,n}(L,t_{0},t_{1})$ are obtained by the Reshetikhin-Turaev construction applied to the Hopf superalgebras $U_{q}\mathfrak{gl}(m \vert n)$. We will specifically study $U_{q}\mathfrak{gl}(2 \vert 1)$ and we shall denote by $LG$ the associated invariant $LG^{2,1}$.
\par
For any oriented link $L$, $LG(L,t_{0},t_{1})$ is a two-variable Laurent polynomial in variables $t_{0}$ and $t_{1}$. It is known to be a generalization of the Alexander-Conway polynomial $\Delta_{L}(t_{0})$ in at least two different ways (see \cite{DWIL,Ko2,KoPat}):
    $$LG(L,t_{0},-t_{0}^{-1})=\Delta_{L}(t_{0}^{2})\text{ ;}$$
    $$LG(L,t_{0},t_{0}^{-1})=\Delta_{L}(t_{0})^{2}.$$

Therefore, using the Seifert inequality \cite{Sei} for the Alexander polynomial $\Delta_{K}$ of a knot $K$, the minimal genus $g(K)$ among all possible Seifert surfaces for $K$ satisfies the following inequality:
\begin{equation}\label{eq:onevariable}
span(LG(K,t_{0},t_{0}^{-1})) \leq 4 \times g(K)    
\end{equation}
where the degree $span(LG(K,t_{0},t_{0}^{-1}))$ of the Laurent polynomial $LG(K,t_{0},t_{0}^{-1})$ is the difference between the degree of the monomial of highest degree in $t_{0}$ and that of the monomial of lowest degree in $t_{0}$ (after identification of $t_{1}$ with $t_{0}^{-1}$).\newline

That is why, defining $deg(t_{0}^{a}t_{1}^{b})$ to be $a-b$, one can wonder whether the following stronger\footnote{ Inequalities~\eqref{eq:onevariable} and \eqref{eq:twovariable} are not equivalent. For example, Laurent polynomial $(t_{0}t_{1}-1)(t_{0}+t_{1})$ has span $2$ while its specialization for $t_{1}=t_{0}^{-1}$ is zero. We will show that \eqref{eq:twovariable} improves \eqref{eq:onevariable} for certain knots.} inequality holds:
\begin{equation}\label{eq:twovariable}
span(LG(K,t_{0},t_{1})) \leq 4 \times g(K).    
\end{equation}

This is one of many features of the Links-Gould invariant, proven or conjectured, that illustrate the fact that this quantum invariant has deep topological meaning (see in particular \cite{Ish, Ko}). In general, it is hard to deduce precise topological properties of a knot or link from the value quantum invariants take on that link. For instance, the span of the Jones polynomial cannot in general be a lower bound for the 3-genus of a knot\footnote{ The family of twist knots $K_n$ all have genus 1 and the span of their Jones polynomials $V(K_n,q)$ grows linearly as $n$ grows. So there can be no $M>0$ such that $span(V(K_n,q)) \leq M \times g(K_n)$ for all $n$.}.\newline

In this paper we prove inequality~\eqref{eq:twovariable} that had been conjectured in \cite{Ko}. 

\begin{thm}\label{thm:MAIN}
Setting $span(\sum\limits_{i,j} a_{i,j} t_{0}^{i}t_{1}^{j}) = max \lbrace{ i-j~\vert~a_{i,j} \neq 0) \rbrace} - min \lbrace{ i-j~\vert~a_{i,j} \neq 0) \rbrace}$, for any knot $K$:
$$span(LG(K,t_{0},t_{1})) \leq 4 \times g(K).$$

\end{thm}
In particular, ~\eqref{eq:onevariable} shows that the bound obtained from Theorem~\ref{thm:MAIN} is an enhancement of the classical lower bound for the genus of a knot given by the Alexander invariant. Note that conjecturally, similar genus bounds should exist for all $LG^{n,1}$, see \cite{Ko}. This conjecture for $n \geq 3$ remains open. \\

As it is explained in \cite{Ko}, Proposition 1.13, a consequence of this genus bound is the following.

\begin{cor}
    For any knot $K$ for which the Alexander genus bound is sharp, the inequality in Theorem~\ref{thm:MAIN} is an equality. This is for example the case for all alternating knots and all fibered knots \cite{Mu, Rap}.
\end{cor}

One example where the bound given by Theorem~\ref{thm:MAIN} improves the usual bound obtained by computing the Alexander polynomial is if you compute $LG$ on the Kinoshita-Terasaka and Conway pair of mutant knots. The new genus bound proves that both these knots, that have the same $LG$, are of genus greater or equal to 2. This is significant since both the Alexander invariant and the Levine-Tristram signature fail to detect genus for the pair. Of course, by work of Gabai \cite{Gab}, it is known that the Conway knot has genus $3$ while the Kinoshita-Terasaka knot has genus $2$. Still, this means that the techniques developed in this work permit the computation of the genus of the Kinoshita-Terasaka knot, while they fail to compute the genus of the Conway knot. However, in subsequent work it will be shown how the genus bound for $LG$ implies the existence of genus bounds for \textit{colored} Links-Gould invariants \cite{GHKKW} that will compute the genus of the Conway knot, and even the genus of all prime knots with up to $16$ crossings, see \cite{GL}.\newline

This study should be understood in parallel with recent work by López Neumann and Van der Veen \cite{LNV, LNV2}, in particular regarding their results relative to $ADO$ invariants. Indeed, the Links-Gould invariants and the $ADO$ invariants are two families of generalizations of the Alexander polynomial. Moreover, it is conjectured that the single-colored $ADO_3$ invariant of links can be recovered as a well chosen specialization of the Links-Gould polynomial, a result conjectured by Geer and Patureau-Mirand in the multi-variable setting \cite{GP}, and proved for closures of $5$-braids by Takenov for a single color \cite{Ta}.\newline

While the core of our study focuses ok knots, we conclude this work by extending our results to links.

\subsection*{Organization of the paper}

Theorem~\ref{thm:MAIN} is proved in Section~\ref{sec:Topology}. The genus bound is seen as a consequence of a ``factorization'' of the diagram of a knot $K$ through the Reshetikhin-Turaev functor in a way that is controlled by the genus of $K$. For that we use the so called bottom tangle presentation of a Seifert surface of a knot \cite{Hab}. This construction follows in spirit the recent work by López Neumann and Van der Veen (see \cite{LNV}).
\par
One crucial step to use that factorization is a representation-theoretic lemma obtained in Section~\ref{sec:RepTheo}. The Links-Gould invariant is associated to a family of highest weight supermodules over $U_{q}\mathfrak{gl}(2 \vert 1)$. Equivalently, it can be obtained from a family of modules $(V_{\alpha})$ over the bosonization $U_{q}\mathfrak{gl}(2 \vert 1)^{\sigma}$ of the superalgebra $U_{q}\mathfrak{gl}(2 \vert 1)$. In Theorem~\ref{thm:MAINRT}, we prove that the tensor product representation $V_{\alpha} \otimes V_{\alpha}^{\ast}$ does not depend on the parameter $\alpha$. For that purpose we provide an explicit change of basis between $U_{q}\mathfrak{gl}(2 \vert 1)^{\sigma}$-modules. We also need to understand the dual action of $U_{q}\mathfrak{gl}(2 \vert 1)^{\sigma}$ on $(V_{\alpha} \otimes V_{\alpha}^{\ast})^{\ast}$ which we do at the end of Section~\ref{sec:RepTheo}. 
\par
Some topological consequences of this new genus bound are presented in Section~\ref{sec:Consequences}. Finally, the genus bound is extended to links in an Appendix.
\newline

\subsection*{Conventions}

All along the paper, we stick to the conventions David De Wit uses in his Ph.D. dissertation (see \cite{DW}). Specifically, $LG$ is a two-variable Laurent polynomial in variables $t_0$ and $t_1$. But it can also be expressed using variables $q$ and $q^{\alpha}$ with the following dictionary: $$t_{0}^{1/2}=q^{-\alpha} \text{, } t_{1}^{1/2}=q^{\alpha}q.$$




In this setting, the degree defined in $\mathbb{Z}[t_0^{\pm 1},t_1^{\pm 1}]$ by $deg(t_{0}^{a}t_{1}^{b}) = a-b$ is none other than the degree with respect to variable $q^{\alpha}$ in $\mathbb{Z}(q)[q^{\alpha},q^{-\alpha}]$ with:
$$deg(q^{2 n \alpha}) = - n.$$ The genus bound from Theorem~\ref{thm:MAIN} becomes:
$$span_{q^{2\alpha}}(LG(K,q,q^{\alpha}))\leq 4 \times g(K).$$

\subsection*{Acknowledgments}

The first author would like to thank BIMSA for the hospitality during his stay in Beijing in September of 2023, when this project was initiated and carried out. The authors would also like to thank Bertrand Patureau-Mirand very warmly for his insightful advice all through the process, as well as Micah Chrisman, David Cimasoni, Daniel López Neumann, Matthew Harper, Jon Links, Emmanuel Wagner, Roland Van der Veen and the anonymous referees for their valuable remarks.

\section{Representations of Hopf superalgebra $U_{q}\mathfrak{gl}(2 \vert 1)$}\label{sec:RepTheo}

 But let us first introduce the objects at hand. A \emph{super vector space} is a $\mathbb{Z}/2\mathbb{Z}$-graded vector space $V=V_{\p0}\oplus V_{\p1}$. The parity of a homogeneous element $v\in V$ is denoted $[v] \in \mathbb{Z}/2\mathbb{Z}$. Morphisms of super vector spaces are parity preserving linear maps. The tensor product of super vector spaces $V$ and $W$ is the tensor product of the underlying vector spaces with $\mathbb{Z}/2\mathbb{Z}$-grading given by
$$
\begin{aligned}
    (V \otimes W)_{\p p} = \bigoplus_{\p a + \p b = \p p} V_{\p a} \otimes W_{\p b}.
\end{aligned}
$$
A (left) module over a (associative) superalgebra $A$ is a super vector space $M$ with an $A$-module structure for which the action map $A \otimes M \rightarrow M$ is a morphism of super vector spaces. Given homogeneous elements $a,b \in A$, set $[a,b] = ab -(-1)^{[a] [b]} ba$.
\subsection{Generators and relations for Hopf superalgebra $U_{q}\mathfrak{gl}(2 \vert 1)$}
\subsubsection{Generators}
Superalgebra $U_{q}\mathfrak{gl}(2 \vert 1)$ is a $\mathbb{C}$-superalgebra defined by the following seven elementary generators:
\begin{itemize}
    \item Cartan generators $q^{E_{1}^{1}}$, $q^{E_{2}^{2}}$ and $q^{E_{3}^{3}}$;
    \item Raising generators $E_{2}^{1}$ and $E_{3}^{2}$;
    \item Lowering generators $E_{1}^{2}$ and $E_{2}^{3}$.
\end{itemize}
\
Elements $q^{E_{i}^{i}}$ are chosen to be invertible and we refer to their inverses as $q^{-E_{i}^{i}}$. We also denote by $q^{\frac{1}{2}E_{i}^{i}}$ some choice of a square root for $q^{E_{i}^{i}}$.

In addition to the 7 generators, we will use two other elements defined in terms of them:
\begin{itemize}
    \item $E_{3}^{1}=E_{2}^{1}E_{3}^{2}-q^{-1}E_{3}^{2}E_{2}^{1}$;
    \item $E_{1}^{3}=E_{2}^{3}E_{1}^{2}-qE_{1}^{2}E_{2}^{3}$.
\end{itemize} 

\begin{rmk}
This choice of variables is coherent with \cite{Zh2} up to the fact that what Zhang calls $K_{a}$ corresponds to what we denote $q^{(-1)^{\vert a \vert }E_{a}^{a}}$ in the following paragraph.
\end{rmk}

\subsubsection{$\mathbb{Z}/2\mathbb{Z}$-grading}
We define a $\mathbb{Z}/2\mathbb{Z}$-grading on $U_{q}\mathfrak{gl}(2 \vert 1)$ by fixing $q^{E_{1}^{1}}$, $q^{E_{2}^{2}}$, $q^{E_{3}^{3}}$, $E_{2}^{1}$ and $E_{1}^{2}$ to be even while declaring the two last generators $E_{3}^{2}$ and $E_{2}^{3}$ odd. We can extend that grading to the rest of the algebra. Setting $[x]$ the degree of homogeneous $x$ ($[x]=0$ if $x$ is even and $[x]=1$ if $x$ is odd), the product of homogeneous $x$, $y$ $\in U_{q}\mathfrak{gl}(2 \vert 1)$ has degree:
$$[xy] = [x] + [y] \text{ modulo 2.}$$
In particular, $E_{3}^{1}$ and $E_{1}^{3}$ are odd.

\subsubsection{Relations}

\begin{defn}
Whenever it makes sense, we will use the following q-bracket notation:
$$[X]_q = \frac{q^X - q^{-X}}{q - q^{-1}}.$$
\end{defn}That being said, the relations defining superalgebra $U_{q}\mathfrak{gl}(2 \vert 1)$ are:
\begin{itemize}
    \item The Cartan generators $q^{E_{1}^{1}}$, $q^{E_{2}^{2}}$, $q^{E_{3}^{3}}$ all commute;
    \item The raising and lowering generators commute with the Cartan genarators as follows:
    $$q^{(-1)^{\vert a\vert}E^a_a}E^b_{b\pm 1}q^{-(-1)^{\vert a \vert}E^a_a} = q^{(-1)^{\vert a \vert}(\delta^a_b - \delta^a_{b\pm 1})}E^b_{b\pm 1},$$ $\text{where }\vert 1 \vert = \vert 2 \vert = 0 \text{, } \vert 3 \vert = 1 \text{ and } \delta^i_j = 1 \text{ if and only if } i=j;$
    \item $(E_{2}^{3})^{2}=(E_{3}^{2})^{2}=0$ (the squares of odd generators are zero);
    \item $E^1_{2}$ commutes with $E_{2}^{3}$ while $E_{1}^{2}$ commutes with $E_{3}^{2}$;
    \item The non Cartan generators also satisfy the following interchange rules:
    $$E_{2}^{1}E_{1}^{2} = [E_{1}^{1} - E_{2}^{2}]_q + E_{1}^{2} E_{2}^{1},$$
    $$E_{3}^{2}E_{2}^{3} = [E_{2}^{2} + E_{3}^{3}]_q - E_{2}^{3} E_{3}^{2}.$$
\end{itemize}

The Serre relations can be written in terms of generators, but are more easily understood when expressed using $E^3_{1}$ and $E^1_{3}$:
$$E_{2}^{1} E_{3}^{1} - q E_{3}^{1} E_{2}^{1} =0;$$
$$E_{1}^{3} E_{1}^{2} - q^{-1} E_{1}^{2} E_{1}^{3} =0.$$

Other interesting relations can be deduced from the defining ones. They are all summed up in \cite{DW}, Sections 4.2.3 and 4.2.4. 

\subsection{A Hopf superalgebra structure on \textbf{$U_{q}\mathfrak{gl}(2 \vert 1)$}}

Superalgebra $U_{q}\mathfrak{gl}(2 \vert 1)$ is equipped with a coproduct $\Delta$, a co-unit $\epsilon$ and an antipode $S$, turning it into a Hopf superalgebra.\newline

\subsubsection{Coproduct}

Coproduct $\Delta: U_{q}\mathfrak{gl}(2 \vert 1) \rightarrow U_{q}\mathfrak{gl}(2 \vert 1) \otimes U_{q}\mathfrak{gl}(2 \vert 1)$ is an algebra homomorphism defined by $\Delta(xy)=\Delta(x)\Delta(y) \text{ for any } x,y \in U_{q}\mathfrak{gl}(2 \vert 1) \text{, by }$
    $\Delta(1)=1 \otimes 1$ and by:

    $$\text{ for any }i\text{, } \Delta(q^{E_{i}^{i}})= q^{E_{i}^{i}} \otimes q^{E_{i}^{i}};$$
    $$\Delta(E_{2}^{1})= E_{2}^{1} \otimes q^{-\frac{1}{2}(E_{1}^{1}-E_{2}^{2})} + q^{\frac{1}{2}(E_{1}^{1}-E_{2}^{2})}\otimes E_{2}^{1};$$
    $$\Delta(E_{1}^{2})= E_{1}^{2} \otimes q^{-\frac{1}{2}(E_{1}^{1}-E_{2}^{2})} + q^{\frac{1}{2}(E_{1}^{1}-E_{2}^{2})}\otimes E_{1}^{2};$$
    $$\Delta(E_{3}^{2})= E_{3}^{2} \otimes q^{-\frac{1}{2}(E_{2}^{2}+E_{3}^{3})} + q^{\frac{1}{2}(E_{2}^{2}+E_{3}^{3})}\otimes E_{3}^{2};$$
    $$\Delta(E_{2}^{3})= E_{2}^{3} \otimes q^{-\frac{1}{2}(E_{2}^{2}+E_{3}^{3})} + q^{\frac{1}{2}(E_{2}^{2}+E_{3}^{3})}\otimes E_{2}^{3}.$$

\begin{rmk}
Observe that coproduct $\Delta$ preserves grading: for any $x \in U_{q}\mathfrak{gl}(2 \vert 1)$, $[\Delta(x)]=[x]$.
\end{rmk}

It will be necessary to know the expression for coproducts $\Delta(E_{1}^{3})$ and $\Delta(E_{3}^{1})$. They have been computed in \cite{DW}, Section 4.5.3:
\begin{itemize}
    \item $\Delta(E_{1}^{3})= E_{1}^{3} \otimes q^{-\frac{1}{2}(E_{1}^{1}+E_{3}^{3})} + q^{\frac{1}{2}(E_{1}^{1}+E_{3}^{3})} E_{1}^{3} - (q - q^{-1}) q^{\frac{1}{2}(E_{1}^{1}-E_{2}^{2})} E_{2}^{3} \otimes E_{1}^{2} q^{-\frac{1}{2}(E_{2}^{2}+E_{3}^{3})}$;
    \item $\Delta(E_{3}^{1})=E_{3}^{1} \otimes q^{-\frac{1}{2}(E_{1}^{1}+E_{3}^{3})} + q^{\frac{1}{2}(E_{1}^{1}+E_{3}^{3})} E_{3}^{1} + (q - q^{-1}) q^{\frac{1}{2}(E_{2}^{2}+E_{3}^{3})} E_{2}^{1} \otimes E_{3}^{2} q^{-\frac{1}{2}(E_{1}^{1}-E_{2}^{2})}$.
\end{itemize}

\subsubsection{Co-unit}

Co-unit $\epsilon: U_{q}\mathfrak{gl}(2 \vert 1) \rightarrow \mathbb{C}$ is an algebra homomorphism satisfying:
 $$\epsilon(q^{E_{1}^{1}})=\epsilon(q^{E_{2}^{2}})=\epsilon(q^{E_{3}^{3}})=\epsilon(1)=1;$$
$$\epsilon(E_{2}^{1})=\epsilon(E_{1}^{2})=\epsilon(E_{3}^{2})=\epsilon(E_{2}^{3})=0.$$

\subsubsection{Antipode}

Antipode $S:U_{q}\mathfrak{gl}(2 \vert 1) \rightarrow U_{q}\mathfrak{gl}(2 \vert 1)$ is a graded algebra antihomomorphism. It satisfies $S(1)=1$ and $S(xy)=(-1)^{[x][y]}S(y)S(x)$ for any $x,y \in U_{q}\mathfrak{gl}(2 \vert 1)$. Also:
$$S(q^{E_{i}^{i}})=q^{-E_{i}^{i}} \text{ for } i = 1, 2, 3;$$
$$S(E_{2}^{1}) = -q^{-1}E_{2}^{1} \text{ and }  S(E_{1}^{2}) = -q E_{1}^{2};$$
$$S(E_{3}^{2}) = -E_{3}^{2} \text{ and } S(E_{2}^{3})=-E_{2}^{3}.$$

Superalgebra $U_{q}\mathfrak{gl}(2 \vert 1)$ equipped with $\Delta,\epsilon,S$ satisfies the standard axioms for Hopf superalgebras, see \cite{Zh1}.

\begin{rmk}
    Some authors use the Hopf superalgebra $U_q\mathfrak{sl}(2\vert1)$ to define the Links--Gould polynomial rather than $U_q\mathfrak{gl}(2\vert1)$; see, for example, \cite{GP,GPM10}. One passes from $\mathfrak{gl}$ to $\mathfrak{sl}$ by removing a central Cartan generator. This does not affect the resulting invariants, although the representation theory of $U_q\mathfrak{gl}(2\vert1)$ is slightly richer than that of $U_q\mathfrak{sl}(2\vert1)$, which can occasionally be advantageous.
\end{rmk}

\subsection{Highest weight representations on $U_{q}\mathfrak{gl}(2 \vert 1)$}

 A super vector space $V=V^{even} \oplus V^{odd}$ is a representation of $U_{q}\mathfrak{gl}(2 \vert 1)$ when $U_{q}\mathfrak{gl}(2 \vert 1)$ acts on $V$ from the left in a way that preserves the grading. In other words, for $a \in U_{q}\mathfrak{gl}(2 \vert 1)$ and $x$ an homogeneous element in $V$, we have $[a.x]=[a].[x]$ (where $[x]=0$ if $x \in V^{even}$ and $[x]=1$ if $x \in V^{odd}$).\newline

We introduce a family of $4$-dimensional representations of  $U_{q}\mathfrak{gl}(2 \vert 1)$ that will help us define the Links-Gould invariant. It is a family of irreducible typical highest weight representations $(V_{\alpha})$, $\alpha \neq -1,0$. Irreducible representations of $\mathrm{U}_q\mathfrak{gl}(2|1)$ are obtained from deformations of representations of the classical supergroup $\mathfrak{gl}(2|1)$ by results of Kac and Geer \cite{Kac, Geer07}, also see \cite{GPM07,GPM10}  and for a classification in the root of unity case \cite{Anghel-Geer}. They are $4n$-dimensional graded vector spaces of the form $V_p(n-1,\alpha)$ where $n$ is an integer with $n \geq 1$,
$\alpha$ is a continuous parameter, and $p\in\mathbb{Z}/2\mathbb{Z}$. The family of representations we denote by $(V_{\alpha})$ is $V_0(0,\alpha)$ in that more general setting. Let us write it down in matrix form following \cite{DW}.

\begin{defn}\label{defn:Valpha}
For any $\alpha \in \mathbb{C}\setminus \lbrace{ -1,0 \rbrace}$, $V_{\alpha}$ is a $(2\vert2)$ super vector space spanned by $e_{1},e_{2},e_{3}\text{ and }e_{4}$ where $[e_{1}]=[e_{4}]=0$ while $[e_{2}]=[e_{3}]=1$.
\par
The matrices representing the action of the generators of $U_{q}\mathfrak{gl}(2 \vert 1)$ on $V_{\alpha}$ are as follows:
$$
\pi_\alpha(q^{E_{1}^{1}})=
\begin{pmatrix}
1 & 0 & 0 & 0 \\
0 & 1 & 0 & 0 \\
0 & 0 & q^{-1} & 0 \\
0 & 0 & 0 & q^{-1}
\end{pmatrix};
\text{     }
\pi_\alpha(q^{E_{2}^{2}})=
\begin{pmatrix}
1 & 0 & 0 & 0 \\
0 & q^{-1} & 0 & 0 \\
0 & 0 & 1 & 0 \\
0 & 0 & 0 & q^{-1}
\end{pmatrix};
\text{     }
\pi_\alpha(q^{E_{3}^{3}})=
\begin{pmatrix}
q^{\alpha} & 0 & 0 & 0 \\
0 & q^{\alpha+1} & 0 & 0 \\
0 & 0 & q^{\alpha+1} & 0 \\
0 & 0 & 0 & q^{\alpha+2}
\end{pmatrix};
$$
$$
\pi_\alpha(E_{2}^{1})=
\begin{pmatrix}
0 & 0 & 0 & 0 \\
0 & 0 & -1 & 0 \\
0 & 0 & 0 & 0 \\
0 & 0 & 0 & 0
\end{pmatrix};
\text{     }
\pi_\alpha(E_{1}^{2})=
\begin{pmatrix}
0 & 0 & 0 & 0 \\
0 & 0 & 0 & 0 \\
0 & -1 & 0 & 0 \\
0 & 0 & 0 & 0
\end{pmatrix};
\text{     }
\pi_\alpha(E_{3}^{2})=
\begin{pmatrix}
0 & 1 & 0 & 0 \\
0 & 0 & 0 & 0 \\
0 & 0 & 0 & 1 \\
0 & 0 & 0 & 0
\end{pmatrix};
$$
$$
\pi_\alpha(E_{2}^{3})
=
\begin{pmatrix}
0 & 0 & 0 & 0 \\
[\alpha]_{q} & 0 & 0 & 0 \\
0 & 0 & 0 & 0 \\
0 & 0 & [\alpha + 1]_{q} & 0
\end{pmatrix};
\text{     }
\pi_\alpha(E_{3}^{1})=
\begin{pmatrix}
0 & 0 & q^{-1} & 0 \\
0 & 0 & 0 & -1 \\
0 & 0 & 0 & 0 \\
0 & 0 & 0 & 0
\end{pmatrix};
\text{     }
\pi_\alpha(E_{1}^{3})=
\begin{pmatrix}
0 & 0 & 0 & 0 \\
0 & 0 & 0 & 0 \\
q[\alpha]_{q} & 0 & 0 & 0 \\
0 & -[\alpha+1]_{q} & 0 & 0
\end{pmatrix}
,
$$
where $[x]_{q}$ stands for $\frac{q^{x}-q^{-x}}{q-q^{-1}}$.
\end{defn}

This family of representations is a one-parameter family of highest weight irreducible representations of $U_{q}\mathfrak{gl}(2 \vert 1)$ (see \cite{GHL}).

\begin{rmk}
Matrices in Definition~\ref{defn:Valpha} are usual matrices and not super-matrices.
\end{rmk}

\begin{rmk}
In order to write these matrices we modified the basis used in \cite{DW}:
\begin{itemize}
    \item $|1 \rangle$ is replaced by $\frac{[\alpha]_{q}^{1/2}}{[\alpha+1]_{q}^{1/2}}|1 \rangle =e_{1}$;
    \item $|2 \rangle$ is replaced by $\frac{1}{[\alpha+1]_{q}^{1/2}}| 2 \rangle = e_{2}$;
    \item $| 3 \rangle$ is replaced by $\frac{1}{[\alpha+1]_{q}^{1/2}} |3 \rangle =e_{3}$;
    \item $| 4 \rangle $ is replaced by $\frac{1}{[\alpha+1]_{q}} | 4 \rangle = e_{4}$.
\end{itemize}
\end{rmk}

\begin{rmk}
    One can verify that the following relation holds. It is a useful tool for computations in this context:
    $$[\alpha+1]_{q} q^{\alpha} - [\alpha]_{q} q^{\alpha+1} = 1.$$
\end{rmk}

\subsection{Bosonization of Hopf superalgebra $U_{q}\mathfrak{gl}(2 \vert 1)$}
In order to make computations a bit more straightforward, we will not in practice use Hopf superalgebra $U_{q}\mathfrak{gl}(2 \vert 1)$, but its non super counterpart $U_{q}\mathfrak{gl}(2 \vert 1)^{\sigma}$. This transformation, known as bosonization, is due to Majid, see \cite{Ma}.

\begin{thm}
Set $(H, \Delta,\epsilon,S)$ a Hopf superalgebra. One can define an ordinary Hopf algebra $H^{\sigma}$ as follows. As an algebra, $H^{\sigma}$ is an extension of $H$ by adjoining an element ${\sigma}$ subject to the following commutation relations:
$$\forall x \in H \text{ homogeneous, } \sigma x = (-1)^{[x]} x \sigma.$$ The coproduct $\Delta^{\sigma}$ on $H^{\sigma}$ is given by $\Delta^{\sigma}(\sigma) = \sigma \otimes \sigma$ and
$$\forall x \in H \text{, } \Delta^{\sigma}(x) = \displaystyle\sum_{i} x_i \sigma^{[y_i]} \otimes y_i \text{ where } \Delta(x) = \displaystyle\sum_{i} x_i  \otimes y_i.$$ 

The counit $\epsilon^{\sigma}$ satisfies $\epsilon^{\sigma}(\sigma) = 1$ and $\epsilon^{\sigma}(x) = \epsilon(x)$ otherwise.

The antipode $S^{\sigma}$ is given by $S^{\sigma}(\sigma) = \sigma$ and $$S^{\sigma}(x)= \sigma^{[x]} S(x) \text{ } \forall x \in H \text{ homogeneous.}$$
\end{thm}

Note that $S$ is a graded algebra antihomomorphism and therefore $S^{\sigma}$ is an ordinary algebra antihomomorphism.

\begin{thm}
For any superalgebra $H$, the category of super $H$-modules (where arrows are even morphisms) is equivalent to the category of $H^{\sigma}$-modules.
\end{thm}

The previous result also follows from Majid's work, see \cite{Ma}. This is essentially due to the fact that if you consider a super-representation $V = V_{even} \oplus V_{odd}$ of $H$, then you get a representation of $H^{\sigma}$ by setting $\sigma_V = \text{Id}_{V_{even}}\oplus -\text{Id}_{V_{odd}}$. Conversely, since $\sigma^2 = 1$, every $H^{\sigma}$-module inherits a natural $\mathbb{Z}/2\mathbb{Z}$ grading: any $V$ splits into $V = V_{even} \oplus V_{odd}$ where $V_{even} = \ker (\sigma -1)$ and $V_{odd} = \ker (\sigma +1)$.

That is why in the following we will consider Hopf algebra $U_{q}\mathfrak{gl}(2 \vert 1)^{\sigma}$ rather than Hopf superalgebra $U_{q}\mathfrak{gl}(2 \vert 1)$. They tell the same story representation-wise.

\subsection{Dual representations $V_{\alpha}^{\ast}$}

We use $U_{q}\mathfrak{gl}(2 \vert 1)^{\sigma}$'s Hopf algebra structure to build the family of dual representations $(V_{\alpha}^{\ast})$. Setting $\pi_{\alpha}^{\ast}$ the dual representation of $\pi_{{\alpha}}$, we can write:\newline
$$
\forall f \in V_{\alpha}^{\ast}, \forall x \in U_{q}\mathfrak{gl}(2 \vert 1)^{\sigma} \text{ : } 
\pi_{\alpha}^{\ast}(x)(f) = f \circ \pi_{{\alpha}}(S^{\sigma}(x)) \in V_{\alpha}^{\ast}
.
$$

Computing the latter in basis $(e_{1}^{\ast},e_{2}^{\ast}, e_{3}^{\ast}, e_{4}^{\ast})$, we obtain for any $x \in U_{q}\mathfrak{gl}(2 \vert 1)^{\sigma}$ that:
$$\pi_{\alpha}^{\ast}(x)(e_{i}^{\ast}) = \sum\limits_{j=1}^{4} [\pi_{\alpha}^{\ast}(x)(e_{i}^{\ast})](e_{j})e_{j}^{\ast}.$$ This means that if $A$ is the matrix for $\pi_{{\alpha}}(S^{\sigma}(x))$ in basis $(e_{1},e_{2},e_{3},e_{4})$, then $A^{T}$ is the matrix for  $\pi_{\alpha}^{\ast}(x)$ in basis $(e_{1}^{\ast},e_{2}^{\ast},e_{3}^{\ast},e_{4}^{\ast})$.\newline

From that we deduce the matrices representing left multiplication by generators of $U_{q}\mathfrak{gl}(2 \vert 1)^{\sigma}$ on $V_{\alpha}^{\ast}$ :
$$
\pi_\alpha^{\ast}(q^{E_{1}^{1}})=
\begin{pmatrix}
1 & 0 & 0 & 0 \\
0 & 1 & 0 & 0 \\
0 & 0 & q & 0 \\
0 & 0 & 0 & q
\end{pmatrix};
\text{     }
\pi_\alpha^{\ast}(q^{E_{2}^{2}})=
\begin{pmatrix}
1 & 0 & 0 & 0 \\
0 & q & 0 & 0 \\
0 & 0 & 1 & 0 \\
0 & 0 & 0 & q
\end{pmatrix};
\text{     }
\pi_\alpha^{\ast}(q^{E_{3}^{3}})=
\begin{pmatrix}
q^{-\alpha} & 0 & 0 & 0 \\
0 & q^{-\alpha-1} & 0 & 0 \\
0 & 0 & q^{-\alpha-1} & 0 \\
0 & 0 & 0 & q^{-\alpha-2}
\end{pmatrix};
$$
\centerline{$
\pi_\alpha^{\ast}(E_{2}^{1})=
\begin{pmatrix}
0 & 0 & 0 & 0 \\
0 & 0 & 0 & 0 \\
0 & q^{-1} & 0 & 0 \\
0 & 0 & 0 & 0
\end{pmatrix};
\text{     }
\pi_\alpha^{\ast}(E_{1}^{2})=
\begin{pmatrix}
0 & 0 & 0 & 0 \\
0 & 0 & q & 0 \\
0 & 0 & 0 & 0 \\
0 & 0 & 0 & 0
\end{pmatrix};
\text{     }
\pi_\alpha^{\ast}(E_{3}^{2})=
\begin{pmatrix}
0 & 0 & 0 & 0 \\
-1 & 0 & 0 & 0 \\
0 & 0 & 0 & 0 \\
0 & 0 & 1 & 0
\end{pmatrix};
\text{     }
\pi_\alpha^{\ast}(E_{2}^{3})
=
\begin{pmatrix}
0 & [\alpha]_{q} & 0 & 0 \\
0 & 0 & 0 & 0 \\
0 & 0 & 0 & -[\alpha+1]_{q} \\
0 & 0 & 0 & 0
\end{pmatrix};
$}
$$
\pi_\alpha^{\ast}(E_{3}^{1})=
\begin{pmatrix}
0 & 0 & 0 & 0 \\
0 & 0 & 0 & 0 \\
-q^{-1} & 0 & 0 & 0 \\
0 & -q^{-2} & 0 & 0
\end{pmatrix};
\text{     }
\pi_\alpha^{\ast}(E_{1}^{3})=
\begin{pmatrix}
0 & 0 & q [\alpha]_{q} & 0\\
0 & 0 & 0 & q^{2}[\alpha+1]_{q} \\
0 & 0 & 0 & 0 \\
0 & 0 & 0 & 0
\end{pmatrix};
\text{     }
\pi_\alpha^{\ast}(\sigma)=
\begin{pmatrix}
1 & 0 & 0 & 0 \\
0 & -1 & 0 & 0 \\
0 & 0 & -1 & 0 \\
0 & 0 & 0 & 1
\end{pmatrix}.
$$

\subsection{Tensor representations $V_{\alpha} \otimes V_{\alpha}^{\ast}$}

Now we will compute the tensor product $V_{\alpha} \otimes V_{\alpha}^{\ast}$ of the representations we studied in the two previous paragraphs. To do so we need to build the tensor product of two representations in the context of Hopf algebras.

\begin{defn}
    If the actions of $U_{q}\mathfrak{gl}(2 \vert 1)^{\sigma}$ on two vector spaces $V$ and $W$ are denoted by $\rho_V$ and $\rho_W$, then $U_{q}\mathfrak{gl}(2 \vert 1)^{\sigma}$ acts on $V \otimes W$ through co-multiplication. Indeed, for any $a \in U_{q}\mathfrak{gl}(2 \vert 1)^{\sigma}$
    $$\rho_{V\otimes W }(a) = \rho_V\otimes\rho_W (\Delta^{\sigma} (a)),$$where the action of $U_{q}\mathfrak{gl}(2 \vert 1)^{\sigma} \otimes  U_{q}\mathfrak{gl}(2 \vert 1)^{\sigma}$ on $V \otimes W$ is defined naturally for $X, Y \in U_{q}\mathfrak{gl}(2 \vert 1)^{\sigma}$, $v \in V$ and $w \in W$ as follows:

    $$(X \otimes Y) (v \otimes w) =  (X.v \otimes Y.w).$$
\end{defn}

With that we can compute the action of $U_{q}\mathfrak{gl}(2 \vert 1)^{\sigma}$ on $V_{\alpha} \otimes V_{\alpha}^{\ast}$. We write the 16x16 matrices representing left multiplication on $V_{\alpha} \otimes V_{\alpha}^{\ast}$ in basis $(e_{1}\otimes e_{1}^{\ast},e_{1}\otimes e_{2}^{\ast},e_{1}\otimes e_{3}^{\ast},e_{1}\otimes e_{4}^{\ast}, e_{2}\otimes e_{1}^{\ast},e_{2}\otimes e_{2}^{\ast}, \ldots )$. The representation is denoted $\Pi_{\alpha}$.\newline
$\text{  }$\newline
$\Pi_{\alpha}(q^{E_{1}^{1}})$ is a diagonal matrix with entries $(1,1,q,q,1,1,q,q,q^{-1},q^{-1},1,1,q^{-1},q^{-1},1,1)$ ;\newline
$\text{  }$\newline
$\Pi_{\alpha}(q^{E_{2}^{2}})$ is a diagonal matrix with entries $(1,q,1,q,q^{-1},1,q^{-1},1,1,q,1,q,q^{-1},1,q^{-1},1)$ ;\newline
$\text{  }$\newline
$\Pi_{\alpha}(q^{E_{3}^{3}})$ is a diagonal matrix with entries $(1,q^{-1},q^{-1},q^{-2},q,1,1,q^{-1},q,1,1,q^{-1},q^{2},q,q,1)$ ;\newline
$\text{  }$\newline
$\Pi_{\alpha}(\sigma)$ is a diagonal matrix with entries $(1,-1, -1, 1, -1, 1, 1, -1, -1, 1, 1, -1,1,-1, -1, 1)$ ;\newline

$$
  \Pi_{\alpha}(E_{2}^{1})=\left(\begin{array}{@{}cccc|cccc|cccc|cccc@{}}
  0 & 0 & 0 & 0 & 0 & 0 & 0 & 0 & 0 & 0 & 0 & 0 & 0 & 0 & 0 & 0 \\
    0 & 0 & 0 & 0 & 0 & 0 & 0 & 0 & 0 & 0 & 0 & 0 & 0 & 0 & 0 & 0  \\
    0 & q^{-1} & 0 & 0 & 0 & 0 & 0 & 0 & 0 & 0 & 0 & 0 & 0 & 0 & 0 & 0 \\
    0 & 0 & 0 & 0 & 0 & 0 & 0 & 0 & 0 & 0 & 0 & 0 & 0 & 0 & 0 & 0  \\\hline
    0 & 0 & 0 & 0 & 0 & 0 & 0 & 0 & -1 & 0 & 0 & 0 & 0 & 0 & 0 & 0 \\
    0 & 0 & 0 & 0 & 0 & 0 & 0 & 0 & 0 & -q^{1/2} & 0 & 0 & 0 & 0 & 0 & 0  \\
    0 & 0 & 0 & 0 & 0 & q^{-1/2} & 0 & 0 & 0 & 0 & -q^{-1/2} & 0 & 0 & 0 & 0 & 0 \\
    0 & 0 & 0 & 0 & 0 & 0 & 0 & 0 & 0 & 0 & 0 & -1 & 0 & 0 & 0 & 0  \\\hline
    0 & 0 & 0 & 0 & 0 & 0 & 0 & 0 & 0 & 0 & 0 & 0 & 0 & 0 & 0 & 0 \\
    0 & 0 & 0 & 0 & 0 & 0 & 0 & 0 & 0 & 0 & 0 & 0 & 0 & 0 & 0 & 0  \\
    0 & 0 & 0 & 0 & 0 & 0 & 0 & 0 & 0 & q^{-3/2} & 0 & 0 & 0 & 0 & 0 & 0 \\
    0 & 0 & 0 & 0 & 0 & 0 & 0 & 0 & 0 & 0 & 0 & 0 & 0 & 0 & 0 & 0  \\\hline
    0 & 0 & 0 & 0 & 0 & 0 & 0 & 0 & 0 & 0 & 0 & 0 & 0 & 0 & 0 & 0 \\
    0 & 0 & 0 & 0 & 0 & 0 & 0 & 0 & 0 & 0 & 0 & 0 & 0 & 0 & 0 & 0  \\
    0 & 0 & 0 & 0 & 0 & 0 & 0 & 0 & 0 & 0 & 0 & 0 & 0 & q^{-1} & 0 & 0 \\
    0 & 0 & 0 & 0 & 0 & 0 & 0 & 0 & 0 & 0 & 0 & 0 & 0 & 0 & 0 & 0  \\
  \end{array}\right);
$$

$$
  \Pi_{\alpha}(E_{1}^{2})=\left(\begin{array}{@{}cccc|cccc|cccc|cccc@{}}
    0 & 0 & 0 & 0 & 0 & 0 & 0 & 0 & 0 & 0 & 0 & 0 & 0 & 0 & 0 & 0 \\
    0 & 0 & q & 0 & 0 & 0 & 0 & 0 & 0 & 0 & 0 & 0 & 0 & 0 & 0 & 0  \\
    0 & 0 & 0 & 0 & 0 & 0 & 0 & 0 & 0 & 0 & 0 & 0 & 0 & 0 & 0 & 0 \\
    0 & 0 & 0 & 0 & 0 & 0 & 0 & 0 & 0 & 0 & 0 & 0 & 0 & 0 & 0 & 0  \\\hline
    0 & 0 & 0 & 0 & 0 & 0 & 0 & 0 & 0 & 0 & 0 & 0 & 0 & 0 & 0 & 0 \\
    0 & 0 & 0 & 0 & 0 & 0 & q^{3/2} & 0 & 0 & 0 & 0 & 0 & 0 & 0 & 0 & 0  \\
    0 & 0 & 0 & 0 & 0 & 0 & 0 & 0 & 0 & 0 & 0 & 0 & 0 & 0 & 0 & 0 \\
    0 & 0 & 0 & 0 & 0 & 0 & 0 & 0 & 0 & 0 & 0 & 0 & 0 & 0 & 0 & 0  \\\hline
    0 & 0 & 0 & 0 & -1 & 0 & 0 & 0 & 0 & 0 & 0 & 0 & 0 & 0 & 0 & 0 \\
    0 & 0 & 0 & 0 & 0 & -q^{1/2} & 0 & 0 & 0 & 0 & q^{1/2} & 0 & 0 & 0 & 0 & 0  \\
    0 & 0 & 0 & 0 & 0 & 0 & -q^{-1/2} & 0 & 0 & 0 & 0 & 0 & 0 & 0 & 0 & 0 \\
    0 & 0 & 0 & 0 & 0 & 0 & 0 & -1 & 0 & 0 & 0 & 0 & 0 & 0 & 0 & 0  \\\hline
    0 & 0 & 0 & 0 & 0 & 0 & 0 & 0 & 0 & 0 & 0 & 0 & 0 & 0 & 0 & 0 \\
    0 & 0 & 0 & 0 & 0 & 0 & 0 & 0 & 0 & 0 & 0 & 0 & 0 & 0 & q & 0  \\
    0 & 0 & 0 & 0 & 0 & 0 & 0 & 0 & 0 & 0 & 0 & 0 & 0 & 0 & 0 & 0 \\
    0 & 0 & 0 & 0 & 0 & 0 & 0 & 0 & 0 & 0 & 0 & 0 & 0 & 0 & 0 & 0  \\
  \end{array}\right);
$$

$$
  \Pi_{\alpha}(E_{3}^{2})=\left(\begin{array}{@{}cccc|cccc|cccc|cccc@{}}
    0 & 0 & 0 & 0 & q^{\alpha/2} & 0 & 0 & 0 & 0 & 0 & 0 & 0 & 0 & 0 & 0 & 0 \\
    -q^{\alpha/2} & 0 & 0 & 0 & 0 & q^{\alpha/2} & 0 & 0 & 0 & 0 & 0 & 0 & 0 & 0 & 0 & 0  \\
    0 & 0 & 0 & 0 & 0 & 0 & q^{\frac{\alpha+1}{2}} & 0 & 0 & 0 & 0 & 0 & 0 & 0 & 0 & 0 \\
    0 & 0 & q^{\alpha/2} & 0 & 0 & 0 & 0 & q^{\frac{\alpha+1}{2}} & 0 & 0 & 0 & 0 & 0 & 0 & 0 & 0  \\\hline
    0 & 0 & 0 & 0 & 0 & 0 & 0 & 0 & 0 & 0 & 0 & 0 & 0 & 0 & 0 & 0 \\
    0 & 0 & 0 & 0 & q^{\alpha/2} & 0 & 0 & 0 & 0 & 0 & 0 & 0 & 0 & 0 & 0 & 0  \\
    0 & 0 & 0 & 0 & 0 & 0 & 0 & 0 & 0 & 0 & 0 & 0 & 0 & 0 & 0 & 0 \\
    0 & 0 & 0 & 0 & 0 & 0 & -q^{\alpha/2} & 0 & 0 & 0 & 0 & 0 & 0 & 0 & 0 & 0  \\\hline
    0 & 0 & 0 & 0 & 0 & 0 & 0 & 0 & 0 & 0 & 0 & 0 & q^{\alpha/2} & 0 & 0 & 0 \\
    0 & 0 & 0 & 0 & 0 & 0 & 0 & 0 & q^{\frac{\alpha+1}{2}} & 0 & 0 & 0 & 0 & q^{\alpha/2} & 0 & 0  \\
    0 & 0 & 0 & 0 & 0 & 0 & 0 & 0 & 0 & 0 & 0 & 0 & 0 & 0 & q^{\frac{\alpha+1}{2}} & 0 \\
    0 & 0 & 0 & 0 & 0 & 0 & 0 & 0 & 0 & 0 & -q^{\frac{\alpha+1}{2}} & 0 & 0 & 0 & 0 & q^{\frac{\alpha+1}{2}}  \\\hline
    0 & 0 & 0 & 0 & 0 & 0 & 0 & 0 & 0 & 0 & 0 & 0 & 0 & 0 & 0 & 0 \\
    0 & 0 & 0 & 0 & 0 & 0 & 0 & 0 & 0 & 0 & 0 & 0 & -q^{\frac{\alpha+1}{2}} & 0 & 0 & 0  \\
    0 & 0 & 0 & 0 & 0 & 0 & 0 & 0 & 0 & 0 & 0 & 0 & 0 & 0 & 0 & 0 \\
    0 & 0 & 0 & 0 & 0 & 0 & 0 & 0 & 0 & 0 & 0 & 0 & 0 & 0 & q^{\frac{\alpha+1}{2}} & 0  \\
  \end{array}\right)
$$
\text{ }

\begin{center}
\small
\resizebox{\textwidth}{!}{
$
  \Pi_{\alpha}(E_{2}^{3})=\left(\begin{array}{@{}cccc|cccc|cccc|cccc@{}}
    0 & \begin{tabular}{@{}c@{}} $q^{\alpha/2}$ \\  $\times [\alpha]_{q}$\end{tabular} & 0 & 0 & 0 & 0 & 0 & 0 & 0 & 0 & 0 & 0 & 0 & 0 & 0 & 0 \\
    0 & 0 & 0 & 0 & 0 & 0 & 0 & 0 & 0 & 0 & 0 & 0 & 0 & 0 & 0 & 0  \\
    0 & 0 & 0 & \begin{tabular}{@{}c@{}} $-q^{\alpha/2}$ \\  $\times [\alpha+1]_{q}$\end{tabular} & 0 & 0 & 0 & 0 & 0 & 0 & 0 & 0 & 0 & 0 & 0 & 0 \\
    0 & 0 & 0 & 0 & 0 & 0 & 0 & 0 & 0 & 0 & 0 & 0 & 0 & 0 & 0 & 0  \\\hline
    \begin{tabular}{@{}c@{}} $q^{\alpha/2}$ \\  $\times [\alpha]_{q}$\end{tabular} & 0 & 0 & 0 & 0 & \begin{tabular}{@{}c@{}} $-q^{\alpha/2}$ \\  $\times [\alpha]_{q}$\end{tabular} & 0 & 0 & 0 & 0 & 0 & 0 & 0 & 0 & 0 & 0 \\
    0 & \begin{tabular}{@{}c@{}} $q^{\alpha/2}$ \\  $\times [\alpha]_{q}$\end{tabular} & 0 & 0 & 0 & 0 & 0 & 0 & 0 & 0 & 0 & 0 & 0 & 0 & 0 & 0  \\
    0 & 0 & \begin{tabular}{@{}c@{}} $q^{(\alpha+1)/2}$ \\  $\times [\alpha]_{q}$\end{tabular} & 0 & 0 & 0 & 0 & \begin{tabular}{@{}c@{}} $q^{\alpha/2}$ \\  $\times [\alpha+1]_{q}$\end{tabular} & 0 & 0 & 0 & 0 & 0 & 0 & 0 & 0 \\
    0 & 0 & 0 & \begin{tabular}{@{}c@{}} $q^{(\alpha+1)/2}$ \\  $\times [\alpha]_{q}$\end{tabular} & 0 & 0 & 0 & 0 & 0 & 0 & 0 & 0 & 0 & 0 & 0 & 0  \\\hline
    0 & 0 & 0 & 0 & 0 & 0 & 0 & 0 & 0 & \begin{tabular}{@{}c@{}} $-q^{(\alpha+1)/2}$ \\  $\times [\alpha]_{q}$\end{tabular} & 0 & 0 & 0 & 0 & 0 & 0 \\
    0 & 0 & 0 & 0 & 0 & 0 & 0 & 0 & 0 & 0 & 0 & 0 & 0 & 0 & 0 & 0  \\
    0 & 0 & 0 & 0 & 0 & 0 & 0 & 0 & 0 & 0 & 0 & \begin{tabular}{@{}c@{}} $q^{(\alpha+1)/2}$ \\  $\times [\alpha+1]_{q}$\end{tabular} & 0 & 0 & 0 & 0 \\
    0 & 0 & 0 & 0 & 0 & 0 & 0 & 0 & 0 & 0 & 0 & 0 & 0 & 0 & 0 & 0  \\\hline
    0 & 0 & 0 & 0 & 0 & 0 & 0 & 0 & \begin{tabular}{@{}c@{}} $q^{\alpha/2}$ \\  $\times [\alpha+1]_{q}$\end{tabular} & 0 & 0 & 0 & 0 & \begin{tabular}{@{}c@{}} $q^{(\alpha+1)/2}$ \\  $\times [\alpha]_{q}$\end{tabular} & 0 & 0 \\
    0 & 0 & 0 & 0 & 0 & 0 & 0 & 0 & 0 & \begin{tabular}{@{}c@{}} $q^{\alpha/2}$ \\  $\times [\alpha+1]_{q}$\end{tabular} & 0 & 0 & 0 & 0 & 0 & 0  \\
    0 & 0 & 0 & 0 & 0 & 0 & 0 & 0 & 0 & 0 & \begin{tabular}{@{}c@{}} $q^{(\alpha+1)/2}$ \\  $\times [\alpha+1]_{q}$\end{tabular} & 0 & 0 & 0 & 0 & \begin{tabular}{@{}c@{}} $-q^{(\alpha+1)/2}$ \\  $\times [\alpha+1]_{q}$\end{tabular} \\
    0 & 0 & 0 & 0 & 0 & 0 & 0 & 0 & 0 & 0 & 0 &  \begin{tabular}{@{}c@{}} $q^{(\alpha+1)/2}$ \\  $\times [\alpha+1]_{q}$\end{tabular} & 0 & 0 & 0 & 0  \\
  \end{array}\right);
$}
\end{center}

\begin{center}
\small
\resizebox{\textwidth}{!}{
$
  \Pi_{\alpha}(E_{3}^{1})=\left(\begin{array}{@{}cccc|cccc|cccc|cccc@{}}
    0 & 0 & 0 & 0 & 0 & 0 & 0 & 0 & q^{-1}q^{\alpha/2 } & 0 & 0 & 0 & 0 & 0 & 0 & 0 \\
    0 & 0 & 0 & 0 & 0 & 0 & 0 & 0 & 0 & \begin{tabular}{@{}c@{}} $q^{-1}$ \\  $\times q^{(\alpha+1)/2 }$\end{tabular} & 0 & 0 & 0 & 0 & 0 & 0  \\
    \begin{tabular}{@{}c@{}} $-q^{-1}$ \\  $\times q^{\alpha/2 }$\end{tabular} & 0 & 0 & 0 & 0 & 0 & 0 & 0 & 0 & 0 & q^{-1}q^{\alpha/2 } & 0 & 0 & 0 & 0 & 0 \\
    0 & \begin{tabular}{@{}c@{}} $-q^{-2}$ \\  $\times q^{\alpha/2 }$\end{tabular} & 0 & 0 & 0 & 0 & 0 & 0 & 0 & 0 & 0 & \begin{tabular}{@{}c@{}} $q^{-1}$ \\  $\times q^{(\alpha+1)/2 }$\end{tabular} & 0 & 0 & 0 & 0  \\\hline
    0 & 0 & 0 & 0 & 0 & 0 & 0 & 0 & 0 & 0 & 0 & 0 & -q^{\alpha/2 } & 0 & 0 & 0 \\
    0 & 0 & 0 & 0 & 0 & 0 & 0 & 0 & \begin{tabular}{@{}c@{}} $-q^{\alpha/2 }$ \\  $\times (q-q^{-1})$\end{tabular} & 0 & 0 & 0 & 0 & -q^{(\alpha+1)/2 } & 0 & 0  \\
    0 & 0 & 0 & 0 & \begin{tabular}{@{}c@{}} $q^{-1}$ \\  $\times q^{(\alpha+1)/2 }$\end{tabular} & 0 & 0 & 0 & 0 & 0 & 0 & 0 & 0 & 0 & -q^{\alpha/2 } & 0 \\
    0 & 0 & 0 & 0 & 0 & \begin{tabular}{@{}c@{}} $q^{-2}$ \\  $\times q^{(\alpha+1)/2 }$\end{tabular} & 0 & 0 & 0 & 0 & \begin{tabular}{@{}c@{}} $q^{-1/2}q^{\alpha/2 }$ \\  $\times (q-q^{-1})$\end{tabular} & 0 & 0 & 0 & 0 & -q^{(\alpha+1)/2 }  \\\hline
    0 & 0 & 0 & 0 & 0 & 0 & 0 & 0 & 0 & 0 & 0 & 0 & 0 & 0 & 0 & 0 \\
    0 & 0 & 0 & 0 & 0 & 0 & 0 & 0 & 0 & 0 & 0 & 0 & 0 & 0 & 0 & 0  \\
    0 & 0 & 0 & 0 & 0 & 0 & 0 & 0 & q^{-1}q^{\alpha/2 } & 0 & 0 & 0 & 0 & 0 & 0 & 0 \\
    0 & 0 & 0 & 0 & 0 & 0 & 0 & 0 & 0 & q^{-2}q^{\alpha/2 } & 0 & 0 & 0 & 0 & 0 & 0  \\\hline
    0 & 0 & 0 & 0 & 0 & 0 & 0 & 0 & 0 & 0 & 0 & 0 & 0 & 0 & 0 & 0 \\
    0 & 0 & 0 & 0 & 0 & 0 & 0 & 0 & 0 & 0 & 0 & 0 & 0 & 0 & 0 & 0  \\
    0 & 0 & 0 & 0 & 0 & 0 & 0 & 0 & 0 & 0 & 0 & 0 & \begin{tabular}{@{}c@{}} $-q^{-1}$ \\  $\times q^{(\alpha+1)/2 }$\end{tabular} & 0 & 0 & 0 \\
    0 & 0 & 0 & 0 & 0 & 0 & 0 & 0 & 0 & 0 & 0 & 0 & 0 & \begin{tabular}{@{}c@{}} $-q^{-2}$ \\  $\times q^{(\alpha+1)/2 }$\end{tabular} & 0 & 0  \\
  \end{array}\right);
$}
\end{center}

\begin{center}
\small
\resizebox{\textwidth}{!}{
$
  \Pi_{\alpha}(E_{1}^{3})=\left(\begin{array}{@{}cccc|cccc|cccc|cccc@{}}
    0 & 0 & \begin{tabular}{@{}c@{}} $q q^{\alpha/2}$ \\  $\times [\alpha]_{q}$\end{tabular} & 0 & 0 & 0 & 0 & 0 & 0 & 0 & 0 & 0 & 0 & 0 & 0 & 0 \\
    0 & 0 & 0 & \begin{tabular}{@{}c@{}} $q^2 q^{\alpha/2}$ \\  $\times [\alpha+1]_{q}$\end{tabular} & 0 & 0 & 0 & 0 & 0 & 0 & 0 & 0 & 0 & 0 & 0 & 0  \\
    0 & 0 & 0 & 0 & 0 & 0 & 0 & 0 & 0 & 0 & 0 & 0 & 0 & 0 & 0 & 0 \\
    0 & 0 & 0 & 0 & 0 & 0 & 0 & 0 & 0 & 0 & 0 & 0 & 0 & 0 & 0 & 0  \\\hline
    0 & 0 & 0 & 0 & 0 & 0 & \begin{tabular}{@{}c@{}} $- q q^{(\alpha+1)/2}$ \\  $\times [\alpha]_{q}$\end{tabular} & 0 & 0 & 0 & 0 & 0 & 0 & 0 & 0 & 0 \\
    0 & 0 & \begin{tabular}{@{}c@{}} $(q^{-1}-q)q^{3/2}$ \\  $\times q^{(\alpha+1)/2} [\alpha]_{q}$\end{tabular} & 0 & 0 & 0 & 0 & \begin{tabular}{@{}c@{}} $ -q^2 q^{(\alpha+1)/2}$ \\  $\times [\alpha+1]_{q}$\end{tabular} & 0 & 0 & 0 & 0 & 0 & 0 & 0 & 0  \\
    0 & 0 & 0 & 0 & 0 & 0 & 0 & 0 & 0 & 0 & 0 & 0 & 0 & 0 & 0 & 0 \\
    0 & 0 & 0 & 0 & 0 & 0 & 0 & 0 & 0 & 0 & 0 & 0 & 0 & 0 & 0 & 0  \\\hline
    \begin{tabular}{@{}c@{}} $q q^{\alpha/2}$ \\  $\times [\alpha]_{q}$\end{tabular} & 0 & 0 & 0 & 0 & 0 & 0 & 0 & 0 & 0 & \begin{tabular}{@{}c@{}} $- q q^{\alpha/2}$ \\  $\times [\alpha]_{q}$\end{tabular} & 0 & 0 & 0 & 0 & 0 \\
    0 & \begin{tabular}{@{}c@{}} $q q^{(\alpha+1)/2}$ \\  $\times [\alpha]_{q}$\end{tabular} & 0 & 0 & 0 & 0 & 0 & 0 & 0 & 0 & 0 & \begin{tabular}{@{}c@{}} $ -q^2 q^{\alpha/2}$ \\  $\times [\alpha+1]_{q}$\end{tabular} & 0 & 0 & 0 & 0  \\
    0 & 0 & \begin{tabular}{@{}c@{}} $q q^{\alpha/2}$ \\  $\times [\alpha]_{q}$\end{tabular} & 0 & 0 & 0 & 0 & 0 & 0 & 0 & 0 & 0 & 0 & 0 & 0 & 0 \\
    0 & 0 & 0 & \begin{tabular}{@{}c@{}} $q q^{(\alpha+1)/2}$ \\  $\times [\alpha]_{q}$\end{tabular} & 0 & 0 & 0 & 0 & 0 & 0 & 0 & 0 & 0 & 0 & 0 & 0  \\\hline
    0 & 0 & 0 & 0 & \begin{tabular}{@{}c@{}} $- q^{\alpha/2}$ \\  $\times [\alpha+1]_{q}$\end{tabular} & 0 & 0 & 0 & 0 & 0 & 0 & 0 & 0 & 0 & \begin{tabular}{@{}c@{}} $q q^{(\alpha+1)/2}$ \\  $\times [\alpha]_{q}$\end{tabular} & 0 \\
    0 & 0 & 0 & 0 & 0 & \begin{tabular}{@{}c@{}} $- q^{(\alpha+1)/2}$ \\  $\times [\alpha+1]_{q}$\end{tabular} & 0 & 0 & 0 & 0 & \begin{tabular}{@{}c@{}} $(q^{-1}-q) \times q \times$ \\  $ q^{(\alpha+1)/2} [\alpha+1]_{q}$\end{tabular} & 0 & 0 & 0 & 0 & \begin{tabular}{@{}c@{}} $q^2 q^{(\alpha+1)/2}$ \\  $\times [\alpha+1]_{q}$\end{tabular}  \\
    0 & 0 & 0 & 0 & 0 & 0 & \begin{tabular}{@{}c@{}} $- q^{\alpha/2}$ \\  $\times [\alpha+1]_{q}$\end{tabular} & 0 & 0 & 0 & 0 & 0 & 0 & 0 & 0 & 0 \\
    0 & 0 & 0 & 0 & 0 & 0 & 0 & \begin{tabular}{@{}c@{}} $- q^{(\alpha+1)/2}$ \\  $\times [\alpha+1]_{q}$\end{tabular} & 0 & 0 & 0 & 0 & 0 & 0 & 0 & 0  \\
  \end{array}\right).
$}
\end{center}

\subsection{Change of basis}

In this section we prove that all representations $\Pi_{\alpha}$ are isomorphic to the same representation $\Theta$. Moreover, that representation does not depend on $q^{\alpha}$. To prove this we build a special basis for $V_{\alpha} \otimes V_{\alpha}^{\ast}$ using the action of some generators of $U_{q}\mathfrak{gl}(2 \vert 1)^{\sigma}$ on a specific elementary tensor.

\begin{lem}\label{lem:basis}
Set $v := e_{1}\otimes e_{1}^{\ast} \in V_{\alpha} \otimes V_{\alpha}^{\ast}$. We denote by $\mathcal{B_{\alpha}}$ the family consisting of the following 16 vectors:
    $$v \text{ ; } E_{3}^{2}.v \text{ ; } E_{3}^{1}.v \text{ ; } E_{1}^{3}.v \text{ ; } E_{2}^{3}.v  \text{ ; } E_{3}^{2} E_{3}^{1}.v \text{ ; } E_{3}^{2} E_{1}^{3}.v \text{ ; } E_{3}^{2} E_{2}^{3}.v \text{ ; } E_{3}^{1} E_{1}^{3}.v \text{ ; } E_{3}^{1} E_{2}^{3}.v \text{ ; } E_{1}^{3} E_{2}^{3}.v \text{ ; } E_{3}^{1} E_{1}^{3} E_{2}^{3}.v \text{ ;}$$ $$ E_{3}^{2} E_{1}^{3} E_{2}^{3}.v \text{ ; } E_{3}^{2} E_{3}^{1} E_{2}^{3}.v \text{ ; } E_{3}^{2} E_{3}^{1} E_{1}^{3}.v  \text{ ; } E_{3}^{2} E_{3}^{1} E_{1}^{3} E_{2}^{3}.v.$$\newline
$\mathcal{B}_{\alpha}$ is a basis for vector space $V_{\alpha} \otimes V_{\alpha}^{\ast}$.
\end{lem}

We can number the 16 vectors we are considering: $v_1$, $v_2$, ... , $v_{16}$ in the same order that was used to write them in Lemma~\ref{lem:basis}, and we will do so in the following.

\begin{thm}\label{thm:MAINRT}
For any element $X \in U_{q}\mathfrak{gl}(2 \vert 1)^{\sigma}$, the coefficients of the matrix representing the action of $X$ on $V_{\alpha} \otimes V_{\alpha}^{\ast}$ in basis $\mathcal{B}_{\alpha}$ belong to $\mathbb{C}(q)$ instead of $\mathbb{C}(q,q^{\alpha})$.
\end{thm} 

When $V_{\alpha} \otimes V_{\alpha}^{\ast}$ is decorated with the basis $\mathcal{B}_{\alpha}$, we refer to the underlying vector space as $W$ and the linear representation as $\Theta$. Theorem ~\ref{thm:MAINRT} proves that for any $\alpha$ representations $\Pi_{\alpha}$ and $\Theta$ are isomorphic, and that $\Theta$ does not depend on $\alpha$. \\

Lemma \ref{lem:basis} and Theorem \ref{thm:MAINRT} are proved through explicit computations in Appendix \ref{sec:proofs}. A conceptual proof of these two results follows from the next Remark \ref{rmk:conceptualproof}. Still, we will need the explicit change of basis in Section~\ref{sec:Topology} -- so we need to go forward with completely explicit computations here.

\begin{rmk}\label{rmk:conceptualproof}
Theorem~\ref{thm:MAINRT} can be deduced in a more theoretical and direct manner. Indeed, setting $v^+ = e_1$ the highest weight vector in $V_{\alpha}$, that vector has weight $(0,0,\alpha)$. The highest weight for module $V_{\alpha}^*$ is $(1, 1, -
2-\alpha)$ with $e_4^*$ the highest weight vector, while the lowest weight vector is $e_1^*$ with weight $(0,0,-\alpha)$. Let $v^- = e_1^*$ denote the lowest weight vector in $V_{\alpha}^*$. Then $v^+ \otimes v^-$ cyclically generates the module $V_{\alpha} \otimes V_{\alpha}^*$. That vector has weight $(0, 0, 0)$. Since the weight of this vector is independent from $\alpha$, the module that it cyclically generates is independent from $\alpha$ as well: lowering operations move $v^+$ through a set of basis vectors of $V_{\alpha}$, while raising operators (acting via coproduct) move $v^-$ through all dual basis vectors, so together the whole tensor space is generated. That is, there exists a basis of weight vectors with weights that do not depend on $\alpha$, and the module-action on this basis that does not depend on $\alpha$ either. The basis presented in Lemma \ref{lem:basis} is an explicit realization of this process.
\end{rmk}

For the purpose of clarity, let us choose notations that will be useful for the next steps.

\begin{defn}
   Define $a: W\rightarrow V_{\alpha} \otimes V_{\alpha}^{\ast}$ to be the $U_{q}\mathfrak{gl}(2 \vert 1)^{\sigma}$-module isomorphism given by Theorem~\ref{thm:MAINRT}. The matrix for $a$ relative to bases $\mathcal{B}_{\alpha}$ and $(e_{i}\otimes e_{j}^{\ast})$ - ordered as in Appendix \ref{sec:proofs} - will be referred to as $A$. It has been written down in Appendix \ref{sec:proofs}.
\end{defn}

\subsection{Representation $\Theta^{\ast}$ on $W^{\ast}$}

Diagrammatic considerations in Section~\ref{sec:Topology} will require an \textit{explicit} $U_{q}\mathfrak{gl}(2 \vert 1)^{\sigma}$-module isomorphism between the dual representation of $W$ and representation $V_{\alpha} \otimes V_{\alpha}^{\ast}$. We will compute such an isomorphism hereafter. To do so let us recall some definitions and elementary facts about representation theory in the context of Hopf algebras. These are completely general results, although we write them when the Hopf algebra we are considering is $U_{q}\mathfrak{gl}(2 \vert 1)^{\sigma}$.

\begin{defn}[Action on the dual vector space]
    Set $V$ a $U_{q}\mathfrak{gl}(2 \vert 1)^{\sigma}$-module and let $\pi_{V}$ be the associated representation. As we stated previously, the dual representation $\pi_{V^\ast}$ on $V^\ast$ is given by:
    $$
\forall \phi \in V^{\ast}, \forall x \in U_{q}\mathfrak{gl}(2 \vert 1)^{\sigma} \text{ : } 
\pi_{{V^\ast}}(x)(\phi) = \phi \circ \pi_{{V}}(S^{\sigma}(x)) \in V^\ast
.
$$
\end{defn}

\begin{defn}[Action on the double dual vector space]
    Setting $V^{\ast \ast} = Hom(V^{\ast}, \mathbb{C})$, that vector space is better described in terms of $(e_v)_{v \in V}$ where 
    $$
    e_v(\phi) = \phi(v) \text{ for all } \phi \in V^{\ast}.
    $$
    Indeed, $V^{\ast \ast} = \{ e_v, v \in V \}$. Considering $x \in U_{q}\mathfrak{gl}(2 \vert 1)^{\sigma}$ and $\phi \in V^{\ast}$, the double dual representation $\pi_{V^{\ast \ast}}$ on $V^{\ast \ast}$ can be specified as follows:
    $$ 
    \pi_{V^{\ast \ast}}(x)(e_v)(\phi) = e_v \circ \pi_{V^{\ast}}(S^{\sigma}(x))(\phi) = e_{\pi_{V}((S^{\sigma})^2 (x))(v)}(\phi).
    $$
\end{defn}

The expression for the action of $U_{q}\mathfrak{gl}(2 \vert 1)^{\sigma}$ on the double dual of a representation $V$ will be more easy to describe given that $U_{q}\mathfrak{gl}(2 \vert 1)^{\sigma}$ has a pivotal structure, that is a group-like element $g$ whose action by conjugation is equal to the square of the antipode.

\begin{prop}
    The elements $g = q^{-E_{1}^{1}} q^{E_{2}^{2}} q^{E_{3}^{3}} \sigma$ and $K = q^{2 E_{3}^{3}} q^{2 E_{2}^{2}} \sigma$ of $U_{q}\mathfrak{gl}(2 \vert 1)^{\sigma}$ are pivotal elements for the corresponding Hopf algebra structure that are represented on $V_{\alpha}$ by 
$$
\pi_\alpha(g)=
\begin{pmatrix}
q^{\alpha} & 0 & 0 & 0 \\
0 & -q^{\alpha} & 0 & 0 \\
0 & 0 & -q^{\alpha+2} & 0 \\
0 & 0 & 0 & q^{\alpha+2}
\end{pmatrix}
$$
and by
$$
\pi_\alpha(K) = \begin{pmatrix}
q^{2\alpha} & 0 & 0 & 0 \\
0 & -q^{2\alpha} & 0 & 0 \\
0 & 0 & -q^{2(\alpha+1)} & 0 \\
0 & 0 & 0 & q^{2(\alpha+1)}
\end{pmatrix} = q^{\alpha} \pi_\alpha(g).
$$
\end{prop}

\begin{proof}
Using the commutation relations in $U_{q}\mathfrak{gl}(2 \vert 1)^{\sigma}$ from \cite{DW}, Sections 4.2.3 and 4.2.4, we can prove that 
$$g X g^{-1} = (S^{\sigma})^2 (X) \,, \quad K X K^{-1} = (S^{\sigma})^2 (X)$$  
for $X = E_{1}^{2}, E_{2}^{1}, E_{2}^{3}, E_{3}^{2}$. This is in particular due to the fact that $(S^{\sigma})^2 (x) = S^2 (x)$ if $x$ is even, while $(S^{\sigma})^2 (x) = - S^2 (x)$ if $x$ is odd.
\end{proof}

\begin{rmk}
    In fact, $U_{q}\mathfrak{gl}(n \vert 1)^{\sigma}$ has a pivot $K_n$ for any $n$, that is used to define the link invariant $LG^{n,1}$, for example in \cite{KoPat}. When $n=2$, this pivot is $K_2 = K$ and this choice is coherent with De Wit's choice for $\Omega^-$ that we will use in the following (see \cite{DW}), up to the fact that there appears to be is an inconsistency in the thesis where $q$ is replaced with $q^{-1}$ in section 5. However, there is no need for the purpose of computing the $U_{q}\mathfrak{gl}(2 \vert 1)^{\sigma}$ action on $V^{\ast \ast}$ or an isomorphism between $V$ and $V^{\ast \ast}$ to choose this particular pivotal element $K$, and we will use $g$ instead for the moment.
\end{rmk}

The following results are more general facts, up to the choice of a pivotal element. However, once again, we write them in the context of Hopf algebra $U_{q}\mathfrak{gl}(2 \vert 1)^{\sigma}$. \\

    For any element $x$ in $U_{q}\mathfrak{gl}(2 \vert 1)^{\sigma}$ and any $U_{q}\mathfrak{gl}(2 \vert 1)^{\sigma}$-module $V$:  $$\pi_{V^{\ast \ast}}(x)(e_v) = e_{\pi_{V}(g x g^{-1})(v)}.$$

Moreover, the pivotal structure $g$ induces a $U_{q}\mathfrak{gl}(2 \vert 1)^{\sigma}$-module isomorphism

$$  \Psi_{g}   : \left\{
\begin{array}{lll}
 V & \rightarrow  & V^{\ast \ast}\\
 v & \mapsto & e_{\pi_{V}(g)(v)}
\end{array}
\right.
.$$

In the same order of ideas, a canonical identification exists between \textit{vector spaces} $V_2^{\ast} \otimes V_1^{\ast}$ and $(V_1 \otimes V_2)^{\ast}$. It is given by:

$$  \theta   : \left\{
\begin{array}{lll}
 V_2^{\ast} \otimes V_1^{\ast} & \rightarrow  & (V_1 \otimes V_2)^{\ast}\\
 \phi_2 \otimes \phi_ 1& \mapsto & \theta(\phi_2 \otimes \phi_1)
\end{array}
\right.
$$
with $\theta(\phi_2 \otimes \phi_ 1)(x_1 \otimes x_ 2) = \phi_ 1(x_1) \phi_ 2(x_2)$.
\begin{prop}
    The linear map $\theta$ is $U_{q}\mathfrak{gl}(2 \vert 1)^{\sigma}$-linear for the induced structures if $V_1$ and $V_2$ are $U_{q}\mathfrak{gl}(2 \vert 1)^{\sigma}$-modules.
\end{prop}

The following identification will be useful as well.

\begin{prop}[Isomorphism induced between the dual representations of two isomorphic representations]
    If $f: V_{1}\rightarrow V_{2}$ is a $U_{q}\mathfrak{gl}(2 \vert 1)^{\sigma}$-module isomorphism, then

$$  f^{\#}   : \left\{
\begin{array}{lll}
 V_1^{\ast} & \rightarrow  & V_2^{\ast}\\
 \phi & \mapsto & \phi \circ f^{-1}
\end{array}
\right.
$$
is also a $U_{q}\mathfrak{gl}(2 \vert 1)^{\sigma}$-module isomorphism. It is the isomorphism induced by $f$ and $(f^{\#})^{-1} = (f^{-1})^{\#}$. Moreover, if $f$ is represented by matrix $A$ with respect to bases $\mathcal{B}$ and $\mathcal{C}$, then $(A^{-1})^{T}=(A^{T})^{-1}$ is the matrix of $f^{\#}$ with respect to $\mathcal{B}^{\ast}$ and $\mathcal{C}^{\ast}$.
\end{prop}

Now we have the tools to show that representation $W$ 
from Theorem~\ref{thm:MAINRT} and its dual representation $W^{\ast}$ are isomorphic $U_{q}\mathfrak{gl}(2 \vert 1)^{\sigma}$-modules.
\begin{thm}
    $W$ and $W^{\ast}$ are isomorphic $U_{q}\mathfrak{gl}(2 \vert 1)^{\sigma}$-modules.
\end{thm}

\begin{proof}
    Recall that $a: W\rightarrow V_{\alpha} \otimes V_{\alpha}^{\ast}$ is the $U_{q}\mathfrak{gl}(2 \vert 1)^{\sigma}$-module isomorphism given by Theorem~\ref{thm:MAINRT} and set
    $$
    \rho   : \left\{
\begin{array}{lll}
 V_{\alpha} \otimes V_{\alpha}^{\ast} & \rightarrow  & V_{\alpha}^{\ast \ast} \otimes V_{\alpha}^{\ast}\\
 v \otimes \phi & \mapsto & e_{\pi_{V}(g)(v)} \otimes \phi
\end{array}
\right.
.
    $$ Denoting by $\theta$ the $U_{q}\mathfrak{gl}(2 \vert 1)^{\sigma}$-module isomorphism sending $V_{\alpha}^{\ast \ast} \otimes V_{\alpha}^{\ast}$ to $(V_{\alpha} \otimes V_{\alpha}^{\ast})^{\ast}$, it is now clear that map $\mathcal{F}: W\rightarrow W^{\ast}$ given by
    $$\mathcal{F} := (a^{-1})^{\#} \circ \theta \circ \rho \circ a $$
    is a $U_{q}\mathfrak{gl}(2 \vert 1)^{\sigma}$-module isomorphism between $W$ and $W^{\ast}$.
\end{proof}

\begin{defn}
Let $\tilde{a}$ be the isomorphism sending $W^{\ast}$ to $V_{\alpha} \otimes V_{\alpha}^{\ast}$ and defined by
$$\tilde{a}=a \circ \mathcal{F}^{-1}.$$
Let us also refer to the matrix for $\tilde{a}$ relative to bases $\mathcal{B}_{\alpha}^{\ast}$ and $(e_i \otimes e_j^{\ast})$ as $\tilde{A}$. Then, if $F$ is the matrix for map $\mathcal{F}$ in the right bases, the following identity holds: $\tilde{A} = A F^{-1}$.
\end{defn}

Putting some patience and effort into computations, one can produce matrix $\tilde{A}$. In bases $\mathcal{B}_{\alpha}^{\ast}$ and $(e_i \otimes e_j^{\ast})$ with the column labeled $i$ referring to $v_i^{\ast}$ and the row identified as $(j,k)$ corresponding to $e_{j}\otimes e_{k}^{\ast}$, $\tilde{A}$ is given by:

\newpage

\begin{landscape}
\thispagestyle{empty}
\centering
\tiny{

  \begin{blockarray}{@{}ccccccccccccccccccc@{}}
   & \matindex{} & \matindex{} & \matindex{} & \matindex{} &\matindex{} & \matindex{} &\matindex{} & \matindex{} &\matindex{} & \matindex{} &\matindex{} & \matindex{} &\matindex{} & \matindex{} &\matindex{} & \matindex{}& \\
    & \matindex{} & \matindex{} & \matindex{} & \matindex{} &\matindex{} & \matindex{} &\matindex{} & \matindex{} &\matindex{} & \matindex{} &\matindex{} & \matindex{} &\matindex{} & \matindex{} &\matindex{} & \matindex{}& \\
     & \matindex{} & \matindex{} & \matindex{} & \matindex{} &\matindex{} & \matindex{} &\matindex{} & \matindex{} &\matindex{} & \matindex{} &\matindex{} & \matindex{} &\matindex{} & \matindex{} &\matindex{} & \matindex{}& \\
   & \matindex{} & \matindex{} & \matindex{} & \matindex{} &\matindex{} & \matindex{} &\matindex{} & \matindex{} &\matindex{} & \matindex{} &\matindex{} & \matindex{} &\matindex{} & \matindex{} &\matindex{} & \matindex{}& \\
    & \matindex{} & \matindex{} & \matindex{} & \matindex{} &\matindex{} & \matindex{} &\matindex{} & \matindex{} &\matindex{} & \matindex{} &\matindex{} & \matindex{} &\matindex{} & \matindex{} &\matindex{} & \matindex{}& \\
     & \matindex{} & \matindex{} & \matindex{} & \matindex{} &\matindex{} & \matindex{} &\matindex{} & \matindex{} &\matindex{} & \matindex{} &\matindex{} & \matindex{} &\matindex{} & \matindex{} &\matindex{} & \matindex{}& \\
  & \matindex{} & \matindex{} & \matindex{} & \matindex{} &\matindex{} & \matindex{} &\matindex{} & \matindex{} &\matindex{} & \matindex{} &\matindex{} & \matindex{} &\matindex{} & \matindex{} &\matindex{} & \matindex{}& \\
   & \matindex{1} & \matindex{8} & \matindex{9} & \matindex{16} &\matindex{2} & \matindex{5} &\matindex{15} & \matindex{12} &\matindex{3} & \matindex{4} &\matindex{14} & \matindex{13} &\matindex{6} & \matindex{11} &\matindex{10} & \matindex{7}& \\
    \begin{block}{(cccccccccccccccccc)c}
     & $q^{-\alpha}$ & 0 & 0 & 0 & 0 & 0 & 0 & 0 & 0 & 0 & 0 & 0 & 0 & 0 & 0 & 0 & &\matindex{(1,1)} \\
    & $q^{-\alpha}$ & \begin{tabular}{@{}c@{}} $-q^{-2\alpha }$ \\  $\times [\alpha]_{q}^{-1} $\end{tabular} & 0 & 0 & 0 & 0 & 0 & 0 & 0 & 0 & 0 & 0 & 0 & 0 & 0 & 0 & &\matindex{(2,2)} \\
     & $q^{-\alpha}$ &  \begin{tabular}{@{}c@{}} $(q^{-2}-1)q^{-2\alpha }$ \\  $\times [\alpha]_{q}^{-1} $\end{tabular} & \begin{tabular}{@{}c@{}} $-q^{-2}q^{-2\alpha }$ \\  $\times [\alpha]_{q}^{-1} $\end{tabular} & 0 & 0 & 0 & 0 & 0 & 0 & 0 & 0 & 0 & 0 & 0 & 0 & 0 & &\matindex{(3,3)} \\
     & $q^{-\alpha}$ & \begin{tabular}{@{}c@{}} $-q^{-2\alpha }$ \\  $\times [\alpha]_{q}^{-1} $\end{tabular} & \begin{tabular}{@{}c@{}} $-q^{-2}q^{-2\alpha }$ \\  $\times [\alpha]_{q}^{-1} $\end{tabular} & \begin{tabular}{@{}c@{}} $q^{-2}q^{-3\alpha }[\alpha]_{q}^{-1}$ \\  $\times [\alpha+1]_{q}^{-1} $\end{tabular} & 0 & 0 & 0 & 0 & 0 & 0 & 0 & 0 & 0 & 0 & 0 & 0 & &\matindex{(4,4)} \\
     & 0 & 0 & 0 & 0 & 0 & $q^{-3\alpha/2}[\alpha]_{q}^{-1}$ & 0 & 0 & 0 & 0 & 0 & 0 & 0 & 0 & 0 & 0 & &\matindex{(1,2)} \\
     & 0 & 0 & 0 & 0 & $q^{-3\alpha/2}$ & 0 & 0 & 0 & 0 & 0 & 0 & 0 & 0 & 0 & 0 & 0 & &\matindex{(2,1)} \\
     & 0 & 0 & 0 & 0 & 0 & \begin{tabular}{@{}c@{}} $q^{-3/2}q^{-3\alpha/2 }$ \\  $\times [\alpha]_{q}^{-1} $\end{tabular} & 0 & \begin{tabular}{@{}c@{}} $-q^{-3/2}q^{-5\alpha/2} $ \\  $\times  [\alpha]_{q}^{-1}[\alpha+1]_{q}^{-1} $\end{tabular} & 0 & 0 & 0 & 0 & 0 & 0 & 0 & 0 & &\matindex{(3,4)} \\
     & 0 & 0 & 0 & 0 & $q^{-1/2}q^{-3\alpha/2}$ & 0 & \begin{tabular}{@{}c@{}} $-q^{-5/2 } q^{-5\alpha/2 } $ \\  $\times [\alpha]_{q}^{-1} $\end{tabular} & 0 & 0 & 0 & 0 & 0 & 0 & 0 & 0 & 0 & &\matindex{(4,3)} \\
     & 0 & 0 & 0 & 0 & 0 & 0 & 0 & 0 & 0 & \begin{tabular}{@{}c@{}} $q^{-1}q^{-3\alpha/2}$ \\  $\times [\alpha]_{q}^{-1}$\end{tabular} & 0 & 0 & 0 & 0 & 0 & 0 & &\matindex{(1,3)} \\
     & 0 & 0 & 0 & 0 & 0 & 0 & 0 & 0 & $q^{-1}q^{-3\alpha/2}$ & 0 & 0 & 0 & 0 & 0 & 0 & 0 & &\matindex{(3,1)} \\
     & 0 & 0 & 0 & 0 & 0 & 0 & 0 & 0 & 0 & \begin{tabular}{@{}c@{}} $-q^{-3/2}q^{-3\alpha/2}$ \\  $\times [\alpha]_{q}^{-1}$\end{tabular} & 0 & \begin{tabular}{@{}c@{}} $-q^{-1/2 } q^{-5\alpha/2 }$ \\  $\times [\alpha]_{q}^{-1}[\alpha+1]_{q}^{-1} $\end{tabular} & 0 & 0 & 0 & 0 & &\matindex{(2,4)} \\
     & 0 & 0 & 0 & 0 & 0 & 0 & 0 & 0 & $-q^{-1/2}q^{-3\alpha/2}$ & 0 & \begin{tabular}{@{}c@{}} $-q^{-3/2 } q^{-5\alpha/2 } $ \\  $\times [\alpha]_{q}^{-1} $\end{tabular} & 0 & 0 & 0 & 0 & 0 & &\matindex{(4,2)} \\
     & 0 & 0 & 0 & 0 & 0 & 0 & 0 & 0 & 0 & 0 & 0 & 0 & 0 & \begin{tabular}{@{}c@{}} $-q^{-2\alpha} $ \\  $\times [\alpha]_{q}^{-1}[\alpha+1]_{q}^{-1} $\end{tabular} & 0 & 0 & &\matindex{(1,4)} \\
     & 0 & 0 & 0 & 0 & 0 & 0 & 0 & 0 & 0 & 0 & 0 & 0 & $-q^{-1}q^{-2\alpha}$ & 0 & 0 & 0 & &\matindex{(4,1)} \\
     & 0 & 0 & 0 & 0 & 0 & 0 & 0 & 0 & 0 & 0 & 0 & 0 & 0 & 0 & 0 & \begin{tabular}{@{}c@{}} $-q^{-3/2} q^{-2\alpha}$ \\  $\times [\alpha+1]_{q}^{-1} $\end{tabular} & &\matindex{(2,3)} \\
     & 0 & 0 & 0 & 0 & 0 & 0 & 0 & 0 & 0 & 0 & 0 & 0 & 0 & 0 & \begin{tabular}{@{}c@{}} $-q^{3/2} q^{-2\alpha}$ \\  $\times [\alpha+1]_{q}^{-1} $\end{tabular} & 0 & &\matindex{(3,2)} \\
    \end{block}
  \end{blockarray}
}
\end{landscape}

\section{Proof of the genus bound}\label{sec:Topology}

Here we will explain how a special type of (1-1)-tangles will allow us to compute the $LG$ invariant through the Reshetikhin-Turaev functor in a way that lets us control the degree of the polynomial in terms of the genus of the knot. But first let us define the Links-Gould invariant of links $LG^{2,1} = LG$.

\subsection{Definition of the Link-Gould invariant of links}

The Links-Gould invariant of links $LG$ was first defined by Jon Links and Mark Gould in \cite{LG}. We will not exactly follow their approach here to define and compute the invariant. We will stick more closely to the approach taken by De Wit, Kauffman and Links in \cite{DWKL} where they obtain $LG$ from an operator invariant built using the Reshetikhin-Turaev construction applied to the Hopf superalgebra $U_{q}\mathfrak{gl}(2 \vert 1)$. In that setting, $LG$ is part of a larger family of invariants $LG^{m,n}$ where $LG = LG^{2,1}$, see \cite{DW2}. In the De Wit-Kauffman-Links approach, $LG$ is obtained from an $R$-matrix and cups, caps that define an operator invariant of oriented tangles. Indeed, Hopf algebra $U_{q}\mathfrak{gl}(2 \vert 1)^{\sigma}$ can be decorated with a ribbon structure (see for example \cite{KT}). Then for $L$ a link in $S^3$, $LG(L,q,q^{\alpha})$ is obtained from a (1-1)-tangle $T$ with closure $L$ and the operator invariant
$$
Q^{U_{q}\mathfrak{gl}(2 \vert 1)^{\sigma} , V_{\alpha}}(T)
$$
derived from ribbon Hopf algebra $U_{q}\mathfrak{gl}(2 \vert 1)^{\sigma}$ via the irreducible representation $V_{\alpha}$, in the following way: $Q^{U_{q}\mathfrak{gl}(2 \vert 1)^{\sigma} , V_{\alpha}}(T)$ is a $U_{q}\mathfrak{gl}(2 \vert 1)^{\sigma}$-linear map from the irreducible $U_{q}\mathfrak{gl}(2 \vert 1)^{\sigma}$- module $V_{\alpha}$ to $V_{\alpha}$. Using the fact that $V_{\alpha}$ is irreducible, Schur's Lemma states that such a map acts as a scalar on $V_{\alpha}$, see e.g. \cite{Bo}.

\begin{defn}[The Links-Gould invariant of links $LG$]
    Set $L$ an oriented link in $S^3$ and $T$ an oriented (1-1)-tangle such that the closure of $T$ is $L$. Then the operator invariant derived from the ribbon Hopf algebra $U_{q}\mathfrak{gl}(2 \vert 1)^{\sigma}$ and irreducible representation $V_{\alpha}$ is such that
    $$
Q^{U_{q}\mathfrak{gl}(2 \vert 1)^{\sigma} , V_{\alpha}}(T) = c.id_{V_{\alpha}}
$$
for some $c \in \mathbb{C}$. Moreover, the value of $c$ does not depend on the choice of $T$, but only on the isotopy class of $L$, because $V_{\alpha}$ is an ambidextrous object in the sense of \cite{GPMT}. Constant c therefore is a link invariant and it is by definition the value of the Links-Gould polynomial $LG(L,q,q^{\alpha})$ for a given choice of parameters $q,q^{\alpha}$.
\end{defn}

To compute $LG$ in the following, we will continue using De Wit's thesis manuscript \cite{DW} as a source of coherent content, as we have done in the previous section. De Wit's approach to defining $LG$ is parallel to that of De Wit-Kauffman-Links in \cite{DWKL} since $LG$ is presented as derived from an operator invariant in both cases, although the notations are not always completely similar in the two presentations.

\subsection{Representing a knot as a particular (1-1)-tangle}

Recall that any \textit{oriented} surface $\Sigma$ with boundary is homeomorphic to a disk with strips attached, see for example Figure~\ref{fig:FIG1}. Moreover, two such surfaces are homeomorphic if and only if they have the same number of strips and the same number of boundary components. In that case the genus of the the surface is given by:
$$
g(\Sigma) = \frac{2 - \# \text{ disks} + \# \text{ strips} - \# \text{ boundary components}}{2}
.$$
Such surfaces play an important role in classical knot theory, for example when studying the Alexander polynomial. 
\begin{defn}
A Seifert surface for a knot $K$ is an oriented surface $\Sigma$ embedded in $S^3$ such that $\partial \Sigma = K$.
\end{defn}

Such a Seifert surface exists for any knot thanks to the Seifert algorithm. Knowing that, the genus of a knot is the minimum among all genera for Seifert surfaces of the knot. Note that for a knot $K$ with genus $g$, one can consider a Seifert surface for $K$ that is homeomorphic to the surface represented in Figure~\ref{fig:FIG1}, with $2g$ strips attached to the one disk. But one can be more precise than that. 
\newline

\begin{figure}
\adjustbox{trim=0cm 0.7cm 0cm 0.5cm}{
  \includegraphics[scale=.4]{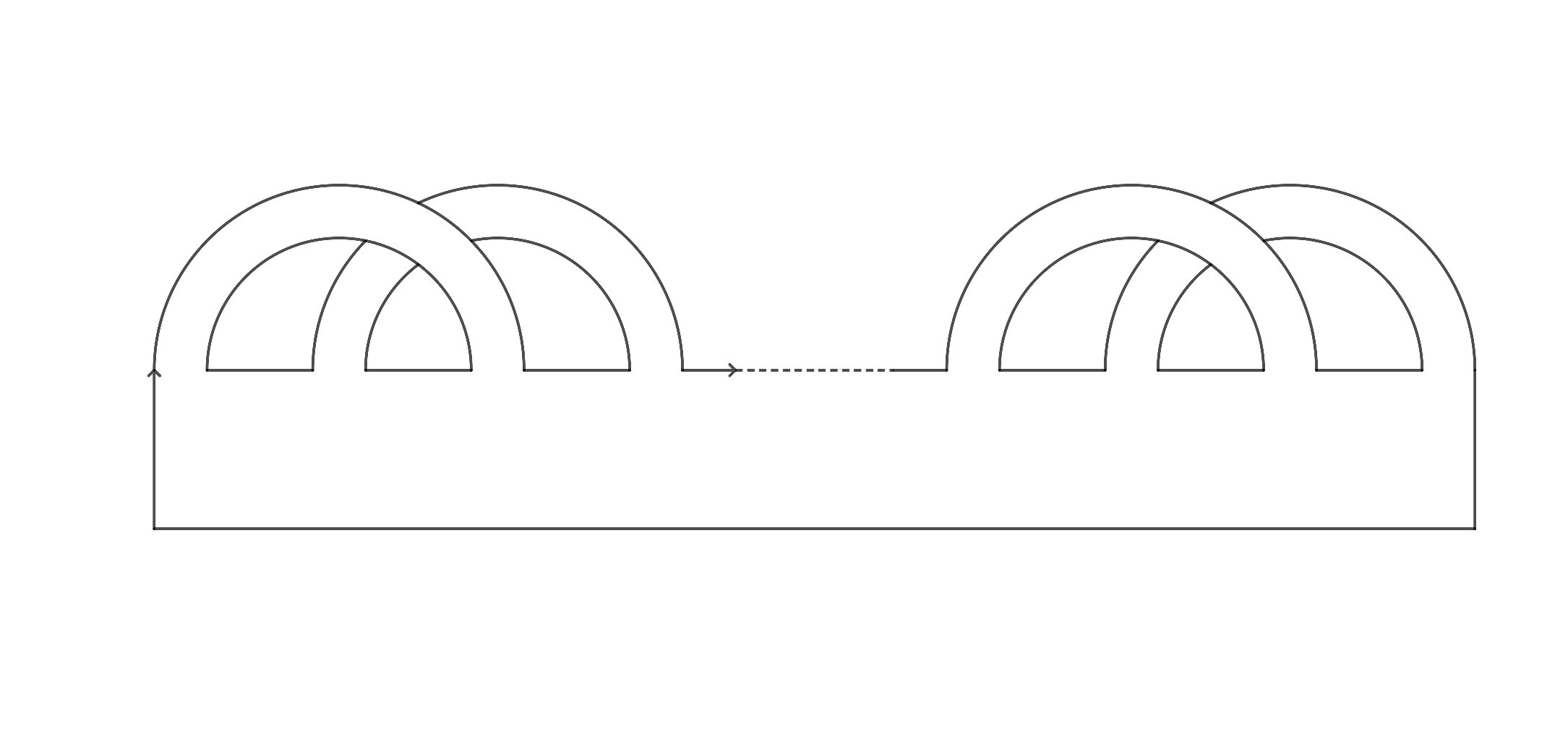}
  }
\caption{A Seifert surface for a knot of genus $g$.} \label{fig:FIG1}
\end{figure}

Indeed, following ideas developed and used in \cite{Hab} and \cite{LNV}, if you consider a knot $K$ with genus $g$ and $\Sigma$ a Seifert surface for $K$ with minimal genus, then up to isotopy, $\Sigma$ can be obtained as the thickening of a tangle $B$ with $2g$ components and $4g$ endpoints on a same horizontal line. More precisely, $B$ is a (4g-0)-tangle, and if each component $i$ of the tangle is an oriented curve $A_i \rightarrow B_i$, the points will be in the following order if you travel along the line from left to right:
$$
A_1, A_2, B_1, B_2, A_3, A_4, B_3, B_4, \ldots ,A_{2g-1}, A_{2g}, B_{2g-1}, B_{2g}.
$$
See Figure~\ref{fig:FIG2} for a drawing.\newline

\begin{figure}
\adjustbox{trim=0cm 0.7cm 0cm 0.2cm}{
  \includegraphics[scale=.4]{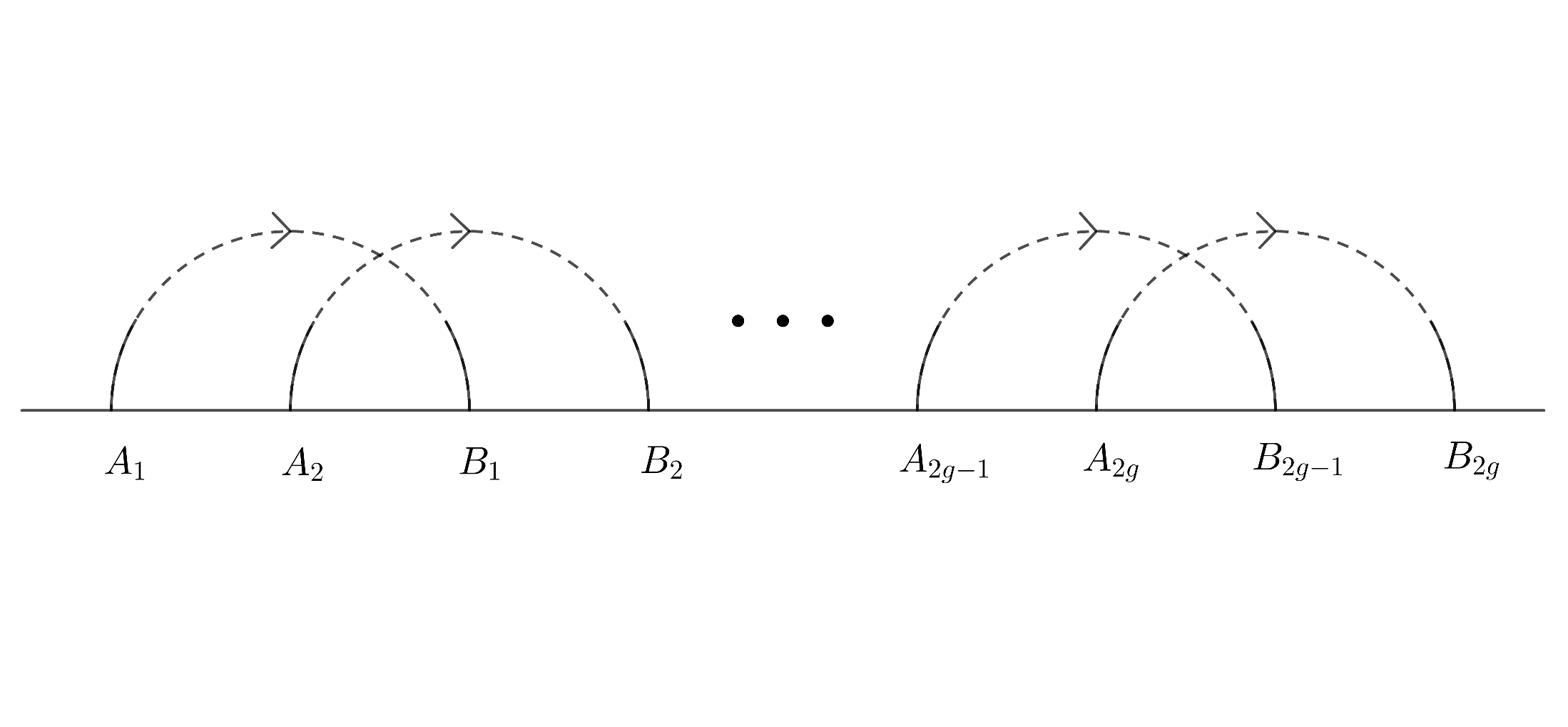}
  }
\caption{(4g-0)-tangle $B$.}\label{fig:FIG2} 
\end{figure}

The thickening is obtained by doubling all the components of $B$ while reversing the orientation of the right strand of each component before attaching the $8g$ endpoints to a circle using cups. Surface $\Sigma$ obtained through that process is represented in Figure~\ref{fig:FIG3} with $K = \partial \Sigma$. Surface $\Sigma$ has $2g$ strips attached to one disk, has one boundary component $K$, and therefore has genus $g$.\newline

\begin{figure}
\adjustbox{trim=0cm 0.7cm 0cm 0.5cm}{
  \includegraphics[scale=.4]{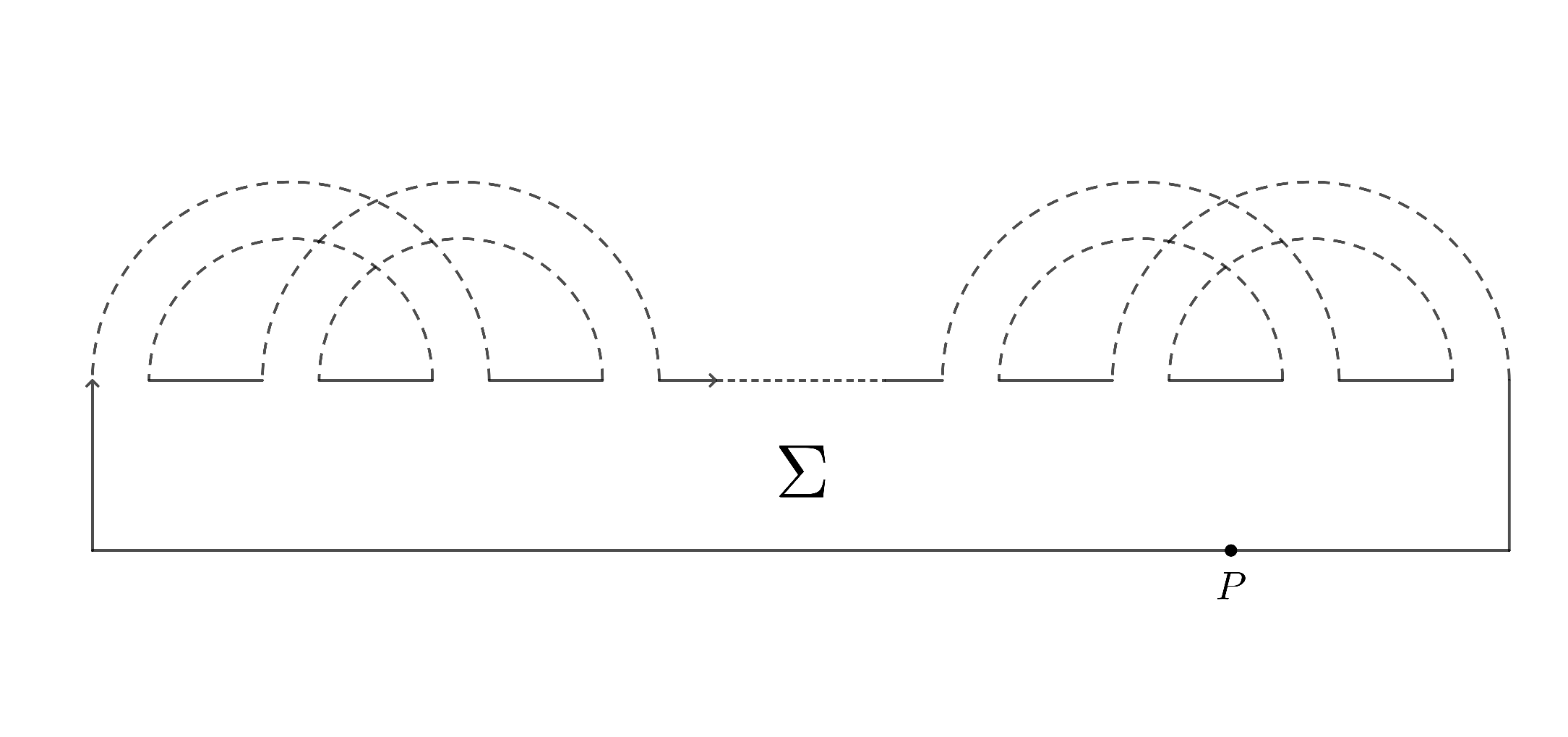}
  }
\caption{Embedded Seifert surface $\Sigma$ of genus $g$ for a knot $K$.}\label{fig:FIG3}
\end{figure}

When you open the disk at some point $P$ like in Figure~\ref{fig:FIG4}, the knot gives rise to $T$ an oriented (1-1)-tangle for which the closure operation recovers $K$.\newline

\begin{figure}
\adjustbox{trim=4cm 1cm 0cm 0cm}{
  \includegraphics[scale=.7]{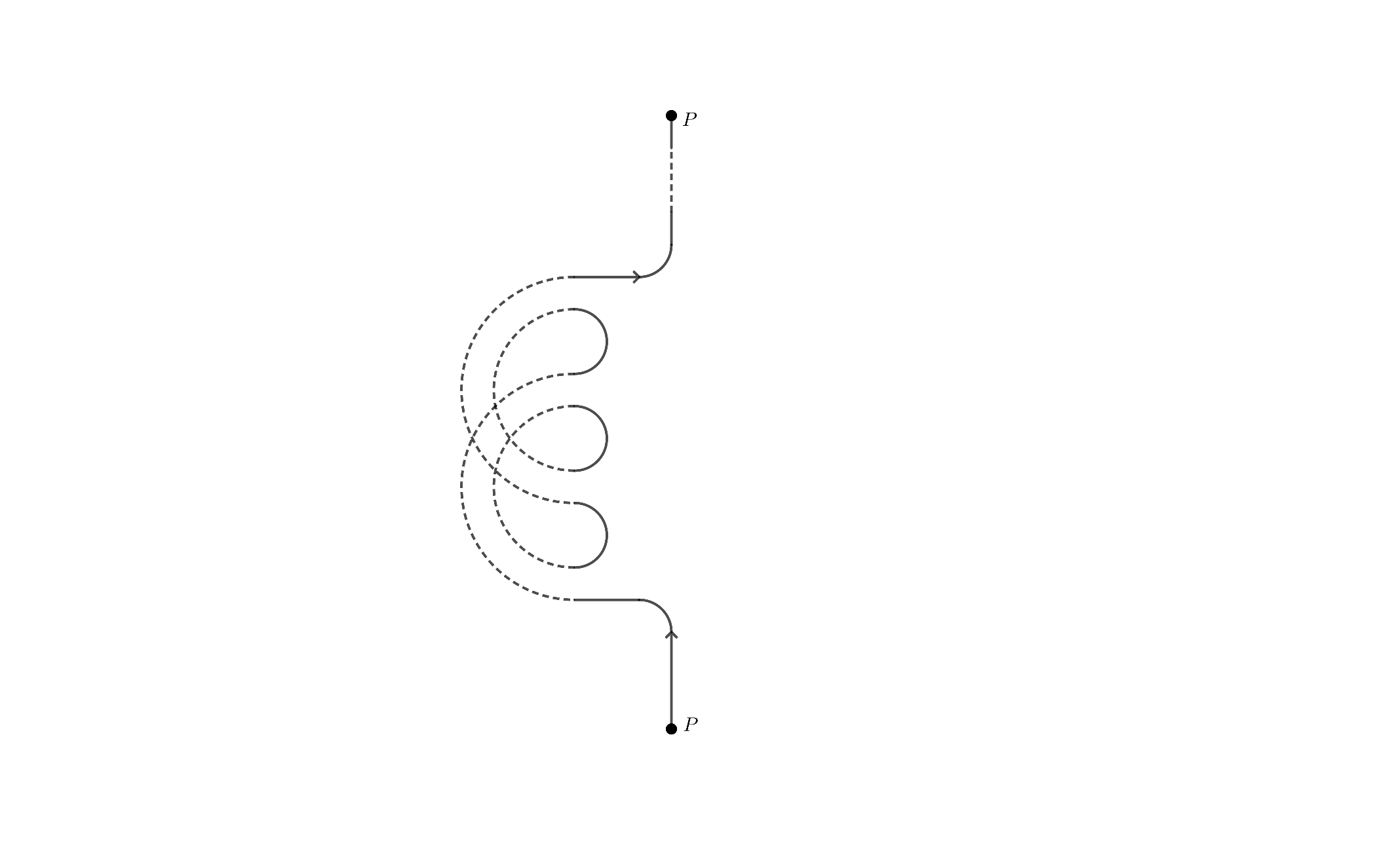}
}
\caption{Tangle $T$.}\label{fig:FIG4} 
\end{figure}

The topological invariance of the operator invariant associated to $LG$ allows us to compute $LG(K,q,q^{\alpha})$ on the (1-1)-tangle that is isotopic to $T$ shown in Figure~\ref{fig:FIG5}:
$$
Q^{U_{q}\mathfrak{gl}(2 \vert 1)^{\sigma} , V_{\alpha}}(T) = LG(K,q,q^{\alpha})~id_{V_{\alpha}}.
$$

\begin{figure}
\adjustbox{trim=4cm 0cm 0cm 0cm}{
  \includegraphics[scale=.7]{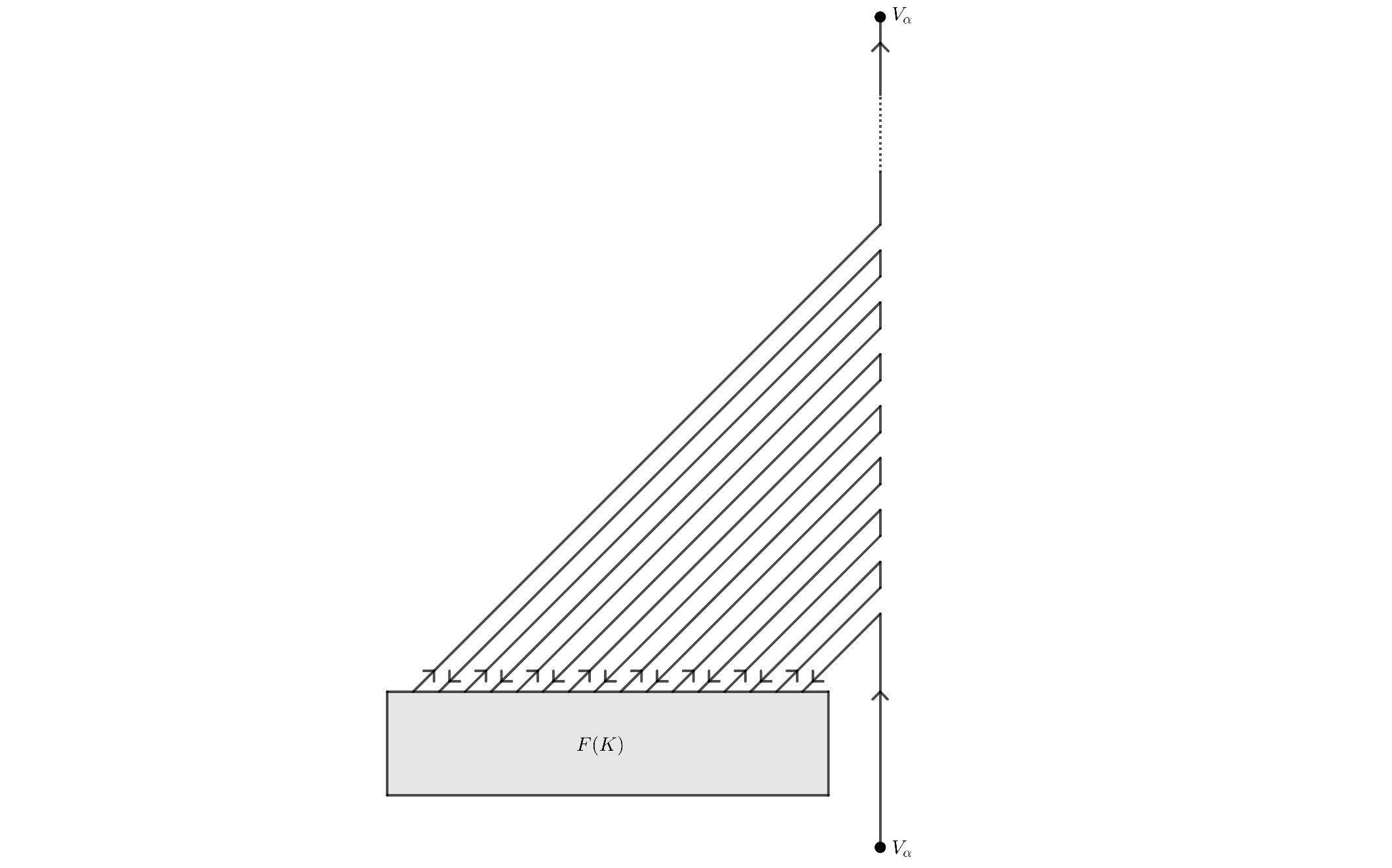}
}
\caption{Tangle $T$ after isotopy.}\label{fig:FIG5}
\end{figure}

\begin{defn}
 In Figure~\ref{fig:FIG5}, tangle $F(K)$ is the doubling of (4g-0)-tangle $B$ after a half-turn. The associated operator invariant is: $$Q^{U_{q}\mathfrak{gl}(2 \vert 1)^{\sigma} , V_{\alpha}}(F(K)):\mathbb{C}(q,q^{\alpha})\rightarrow (V_{\alpha} \otimes V_{\alpha}^{\ast})^{\otimes 4g}.$$
\end{defn}

\begin{lem}\label{lem:naturality}
If $G(K)$ is the oriented $(0-4g)$-tangle obtained from $(4g-0)$-tangle $B$ after a half-turn, then:
$$
Q^{U_{q}\mathfrak{gl}(2 \vert 1)^{\sigma} , V_{\alpha}}(F(K)) = Q^{U_{q}\mathfrak{gl}(2 \vert 1)^{\sigma} , W}(G(K)).
$$
\end{lem}

\begin{proof}
This is due to the naturality properties of the Reshetikhin-Turaev functor.
\end{proof}

\begin{rmk}
    In the following, we will purposefully conflate tangle $G(K)$ and operator invariant/tensor $Q^{U_{q}\mathfrak{gl}(2 \vert 1)^{\sigma} , W}(G(K))$ in terms of notation. This seemingly generates a conflict in notations, but it makes computations more understandable in a context where there are many notations already. In particular, when $G(K)$ is used in the following as an operator, it stands for $Q^{U_{q}\mathfrak{gl}(2 \vert 1)^{\sigma} , W}(G(K))$.
\end{rmk}

Now define $\Omega^-:V_{\alpha} \otimes V_{\alpha}^{\ast}\rightarrow\mathbb{C}(q,q^{\alpha})$ the negative cap derived from ribbon Hopf algebra $U_{q}\mathfrak{gl}(2 \vert 1)^{\sigma}$ using representation $V_{\alpha}$. It can be written in the form of a pairing between bases $(e_1, e_2, e_3, e_4)$ and $(e_1^*, e_2^*, e_3^*, e_4^*)$:
$$
\Omega^- = \begin{pmatrix}
q^{2\alpha} & 0 & 0 & 0 \\
0 & -q^{2\alpha} & 0 & 0 \\
0 & 0 & -q^{2(\alpha+1)} & 0 \\
0 & 0 & 0 & q^{2(\alpha+1)}
\end{pmatrix}.
$$
Also recall that $a: W\rightarrow V_{\alpha} \otimes V_{\alpha}^{\ast}$ and $\tilde{a}: W^*\rightarrow V_{\alpha} \otimes V_{\alpha}^{\ast}$ are the isomorphisms we defined and computed in the previous section. 

Using this we construct two morphisms $b: W \otimes V_{\alpha}\rightarrow V_{\alpha}$ and $c: W^* \otimes V_{\alpha}\rightarrow V_{\alpha}$ that are expressed diagrammatically in Figure~\ref{fig:FIG6}, and can be written explicitly:
$$
b = (id_{V_{\alpha}} \otimes \Omega^-) \circ (a \otimes id_{V_{\alpha}}) \, ,
$$
$$
c = (id_{V_{\alpha}} \otimes \Omega^-) \circ (\tilde{a} \otimes id_{V_{\alpha}}) \, .
$$
\begin{figure}
\adjustbox{trim=1.2cm 1cm 0cm 1.3cm}{
  \includegraphics[scale=.7]{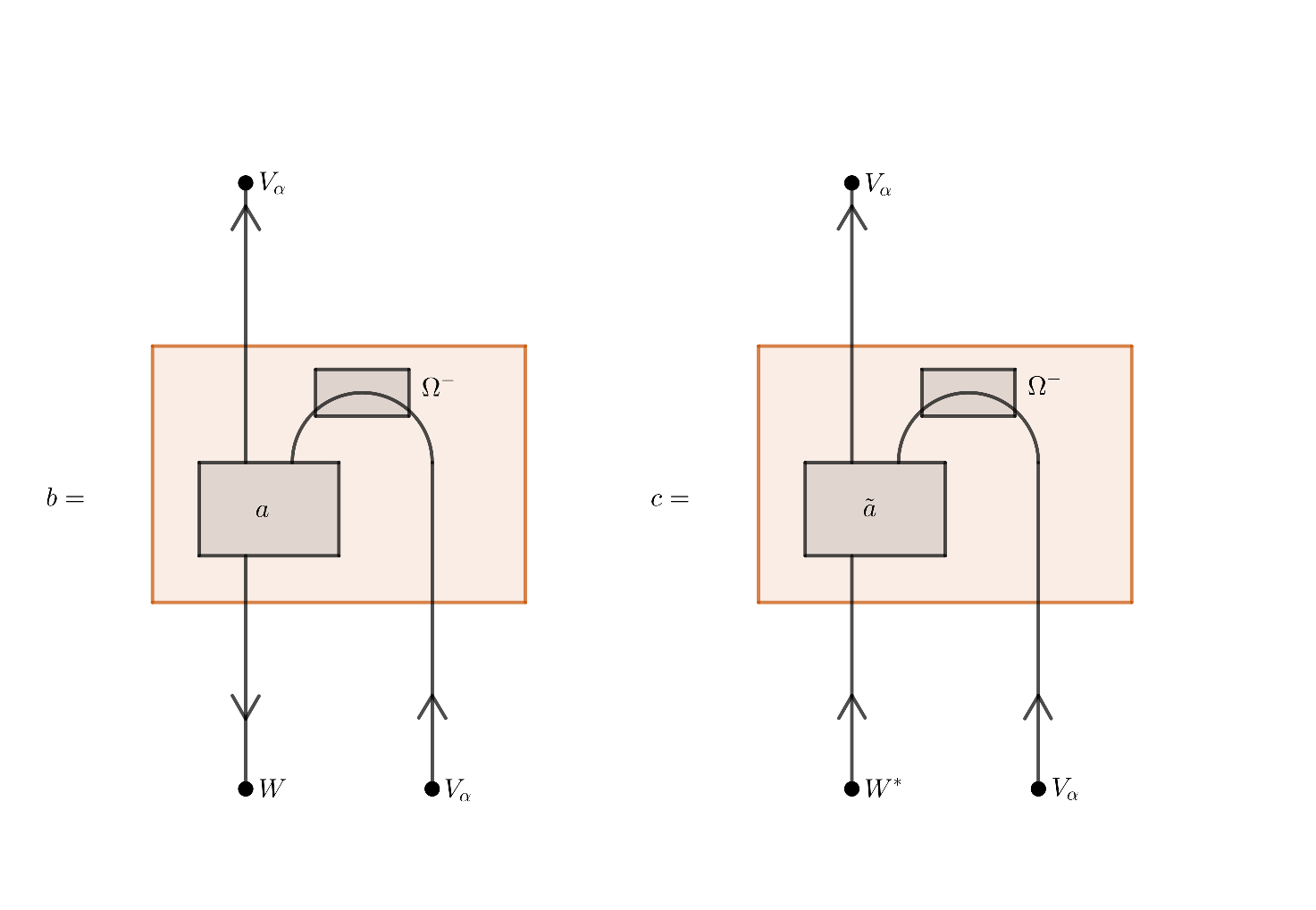}
}
\caption{Maps $b$ and $c$.}\label{fig:FIG6}
\end{figure}
This leads to the (1-1)-tangle diagram presented in Figure~\ref{fig:FIG7} that will be the one we will use to compute $LG(K,q,q^{\alpha})$ as an operator invariant, with $G(K): \mathbb{C}(q,q^{\alpha}) \rightarrow (W^* \otimes W^* \otimes W\otimes W)^{\otimes g}$.\newline

\begin{figure}
\adjustbox{trim=4cm 0cm 0cm 0cm}{
  \includegraphics[scale=.5]{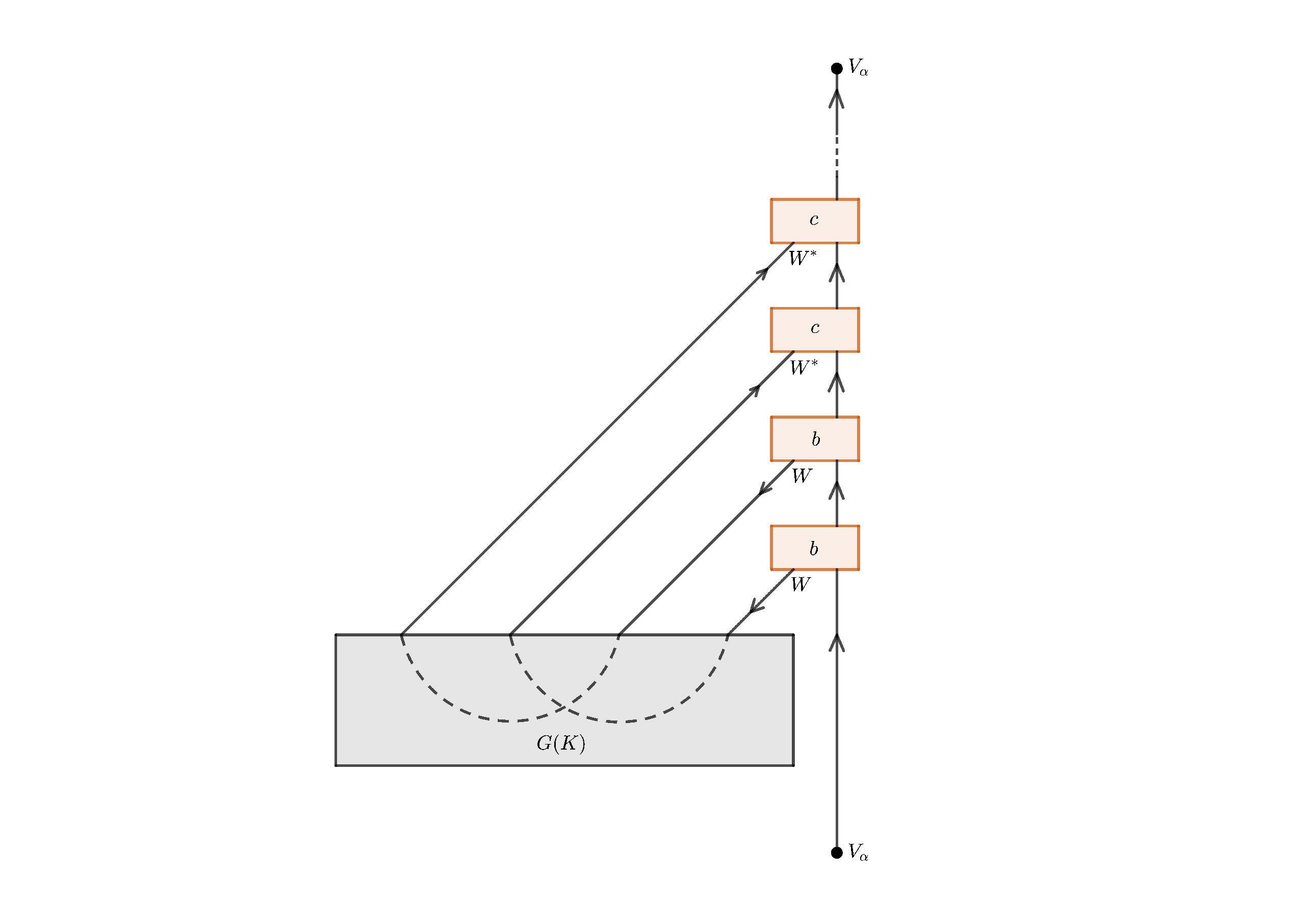}
}
\caption{The diagram used to compute $LG^{2,1}$.}\label{fig:FIG7}
\end{figure}

As a consequence of Theorem~\ref{thm:MAINRT}, $G(K) \in (W^* \otimes W^* \otimes W\otimes W)^{\otimes g}$ is a $\mathbb{C}(q)$-tensor, that is: 

\begin{lem}
Relatively to the tensor basis $(\mathcal{B}_{\alpha}^* \otimes \mathcal{B}_{\alpha}^* \otimes \mathcal{B}_{\alpha}\otimes \mathcal{B}_{\alpha})^{\otimes g}$, the coordinates of the tensor $G(K) \in (W^* \otimes W^* \otimes W\otimes W)^{\otimes g}$ are elements of $\mathbb{C}(q)$:
\begin{equation}\label{eq:tensor}
G(K) = \displaystyle\sum_{i_1,\ldots, i_{4g}} \lambda_{i_1,\ldots, i_{4g}}(q) \text{ } \bigotimes\limits_{k=1}^{g}
v_{i_{4k-3}}^{\ast} \otimes v_{i_{4k-2}}^{\ast} \otimes v_{i_{4k-1}} \otimes v_{i_{4k}} , \,\lambda_{i_1,\ldots, i_{4g}}(q) \in \mathbb{C}(q).
\end{equation}
\end{lem}

\begin{proof}
    It is a straightforward consequence of Lemma~\ref{lem:naturality}. Indeed, the $R$-matrix, cups, and caps are $\alpha$-independent once they are represented in the proper basis through $\Theta$.
\end{proof}

\subsection{Recovering the genus bound on the (1-1)-tangle}

Now we will use the (1-1)-tangle diagram described in Figure~\ref{fig:FIG7} to prove Theorem~\ref{thm:MAIN}. The main problem we face is that the coefficients in matrices $A$ and $\tilde{A}$, that will ``feed'' the vector we apply to $Q^{U_{q}\mathfrak{gl}(2 \vert 1)^{\sigma} , V_{\alpha}}(T)$ with degree in $q^{\alpha}$, are quite complicated. However, we are not interested in the coefficients per se. We wish to calculate an upper bound for the span of $LG(K,q,q^{\alpha})$, so in essence we want an upper bound for the highest degree monomial and a lower bound for the lowest degree monomial of the Laurent polynomial in $q^{\alpha}$ that we are considering. And naturally we would prefer these bounds not to be too loose, that is: they imply the genus bound conjectured in \cite{Ko}, and not some weaker version of the bound. \newline

To do so we will:
\begin{itemize}
    \item first try to simplify matrices $A$ and $\tilde{A}$, and more generally morphisms $b$ and $c$ from the point of view of their action on the span of the resulting Laurent polynomial;
    \item then try to essentialize the different morphisms at hand as regards their contributions in terms of the highest (lowest) degree monomial in $LG(K,q,q^{\alpha})$. 
\end{itemize}

\begin{prop}
The contribution of the negative cap $\Omega^-:V_{\alpha} \otimes V_{\alpha}^{\ast}\rightarrow\mathbb\mathbb{C}(q,q^{\alpha})$ to the span of $LG(K,q,q^{\alpha})$ can be ignored. 
\end{prop}

\begin{proof}
In terms of span, 
$$
\Omega^- = \begin{pmatrix}
q^{2\alpha} & 0 & 0 & 0 \\
0 & -q^{2\alpha} & 0 & 0 \\
0 & 0 & -q^{2(\alpha+1)} & 0 \\
0 & 0 & 0 & q^{2(\alpha+1)}
\end{pmatrix}
$$
can be replaced by 
$$
\Omega^- = \begin{pmatrix}
1 & 0 & 0 & 0 \\
0 & -1 & 0 & 0 \\
0 & 0 & -q^{2} & 0 \\
0 & 0 & 0 & q^{2}
\end{pmatrix}
$$
in the computation of $Q^{U_{q}\mathfrak{gl}(2 \vert 1)^{\sigma} , V_{\alpha}}(T)$. Indeed, this simply multiplies the polynomial by a factor $(q^{-2\alpha})^{4g}$, which shifts all coefficients but does not change the span.
\end{proof}
For the same reason, we can apply the following adjustment:
\begin{prop}
$\tilde{A}$ can be replaced by $q^{\alpha}\tilde{A}$ in the computation of $Q^{U_{q}\mathfrak{gl}(2 \vert 1)^{\sigma} , V_{\alpha}}(T)$ without changing the span of $LG(K,q,q^{\alpha})$.
\end{prop}

We will apply these two changes from this point onwards.

\begin{defn}\label{degree}
Set $z = q^{\alpha}$ and $t = q^{-\alpha}$. Then for $P(q,q^{\alpha}) \in \mathbb{C}(q,q^{\alpha})$, one can define $deg_z(P(q,q^{\alpha}))$ as the degree of the dominant coefficient in the expansion of $P(q,q^{\alpha})$ as a Laurent series in variable $z$ - if such an expansion exists. Likewise we also introduce $deg_t(P(q,q^{\alpha}))$.
\end{defn}

\begin{ex}
    If for example $P$ is a Laurent polynomial in variable $q^{\alpha}$, say $P(q,q^{\alpha}) = q^2 q^{\alpha} - q^{4\alpha}$, then the expansion of $P$ in terms of Laurent series in $z$ or in $t$ can be read directly on the Laurent polynomial itself. Here, $deg_z(P(q,q^{\alpha})) = 4$ and $deg_t(P(q,q^{\alpha})) = -1$. When the rational function $P(q,q^{\alpha}) \in \mathbb{C}(q,q^{\alpha})$ is not a Laurent polynomial, one needs to develop it as a Laurent series first - when it is possible - in order to read these degrees. An example of this where the computations are done is presented hereafter.
\end{ex}

\begin{prop}
The following $z$ and $t$ degrees can be computed, and will be useful later in the discussion:
$$
deg_z(q^{\alpha}) = deg_z([\alpha]_{q}) = deg_z([\alpha+1]_{q}) = 1 ;
$$
$$
deg_z(q^{-\alpha}) = deg_z([\alpha]_{q}^{-1}) = deg_z([\alpha+1]_{q}^{-1}) = -1 ;
$$
$$
deg_t(q^{-\alpha}) = deg_t([\alpha]_{q}) = deg_t([\alpha+1]_{q}) = 1 ;
$$
$$
deg_t(q^{\alpha}) = deg_t([\alpha]_{q}^{-1}) = deg_t([\alpha+1]_{q}^{-1}) = -1.
$$
\end{prop}

\begin{proof}
    Let us for example compute $deg_z([\alpha]_{q}^{-1})$. The others follow similarly or more easily.

    \begin{align*}
    &   &[\alpha]_{q}^{-1} = & \text{  } \frac{q - q^{-1}}{q^{\alpha} - q^{-\alpha}} = \frac{q - q^{-1}}{z - z^{-1}} = z^{-1} \frac{q - q^{-1}}{1 - z^{-2}}\\
   & & = &\text{  }(q - q^{-1}) z^{-1}(1 + z^{-2} + z^{-4} + z^{-6} + \ldots )\\
   & & = &\text{  }(q - q^{-1}) (z^{-1} + z^{-3} + z^{-5} + z^{-7} + \ldots ), \\
\end{align*}
so using the definition of the $z$ degree, $deg_z([\alpha]_{q}^{-1}) = -1$.
\end{proof}

\begin{prop}
For $P,Q$ such that the degrees exist, $$deg_z(PQ) = deg_z(P) + deg_z(Q) \text{ and } deg_t(PQ) = deg_t(P) + deg_t(Q).$$
\end{prop}

These remarks allow us to compute the following 4 matrices : $A_z = (z^{deg_z(Aij)})$, $A_t = (t^{deg_t(Aij)})$, $\tilde{A}_z = (z^{deg_z(\tilde{A}ij)})$ and $\tilde{A}_t = (t^{deg_t(\tilde{A}ij)})$. Matrices $A_z$, $A_t$, $\tilde{A}_z$ and $\tilde{A}_t$ are the reductions of matrices $A$ and $\tilde{A}$ with respect to the degree notions we just defined. They will be a way to track the highest (lowest) degree monomial while computing $LG(K,q,q^{\alpha})$ through the tangle described in Figure~\ref{fig:FIG7}. \\

We will write these four matrices relative to the bases we used in the previous section and that we recall here.

\begin{defn}
Given that $V_{\alpha}$ has a basis $(e_1, e_2, e_3, e_4)$, the basis we consider on vector space $V_{\alpha} \otimes V_{\alpha}^{\ast}$ is $(e_i \otimes e_j^{\ast})$. On the other hand, the basis attached to $W$ is $\mathcal{B}_{\alpha} = (v_1, \ldots, v_{16})$ as defined in Lemma~\ref{lem:basis}, and so the basis on the dual space $W^{\ast}$ is the dual basis $\mathcal{B}_{\alpha}^{\ast}$.
\end{defn}

So $A$ is the matrix for isomorphism $a$   relative to bases $\mathcal{B}_{\alpha}$ and $(e_i \otimes e_j^{\ast})$ while $\tilde{A}$ is the matrix for isomorphism $\tilde{a}$ relative to bases $\mathcal{B}_{\alpha}^{\ast}$ and $(e_i \otimes e_j^{\ast})$. Their reductions are expressed in the same bases. Explicitly, we have:\newline
\newline

$$
A_{z} =
  \begin{blockarray}{@{}ccccccccccccccccccc@{}}
   & \matindex{1} & \matindex{8} & \matindex{9} & \matindex{16} &\matindex{2} & \matindex{5} &\matindex{15} & \matindex{12} &\matindex{3} & \matindex{4} &\matindex{14} & \matindex{13} &\matindex{6} & \matindex{11} &\matindex{10} & \matindex{7}& \\
    \begin{block}{(cccccccccccccccccc)c}
     & 1 & z^{2} & z^{2} & z^{4} & 0 & 0 & 0 & 0 & 0 & 0 & 0 & 0 & 0 & 0 & 0 & 0 & &\matindex{(1,1)} \\
    & 0 & z^{2} & z^{2} & z^{4} & 0 & 0 & 0 & 0 & 0 & 0 & 0 & 0 & 0 & 0 & 0 & 0 & &\matindex{(2,2)} \\
     & 0 & 0 & z^{2} & z^{4} & 0 & 0 & 0 & 0 & 0 & 0 & 0 & 0 & 0 & 0 & 0 & 0 & &\matindex{(3,3)} \\
     & 0 & 0 & 0 & z^{4} & 0 & 0 & 0 & 0 & 0 & 0 & 0 & 0 & 0 & 0 & 0 & 0 & &\matindex{(4,4)} \\
     & 0 & 0 & 0 & 0 & z^{1/2} & 0 & z^{5/2} & 0 & 0 & 0 & 0 & 0 & 0 & 0 & 0 & 0 & &\matindex{(1,2)} \\
     & 0 & 0 & 0 & 0 & 0 & z^{3/2} & 0 & z^{7/2} & 0 & 0 & 0 & 0 & 0 & 0 & 0 & 0 & &\matindex{(2,1)} \\
     & 0 & 0 & 0 & 0 & 0 & 0 & z^{5/2} & 0 & 0 & 0 & 0 & 0 & 0 & 0 & 0 & 0 & &\matindex{(3,4)} \\
     & 0 & 0 & 0 & 0 & 0 & 0 & 0 & z^{7/2} & 0 & 0 & 0 & 0 & 0 & 0 & 0 & 0 & &\matindex{(4,3)} \\
     & 0 & 0 & 0 & 0 & 0 & 0 & 0 & 0 & z^{1/2} & 0 & z^{5/2} & 0 & 0 & 0 & 0 & 0 & &\matindex{(1,3)} \\
     & 0 & 0 & 0 & 0 & 0 & 0 & 0 & 0 & 0 & z^{3/2} & 0 & z^{7/2} & 0 & 0 & 0 & 0 & &\matindex{(3,1)} \\
     & 0 & 0 & 0 & 0 & 0 & 0 & 0 & 0 & 0 & 0 & z^{5/2} & 0 & 0 & 0 & 0 & 0 & &\matindex{(2,4)} \\
     & 0 & 0 & 0 & 0 & 0 & 0 & 0 & 0 & 0 & 0 & 0 & z^{7/2} & 0 & 0 & 0 & 0 & &\matindex{(4,2)} \\
     & 0 & 0 & 0 & 0 & 0 & 0 & 0 & 0 & 0 & 0 & 0 & 0 & z & 0 & 0 & 0 & &\matindex{(1,4)} \\
     & 0 & 0 & 0 & 0 & 0 & 0 & 0 & 0 & 0 & 0 & 0 & 0 & 0 & z^{3} & 0 & 0 & &\matindex{(4,1)} \\
     & 0 & 0 & 0 & 0 & 0 & 0 & 0 & 0 & 0 & 0 & 0 & 0 & 0 & 0 & z^{2} & 0 & &\matindex{(2,3)} \\
     & 0 & 0 & 0 & 0 & 0 & 0 & 0 & 0 & 0 & 0 & 0 & 0 & 0 & 0 & 0 & z^{2} & &\matindex{(3,2)} \\
    \end{block}
  \end{blockarray}
$$

$\tilde{A}_z =$
  \begin{blockarray}{@{}ccccccccccccccccccc@{}}
   & \matindex{1} & \matindex{8} & \matindex{9} & \matindex{16} &\matindex{2} & \matindex{5} &\matindex{15} & \matindex{12} &\matindex{3} & \matindex{4} &\matindex{14} & \matindex{13} &\matindex{6} & \matindex{11} &\matindex{10} & \matindex{7}& \\
    \begin{block}{(cccccccccccccccccc)c}
     & $1$ & 0 & 0 & 0 & 0 & 0 & 0 & 0 & 0 & 0 & 0 & 0 & 0 & 0 & 0 & 0 & &\matindex{(1,1)} \\
    & $1$ & $z^{-2}$ & 0 & 0 & 0 & 0 & 0 & 0 & 0 & 0 & 0 & 0 & 0 & 0 & 0 & 0 & &\matindex{(2,2)} \\
     & $1$ & $z^{-2}$ & $z^{-2}$ & 0 & 0 & 0 & 0 & 0 & 0 & 0 & 0 & 0 & 0 & 0 & 0 & 0 & &\matindex{(3,3)} \\
     & $1$ & $z^{-2}$ & $z^{-2}$ & $z^{-4}$ & 0 & 0 & 0 & 0 & 0 & 0 & 0 & 0 & 0 & 0 & 0 & 0 & &\matindex{(4,4)} \\
     & 0 & 0 & 0 & 0 & 0 & $z^{-3/2}$ & 0 & 0 & 0 & 0 & 0 & 0 & 0 & 0 & 0 & 0 & &\matindex{(1,2)} \\
     & 0 & 0 & 0 & 0 & $z^{-1/2}$ & 0 & 0 & 0 & 0 & 0 & 0 & 0 & 0 & 0 & 0 & 0 & &\matindex{(2,1)} \\
     & 0 & 0 & 0 & 0 & 0 & $z^{-3/2}$ & 0 & $z^{-7/2}$ & 0 & 0 & 0 & 0 & 0 & 0 & 0 & 0 & &\matindex{(3,4)} \\
     & 0 & 0 & 0 & 0 & $z^{-1/2}$ & 0 & $z^{-5/2}$ & 0 & 0 & 0 & 0 & 0 & 0 & 0 & 0 & 0 & &\matindex{(4,3)} \\
     & 0 & 0 & 0 & 0 & 0 & 0 & 0 & 0 & 0 & $z^{-3/2}$ & 0 & 0 & 0 & 0 & 0 & 0 & &\matindex{(1,3)} \\
     & 0 & 0 & 0 & 0 & 0 & 0 & 0 & 0 & $z^{-1/2}$ & 0 & 0 & 0 & 0 & 0 & 0 & 0 & &\matindex{(3,1)} \\
     & 0 & 0 & 0 & 0 & 0 & 0 & 0 & 0 & 0 & $z^{-3/2}$ & 0 & $z^{-7/2}$ & 0 & 0 & 0 & 0 & &\matindex{(2,4)} \\
     & 0 & 0 & 0 & 0 & 0 & 0 & 0 & 0 & $z^{-1/2}$ & 0 & $z^{-5/2}$ & 0 & 0 & 0 & 0 & 0 & &\matindex{(4,2)} \\
     & 0 & 0 & 0 & 0 & 0 & 0 & 0 & 0 & 0 & 0 & 0 & 0 & 0 & $z^{-3}$ & 0 & 0 & &\matindex{(1,4)} \\
     & 0 & 0 & 0 & 0 & 0 & 0 & 0 & 0 & 0 & 0 & 0 & 0 & $z^{-1}$ & 0 & 0 & 0 & &\matindex{(4,1)} \\
     & 0 & 0 & 0 & 0 & 0 & 0 & 0 & 0 & 0 & 0 & 0 & 0 & 0 & 0 & 0 & $z^{-2}$ & &\matindex{(2,3)} \\
     & 0 & 0 & 0 & 0 & 0 & 0 & 0 & 0 & 0 & 0 & 0 & 0 & 0 & 0 & $z^{-2}$ & 0 & &\matindex{(3,2)} \\
    \end{block}
  \end{blockarray}

$$A_t =
  \begin{blockarray}{@{}ccccccccccccccccccc@{}}
   & \matindex{1} & \matindex{8} & \matindex{9} & \matindex{16} &\matindex{2} & \matindex{5} &\matindex{15} & \matindex{12} &\matindex{3} & \matindex{4} &\matindex{14} & \matindex{13} &\matindex{6} & \matindex{11} &\matindex{10} & \matindex{7}& \\
    \begin{block}{(cccccccccccccccccc)c}
     & 1 & 1 & 1 & 1 & 0 & 0 & 0 & 0 & 0 & 0 & 0 & 0 & 0 & 0 & 0 & 0 & &\matindex{(1,1)} \\
    & 0 & 1 & 1 & 1 & 0 & 0 & 0 & 0 & 0 & 0 & 0 & 0 & 0 & 0 & 0 & 0 & &\matindex{(2,2)} \\
     & 0 & 0 & 1 & 1 & 0 & 0 & 0 & 0 & 0 & 0 & 0 & 0 & 0 & 0 & 0 & 0 & &\matindex{(3,3)} \\
     & 0 & 0 & 0 & 1 & 0 & 0 & 0 & 0 & 0 & 0 & 0 & 0 & 0 & 0 & 0 & 0 & &\matindex{(4,4)} \\
     & 0 & 0 & 0 & 0 & t^{-1/2} & 0 & t^{-1/2} & 0 & 0 & 0 & 0 & 0 & 0 & 0 & 0 & 0 & &\matindex{(1,2)} \\
     & 0 & 0 & 0 & 0 & 0 & t^{1/2} & 0 & t^{1/2} & 0 & 0 & 0 & 0 & 0 & 0 & 0 & 0 & &\matindex{(2,1)} \\
     & 0 & 0 & 0 & 0 & 0 & 0 & t^{-1/2} & 0 & 0 & 0 & 0 & 0 & 0 & 0 & 0 & 0 & &\matindex{(3,4)} \\
     & 0 & 0 & 0 & 0 & 0 & 0 & 0 & t^{1/2} & 0 & 0 & 0 & 0 & 0 & 0 & 0 & 0 & &\matindex{(4,3)} \\
     & 0 & 0 & 0 & 0 & 0 & 0 & 0 & 0 & t^{-1/2} & 0 & t^{-1/2} & 0 & 0 & 0 & 0 & 0 & &\matindex{(1,3)} \\
     & 0 & 0 & 0 & 0 & 0 & 0 & 0 & 0 & 0 & t^{1/2} & 0 & t^{1/2} & 0 & 0 & 0 & 0 & &\matindex{(3,1)} \\
     & 0 & 0 & 0 & 0 & 0 & 0 & 0 & 0 & 0 & 0 & t^{-1/2} & 0 & 0 & 0 & 0 & 0 & &\matindex{(2,4)} \\
     & 0 & 0 & 0 & 0 & 0 & 0 & 0 & 0 & 0 & 0 & 0 & t^{1/2} & 0 & 0 & 0 & 0 & &\matindex{(4,2)} \\
     & 0 & 0 & 0 & 0 & 0 & 0 & 0 & 0 & 0 & 0 & 0 & 0 & t^{-1} & 0 & 0 & 0 & &\matindex{(1,4)} \\
     & 0 & 0 & 0 & 0 & 0 & 0 & 0 & 0 & 0 & 0 & 0 & 0 & 0 & t & 0 & 0 & &\matindex{(4,1)} \\
     & 0 & 0 & 0 & 0 & 0 & 0 & 0 & 0 & 0 & 0 & 0 & 0 & 0 & 0 & 1 & 0 & &\matindex{(2,3)} \\
     & 0 & 0 & 0 & 0 & 0 & 0 & 0 & 0 & 0 & 0 & 0 & 0 & 0 & 0 & 0 & 1 & &\matindex{(3,2)} \\
    \end{block}
  \end{blockarray}
$$

$$\tilde{A}_t =
  \begin{blockarray}{@{}ccccccccccccccccccc@{}}
   & \matindex{1} & \matindex{8} & \matindex{9} & \matindex{16} &\matindex{2} & \matindex{5} &\matindex{15} & \matindex{12} &\matindex{3} & \matindex{4} &\matindex{14} & \matindex{13} &\matindex{6} & \matindex{11} &\matindex{10} & \matindex{7}& \\
    \begin{block}{(cccccccccccccccccc)c}
     & 1 & 0 & 0 & 0 & 0 & 0 & 0 & 0 & 0 & 0 & 0 & 0 & 0 & 0 & 0 & 0 & &\matindex{(1,1)} \\
    & 1 & 1 & 0 & 0 & 0 & 0 & 0 & 0 & 0 & 0 & 0 & 0 & 0 & 0 & 0 & 0 & &\matindex{(2,2)} \\
     & 1 & 1 & 1 & 0 & 0 & 0 & 0 & 0 & 0 & 0 & 0 & 0 & 0 & 0 & 0 & 0 & &\matindex{(3,3)} \\
     & 1 & 1 & 1 & 1 & 0 & 0 & 0 & 0 & 0 & 0 & 0 & 0 & 0 & 0 & 0 & 0 & &\matindex{(4,4)} \\
     & 0 & 0 & 0 & 0 & 0 & t^{-1/2} & 0 & 0 & 0 & 0 & 0 & 0 & 0 & 0 & 0 & 0 & &\matindex{(1,2)} \\
     & 0 & 0 & 0 & 0 & t^{1/2} & 0 & 0 & 0 & 0 & 0 & 0 & 0 & 0 & 0 & 0 & 0 & &\matindex{(2,1)} \\
     & 0 & 0 & 0 & 0 & 0 & t^{-1/2} & 0 & t^{-1/2} & 0 & 0 & 0 & 0 & 0 & 0 & 0 & 0 & &\matindex{(3,4)} \\
     & 0 & 0 & 0 & 0 & t^{1/2} & 0 & t^{1/2} & 0 & 0 & 0 & 0 & 0 & 0 & 0 & 0 & 0 & &\matindex{(4,3)} \\
     & 0 & 0 & 0 & 0 & 0 & 0 & 0 & 0 & 0 & t^{-1/2} & 0 & 0 & 0 & 0 & 0 & 0 & &\matindex{(1,3)} \\
     & 0 & 0 & 0 & 0 & 0 & 0 & 0 & 0 & t^{1/2} & 0 & 0 & 0 & 0 & 0 & 0 & 0 & &\matindex{(3,1)} \\
     & 0 & 0 & 0 & 0 & 0 & 0 & 0 & 0 & 0 & t^{-1/2} & 0 & t^{-1/2} & 0 & 0 & 0 & 0 & &\matindex{(2,4)} \\
     & 0 & 0 & 0 & 0 & 0 & 0 & 0 & 0 & t^{1/2} & 0 & t^{1/2} & 0 & 0 & 0 & 0 & 0 & &\matindex{(4,2)} \\
     & 0 & 0 & 0 & 0 & 0 & 0 & 0 & 0 & 0 & 0 & 0 & 0 & 0 & t^{-1} & 0 & 0 & &\matindex{(1,4)} \\
     & 0 & 0 & 0 & 0 & 0 & 0 & 0 & 0 & 0 & 0 & 0 & 0 & t & 0 & 0 & 0 & &\matindex{(4,1)} \\
     & 0 & 0 & 0 & 0 & 0 & 0 & 0 & 0 & 0 & 0 & 0 & 0 & 0 & 0 & 0 & 1 & &\matindex{(2,3)} \\
     & 0 & 0 & 0 & 0 & 0 & 0 & 0 & 0 & 0 & 0 & 0 & 0 & 0 & 0 & 1 & 0 & &\matindex{(3,2)} \\
    \end{block}
  \end{blockarray}
$$

Taking a careful look at these 4 matrices has several consequences.

\begin{prop}\label{uperbound}
    $deg_z(LG(K,q,q^{\alpha})) \leq 8 \times g(K).$
\end{prop}

\begin{proof}
Since we have $Q^{U_{q}\mathfrak{gl}(2 \vert 1)^{\sigma} , V_{\alpha}}(T) = LG(K,q,q^{\alpha}) id_{V_{\alpha}}$, one way to compute $LG(K,q,q^{\alpha})$ is to compute
$Q^{U_{q}\mathfrak{gl}(2 \vert 1)^{\sigma} , V_{\alpha}}(T)(e_{1})$ using Figure~\ref{fig:FIG7}. \textbf{Our goal will be to keep track of the $z$-degree of the quantity we have while we advance the computation towards finding $Q^{U_{q}\mathfrak{gl}(2 \vert 1)^{\sigma} , V_{\alpha}}(T)(e_{1})$}.
\par
To compute $Q^{U_{q}\mathfrak{gl}(2 \vert 1)^{\sigma} , V_{\alpha}}(T)(e_{1})$, the first step is to make $e_{1}$ go through the ``box'' (operator) $G(K)$. After that, using \eqref{eq:tensor}, the element at hand is
$$G(K) \otimes e_{1}=
\sum\limits_{i_{1},\dots,i_{4g}}
\lambda_{i_{1},\dots,i_{4g}}(q)
\bigg(\bigotimes\limits_{k=1}^{g}
v_{i_{4k-3}}^{\ast} \otimes v_{i_{4k-2}}^{\ast} \otimes v_{i_{4k-1}} \otimes v_{i_{4k}}\bigg)\otimes e_{1}
.
$$
One wants to know what happens when computing $b(v_{i_{4g}} \otimes e_{1})$.\newline
$$b(v_{i_{4g}} \otimes e_{1})
= (id_{V_{\alpha}} \otimes \Omega^{-})
\circ (a \otimes id_{V_{\alpha}})(v_{i_{4g}}\otimes e_{1})
$$
$$
=
(id_{V_{\alpha}} \otimes \Omega^{-})(a(v_{i_{4g}})\otimes e_{1})
.$$
This element is equal to zero, except if $a(v_{i_{4g}})$ has a nonzero coefficient in front of $e_{1} \otimes e_{1}^{\ast}$, $e_{2} \otimes e_{1}^{\ast}$, $e_{3} \otimes e_{1}^{\ast}$ or $e_{4} \otimes e_{1}^{\ast}$. Reading matrix $A_z$ shows that this is the case if and only if
$$
v_{i_{4g}} \in \lbrace{ v_{1}, v_{8}, v_{9}, v_{16}, v_{5}, v_{12}, v_{4}, v_{13}, v_{11} \rbrace}
.
$$
In these different cases, $b(v_{i_{4g}} \otimes e_{1})$ is a non zero scalar multiple of one of the basis vectors $e_{1}$, $e_2$, $e_3$ or $e_4$. We call the $z$-degree of this scalar its \textit{weight}. The weight in the different possible cases is summed up in the following graph that should be read for example considering the second edge from the left: ``if $v_{i_{4g}} = v_8$, then $b(v_{i_{4g}} \otimes e_{1})$ is a scalar multiple of $e_1$, and that scalar has weight $z^2$.''

\begin{figure}[H]
\includegraphics[scale=0.4]{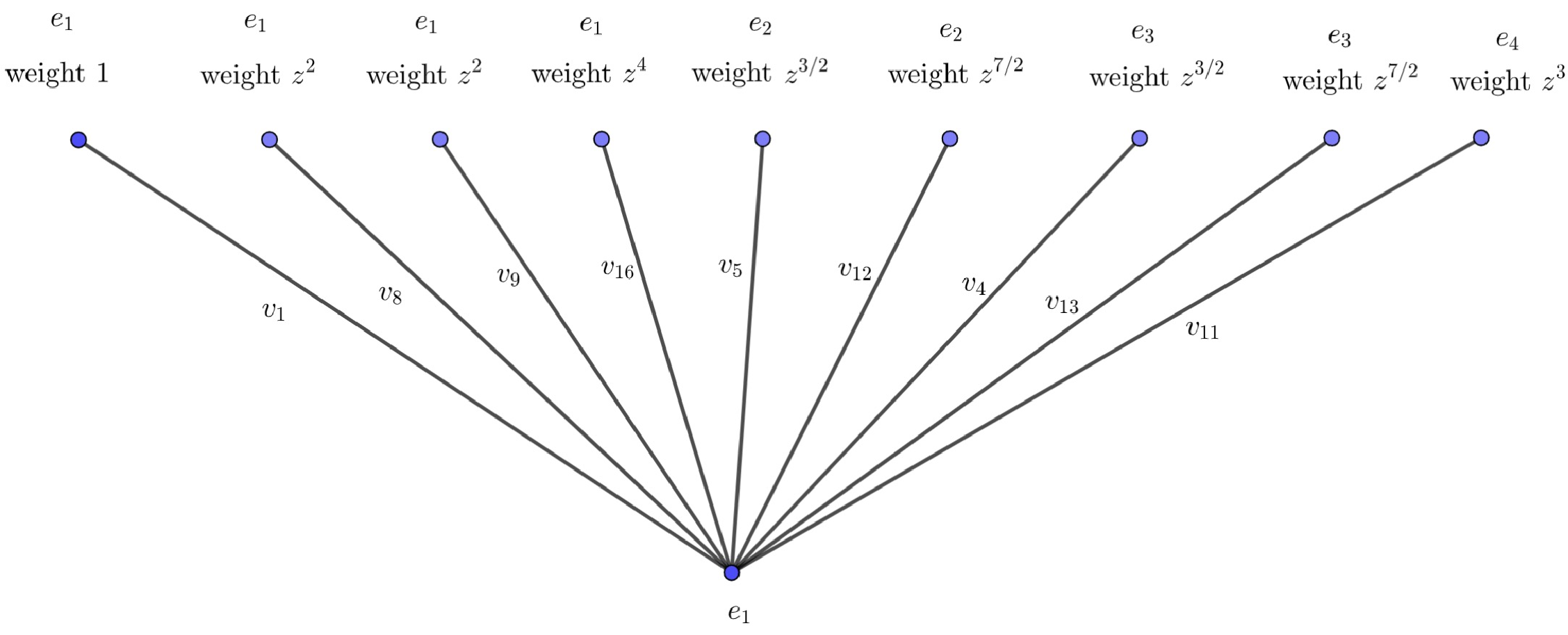}
\end{figure}

In the worst case, the contribution of map $b$ to the degree in $z$ is $z^4$.
\par
Now as we advance the computation, the initial vector entering box $b$ will not necessarily be $e_1$ anymore. However, drawing similar graphs from matrix $A_z$ for initial vectors $e_2$, $e_3$ and $e_4$, one can see that this $z^4$ upper bound is true regardless of the basis vector entering $b$. Analog considerations on matrix $\tilde{A}_z$ prove that in the worst case scenario, the contribution of $c$ to the degree in $z$ is $z^0 = 1$.
\par
Therefore, going through $b,b,c,c$, which is what happens when computing $Q^{U_{q}\mathfrak{gl}(2 \vert 1)^{\sigma} , V_{\alpha}}(T)(e_{1})$, we obtain $z^{4}\cdot z^{4}\cdot 1 \cdot 1=z^{8}$ as the highest possible degree in $z$.
\par
Note that this is a possible scenario, if you consider that $G(K)$ is generic, and so in that sense that bound is sharp:
\begin{figure}[H]
\includegraphics[scale=0.9]{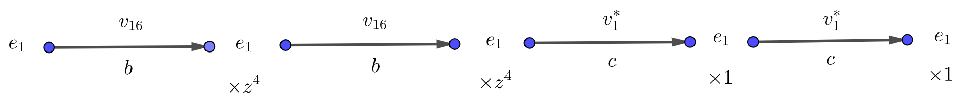}
\end{figure}
\par
Now there are $g$ consecutive ``boxes'' $(b,b,c,c)$ (where $g$ is the genus of a Seifert surface for $K$ with minimal genus) in Figure~\ref{fig:FIG7}, and we deduce that
$$
deg_{z}(LG(K,q,q^{\alpha})) \leq 8 \times g(K)
.
$$
\end{proof}

\begin{prop}\label{lowerbound}
     $deg_t(LG(K,q,q^{\alpha})) \leq 0.$
\end{prop}

\begin{proof}
Like in the proof of Proposition~\ref{uperbound}, we compute $Q^{U_{q}\mathfrak{gl}(2 \vert 1)^{\sigma} , V_{\alpha}}(T)(e_{1})$ and \textbf{we want to keep track of the $t$-degree of the quantity we have while we advance the computation towards finding the result}. We just have to understand what the cost of entering a box $b$ or $c$ is in terms of degree in $t$ for a vector $e_{1}$, $e_{2}$, $e_{3}$ or $e_{4}$ (and what the outcoming vector is). Like in the previous proof, this amounts to computing $b(v_i \otimes e_j)$ and $c(v_i^* \otimes e_j)$ for all possible values of $i$ and $j$. For each of these quantities that is non zero, it is a scalar multiple of one of the basis vectors $e_k$. We are interested in the $t$-degree of that scalar, and in the value of $k$ as a function of $(i,j)$. 
\par
This comes entirely from observing the coefficients of matrices $A_{t}$ and $\tilde{A_{t}}$. In this case, we can check that the $t$-degree of the scalar only depends on the value of the pair $(j,k)$, which encourages us to interpret vectors $e_1$, $e_2$, $e_3$ and $e_4$ as \textit{states} that evolve as the computation progresses, with the $t$-degree that will evolve in parallel. We refer to the degree of a specific scalar as its weight like in the proof of Proposition~\ref{uperbound} and we use the same graph notation to track these weights. We write all possible such graphs below. Specifically, the left most edge in the top left graph should be understood as stating: ``for any $i$ such that $b(v_i \otimes e_1)$ is non zero, $b(v_i \otimes e_1) = \lambda_i(q,t) ~ e_1$ with the weight of $\lambda_i(q,t)$ being $1$.''

\begin{itemize}
    \item \underline{Entering in state $e_1$}:
    \begin{figure}[H]
\includegraphics[scale=0.55]{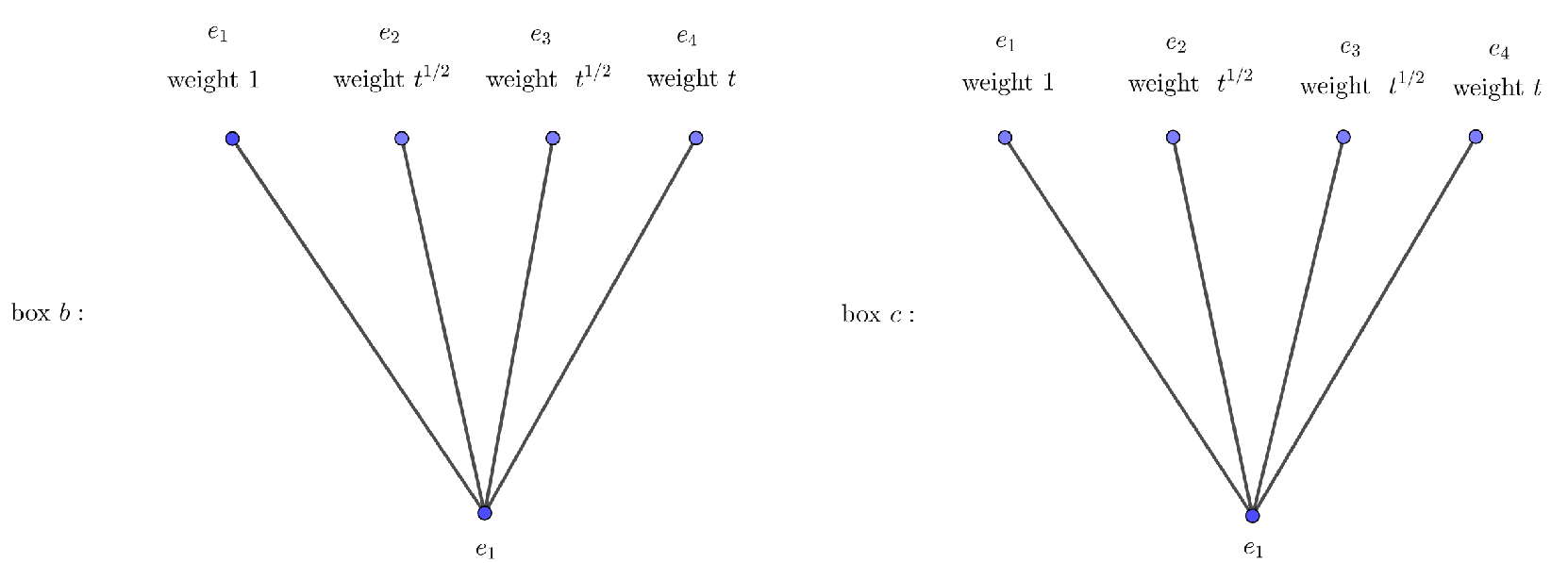}
\end{figure}
    \item \underline{Entering in state $e_2$}:
    \begin{figure}[H]
\includegraphics[scale=0.55]{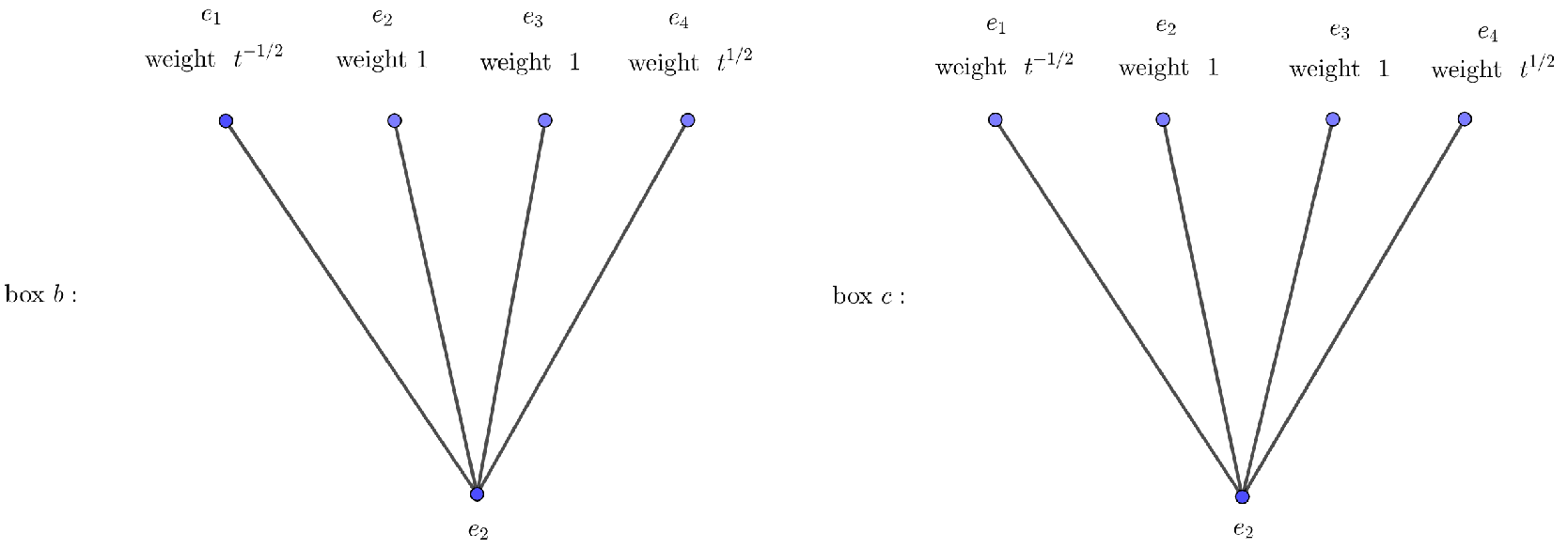}
\end{figure}
    \item \underline{Entering in state $e_3$}:
    \begin{figure}[H]
\includegraphics[scale=0.55]{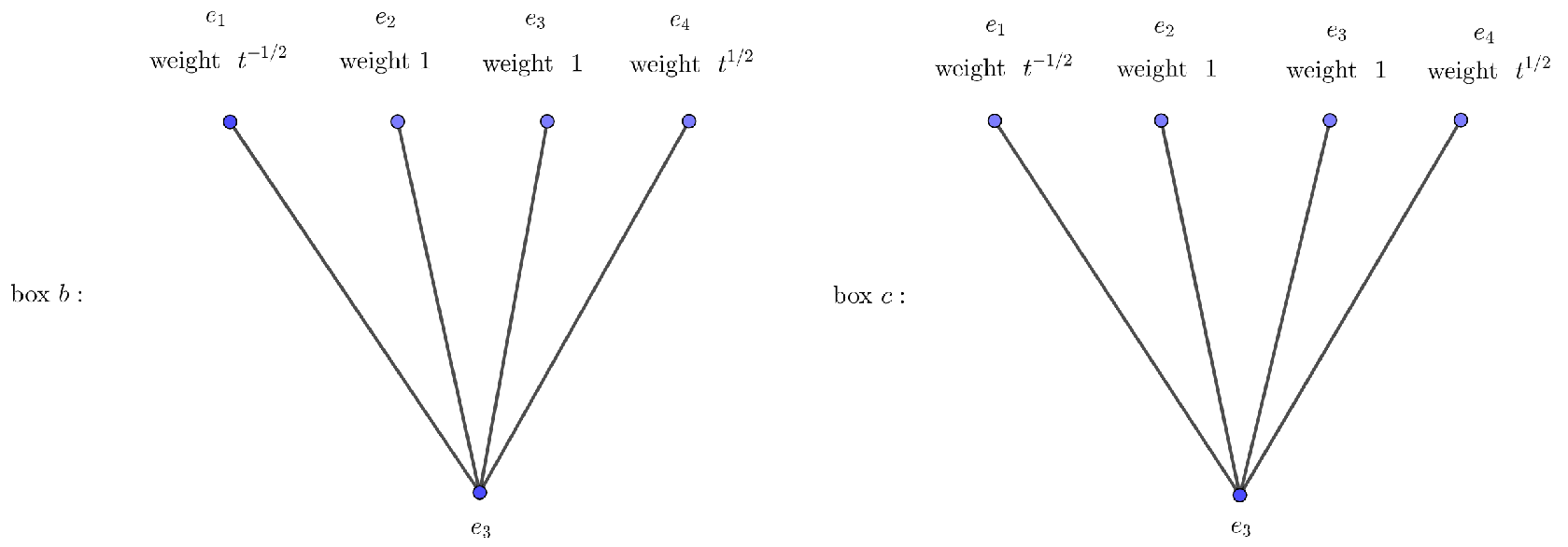}
\end{figure}
    \item \underline{Entering in state $e_4$}:
    \begin{figure}[H]
\includegraphics[scale=0.55]{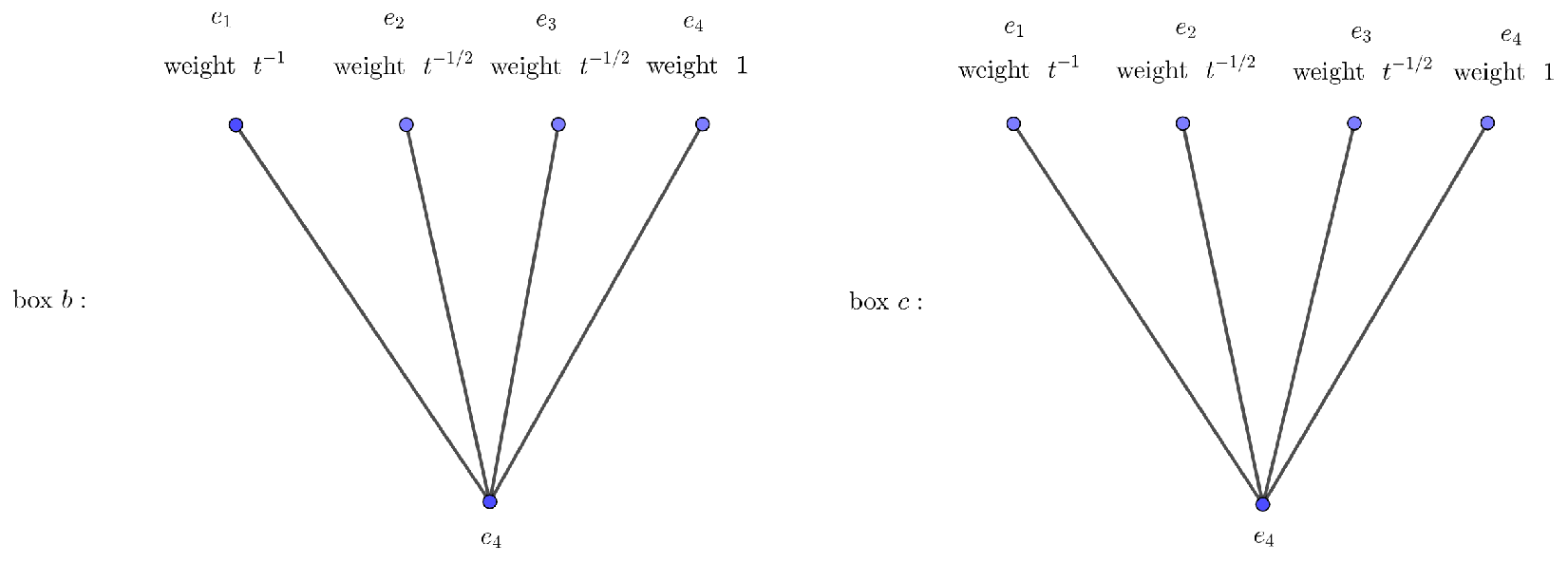}
\end{figure}
\end{itemize}

Having a careful look at these 8 graphs makes several thing appear to be true:

\begin{itemize}
    \item On each of the four previous rows, the left graph and the right one are the same. This means that with respect to what we are interested in, going through $b$ or through $c$ is equivalent. This means that boxes $b$ and $c$ can be \textit{identified} here;
    \item Vectors $e_2$ and $e_3$ always play the exact same role and can be \textit{identified} as well;
    \item These graphs are coherent weight-wise, in the sense that if in one of them, you go from $e_i$ to $e_j$ with weight $t^x$, then you will have the same weight $t^x$ for \textit{any} transition from $e_i$ to $e_j$;
    \item These graphs are also coherent orientation-wise because if the weight for the transition from state $e_i$ to state $e_j$ is $t^x$, then the weight for the opposite transition from $e_j$ to $e_i$ is $t^{-x}$.
\end{itemize}

This means that if you track the degree in $t$ while advancing the computation of $Q^{U_{q}\mathfrak{gl}(2 \vert 1)^{\sigma} , V_{\alpha}}(T)(e_{1})$, the degree of any term at any step of these intermediate computations is a \textit{state function} in the following sense: it only depends on the term's state, that is if it is a multiple of $e_1$, $e_2 = e_3$, or $e_4$. The computation of $Q^{U_{q}\mathfrak{gl}(2 \vert 1)^{\sigma} , V_{\alpha}}(T)(e_{1})$ can thus be summed up using the following graph:
\begin{figure}[H]
\includegraphics[scale=0.9]{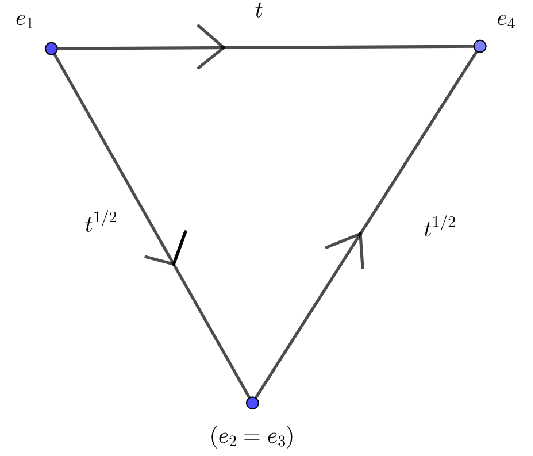}
\end{figure} But since $Q^{U_{q}\mathfrak{gl}(2 \vert 1)^{\sigma} , V_{\alpha}}(T)(e_{1}) = \lambda e_1$, at the end of the computation all terms return to their initial state: $e_1$. In particular they all have the same degree in $t$ as the initial term: $0$. This proves that 
$$
deg_t(\lambda) = deg_t(LG(K,q,q^{\alpha})) \leq 0.
$$

\end{proof}

Combining Propositions \ref{uperbound} and \ref{lowerbound} proves that $span_{q^{\alpha}}(LG(K,q,q^{\alpha}))\leq 8 \times g(K).$ Now one should remember that the span in variables $(t_0, t_1)$ corresponds to the span with respect to the variable $q^{-2\alpha}$, which provides the inequality stated in Theorem \ref{thm:MAIN}.

\begin{rmk}
    The reader may wonder why analyzing the maximum and minimum degree coefficients of the $A$ and $\tilde{A}$ matrices was not enough to find the lower bound for $4 \times g(K)$. The reason is that such an attempt can be made and will indeed result in some lower bound for the 3-genus of knots, but not the optimal one we wish to prove. The dynamic argument we just presented fixes that problem in a somewhat mysterious way, suggesting that a more natural proof should exist. 
    
   We expect that a classical Alexander-like interpretation of the Links–Gould polynomial would yield a more straightforward explanation of the genus bound than the approach taken in the present work.
\end{rmk}

\section{Some interesting topological consequences}\label{sec:Consequences}

The first question that comes to mind is whether this new genus bound actually is better than the one provided by the Alexander polynomial $\Delta$. In other words, is there at least one knot $K$ such that the bound obtained using $LG$ enhances the bound produced by $\Delta$:
$$
2 \times span (\Delta_{K}(t_{0})) < span(LG(K,t_{0},t_{1})) \text{ ?}
$$
This is indeed the case. We will in particular look at what happens for the Kinoshita-Terasaka and Conway pair of mutant knots, where both the Alexander invariant and the Levine-Tristram signature fail to detect genus.

\subsection{The genus bound provided by $LG$ improves the one derived from the Alexander invariant}

For prime knots of 10 or fewer crossings, the span of the Alexander invariant is always equal to the Seifert genus of the knot. There are seven prime knots with 11 crossings for which the Alexander polynomial fails to recover the 3-genus. When listed with respect to the HTW ordering for tables of prime knots of up to 16 crossings (\cite{HTW}), these knots are:
$$
11^N_{34}, 11^N_{42}, 11^N_{45}, 11^N_{67}, 11^N_{73}, 11^N_{97} \text{ and } 11^N_{152}.
$$ For four of them, the lower bound obtained using $LG$ is more precise than the one computed from $\Delta$:
$$
11^N_{34} \text{ (the Conway knot)}, 11^N_{42} \text{ (the Kinoshita–Terasaka knot)}, 11^N_{67} \text{ and } 11^N_{97}.
$$

This is explained in \cite{Ko}, Proposition 2.1, where the genus bound was first conjectured, and supported by empirical evidence. The list of prime knots of 12 crossings for which we get a more precise lower bound with $LG$ than with $\Delta$ can also be found in \cite{Ko}. For the KT Conway pair of knots, this improvement is interesting since it means that the techniques developed in this work allow the computation of the genus of the Kinoshita-Terasaka knot. Still they fail to compute the genus of the Conway knot. However, in subsequent work it will be shown how the genus bound for $LG$ implies the existence of genus bounds for \textit{colored} Links-Gould invariants \cite{GHKKW} that will compute the genus of the Conway knot, and even the genus of all prime knots with up to $16$ crossings, see \cite{GL}.

\subsection{The Kinoshita-Terasaka and Conway pair of mutant knots}

The Kinoshita–Terasaka knot ($11^N_{42}$) and the Conway knot ($11^N_{34}$) are a pair of prime knots that are related by mutation, see Figure~\ref{fig:FIG8}. 

\begin{figure}
\adjustbox{trim=.5cm 0cm 0cm 0cm}{
  \includegraphics[scale=.9]{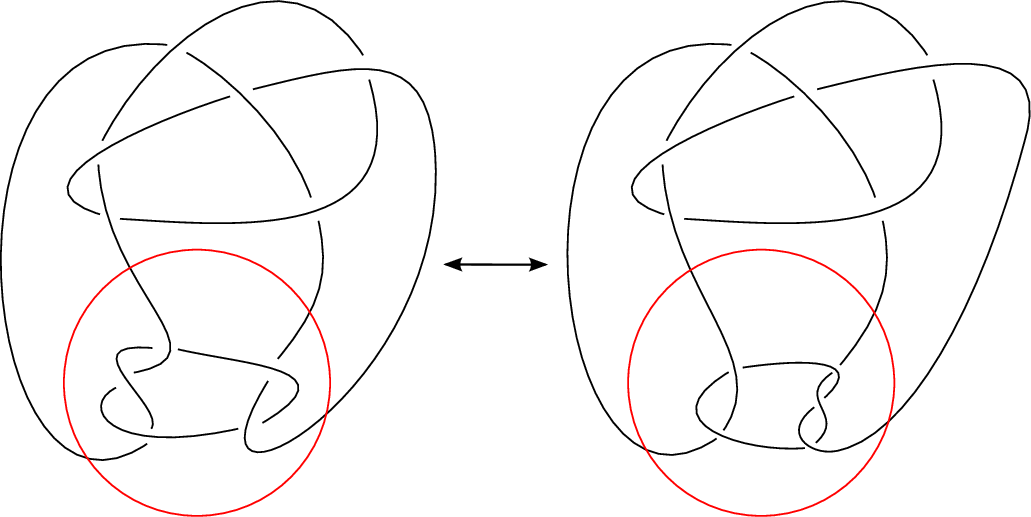}
  }
\caption{The Kinoshita–Terasaka knot (left), the Conway knot (right), and how they are related by mutation.}\label{fig:FIG8}
\end{figure} 

These two knots share the same Alexander invariant, as well as the same Jones polynomial. Also, their common Alexander polynomial ($=1$) and Levine-Tristram signature\footnote{The modulus of the Levine-Tristram signature gives a lower bound for the 4-genus of a link (see Powell's Murasugi-Tristram type inequality \cite{Pow}) that itself is a lower bound for the Seifert genus of the link.} ($=0$) fail to say anything about the genus of either of these knots. By work of Gabai \cite{Gab}, it is known that the Conway knot has genus $3$ while the Kinoshita-Terasaka knot has genus $2$. These knots have been studied extensively, including recently, e.g. \cite{Pic}. 

As a matter of fact, they also share the same $LG$, since $LG$ does not detect mutation. Still, the Links-Gould invariant offers a nontrivial lower bound for their genera.

\begin{thm}\label{thm:KTC}
$g(11^N_{34}) \geq 2$ and $g(11^N_{42}) \geq 2$.
\end{thm}

\begin{proof}
The value of the Links-Gould invariant for $11^N_{34}$ and $11^N_{42}$, that is the same since they are mutants, can be found in \cite{DWKL}, section 4.9. We translate it here because it is written using the substitution $p := q^{\alpha}$:

\begin{align*}
   LG(11^N_{34}) & = LG(11^N_{42}) & = & \text{ } q^{-6 \alpha}(-q^{-8}-q^{-6}+2 q^{-4} + q^{-2} -1) + q^{-4 \alpha}(q^{-8} + 6 q^{-6} -3 q^{-4} -9 q^{-2} + 2 + 3 q^2) \\
   & & &+ q^{-2 \alpha} (-7 q^{-6} - 7 q^{-4} + 18 q^{-2} + 9 - 11 q^2 - 2 q^4) + ( 2 q^{-6} + 14 q^{-4} -8 q^{-2} -23 \\& & &+ 6 q^2 + 10 q^4) + q^{2 \alpha}(-7 q^{-4} - 7 q^{-2} + 18 + 9 q^2 - 11 q^4 - 2 q^6) + q^{4 \alpha}(q^{-4} + 6 q^{-2} \\
   & & &- 3 - 9 q^2 + 2 q^4 + 3 q^6) + q^{6 \alpha} (-q^{-2} - 1 + 2 q^2 + q^4 - q^6). \\
\end{align*}
Hence, using Theorem~\ref{thm:MAIN}, we obtain $4 \times g \geq 6$. Since the genus of a knot is an integer number, it follows that the genus of these knots is at least two.
\end{proof}

\appendix
\section{Proofs of Lemma \ref{lem:basis} and Theorem \ref{thm:MAINRT}}\label{sec:proofs}
\begin{proof}[Proof of Lemma~\ref{lem:basis}]
 First we compute $v_1$, $v_2$, ... , $v_{16}$ in terms of basis $(e_{1}\otimes e_{1}^{\ast},e_{1}\otimes e_{2}^{\ast}, e_{1}\otimes e_{3}^{\ast},e_{1}\otimes e_{4}^{\ast}, e_{2}\otimes e_{1}^{\ast},e_{2}\otimes e_{2}^{\ast}, \ldots )$. To do so we use the matrices describing the action of $U_{q}\mathfrak{gl}(2 \vert 1)^{\sigma}$ on $V_{\alpha} \otimes V_{\alpha}^{\ast}$.\newline
$\text{  }$\newline
$v_1 = v = e_{1}\otimes e_{1}^{\ast}$ ; \newline
$\text{  }$\newline
$v_2 = E_{3}^{2}.v = -q^{\alpha/2}e_{1}\otimes e_{2}^{\ast}$ ; \newline
$\text{  }$\newline
$v_3 = E_{3}^{1}.v = -q^{-1}q^{\alpha/2}e_{1}\otimes e_{3}^{\ast}$ ; \newline
$\text{  }$\newline
$v_4 = E_{1}^{3}.v = q q^{\alpha/2}  [\alpha]_{q} e_{3}\otimes e_{1}^{\ast}$ ; \newline
$\text{  }$\newline
$v_5 = E_{2}^{3}.v = q^{\alpha/2}  [\alpha]_{q} e_{2}\otimes e_{1}^{\ast}$ ; \newline
$\text{  }$\newline
$v_6 = E_{3}^{2} E_{3}^{1}.v = -q^{-1}q^{\alpha}e_{1}\otimes e_{4}^{\ast}$ ; \newline
$\text{  }$\newline
$v_7 = E_{3}^{2} E_{1}^{3}.v = q^{3/2} q^{\alpha}  [\alpha]_{q} e_{3}\otimes e_{2}^{\ast}$ ; \newline
$\text{  }$\newline
$v_8 = E_{3}^{2} E_{2}^{3}.v = q^{\alpha}  [\alpha]_{q} (e_{1}\otimes e_{1}^{\ast} + e_{2}\otimes e_{2}^{\ast}$); \newline
$\text{  }$\newline
$v_9 = E_{3}^{1} E_{1}^{3}.v = q^{\alpha}  [\alpha]_{q} e_{1}\otimes e_{1}^{\ast} - (q^2-1) q^{\alpha}  [\alpha]_{q} e_{2}\otimes e_{2}^{\ast} + q^{\alpha}  [\alpha]_{q} e_{3}\otimes e_{3}^{\ast}$ ; \newline
$\text{  }$\newline
$v_{10} = E_{3}^{1} E_{2}^{3}.v = q^{-1/2} q^{\alpha}  [\alpha]_{q} e_{2}\otimes e_{3}^{\ast}$ ; \newline
$\text{  }$\newline
$v_{11} = E_{1}^{3} E_{2}^{3}.v = -q^{\alpha} [\alpha]_{q} [\alpha+1]_{q} e_{4}\otimes e_{1}^{\ast}$ ; \newline
$\text{  }$\newline
$v_{12} = E_{3}^{1} E_{1}^{3} E_{2}^{3}.v = q^{3 \alpha/2} [\alpha]_{q} [\alpha+1]_{q}e_{2}\otimes e_{1}^{\ast} + q^{-1/2} q^{3\alpha/2} [\alpha]_{q} [\alpha+1]_{q}e_{4}\otimes e_{3}^{\ast}$ ; \newline
$\text{  }$\newline
$v_{13} = E_{3}^{2} E_{1}^{3} E_{2}^{3}.v = -q^{3 \alpha/2} [\alpha]_{q} [\alpha+1]_{q}e_{3}\otimes e_{1}^{\ast} + q^{1/2} q^{3\alpha/2} [\alpha]_{q} [\alpha+1]_{q}e_{4}\otimes e_{2}^{\ast}$ ; \newline
$\text{  }$\newline
$v_{14} = E_{3}^{2} E_{3}^{1} E_{2}^{3}.v = q^{3 \alpha/2} [\alpha]_{q} e_{1}\otimes e_{3}^{\ast} - q^{-1/2} q^{3\alpha/2} [\alpha]_{q} e_{2}\otimes e_{4}^{\ast}$ ; \newline
$\text{  }$\newline
$v_{15} = E_{3}^{2} E_{3}^{1} E_{1}^{3}.v = - q^2 q^{3 \alpha/2} [\alpha]_{q} e_{1}\otimes e_{2}^{\ast} - q^{1/2} q^{3\alpha/2} [\alpha]_{q} e_{3}\otimes e_{4}^{\ast}$ ; \newline
$\text{  }$\newline
$v_{16} = E_{3}^{2} E_{3}^{1} E_{1}^{3} E_{2}^{3}.v = q^{2 \alpha} [\alpha]_{q} [\alpha+1]_{q}(e_{1}\otimes e_{1}^{\ast} + e_{2}\otimes e_{2}^{\ast} + e_{3}\otimes e_{3}^{\ast}+ e_{4}\otimes e_{4}^{\ast})$.\\ \newline
$\text{  }$\newline
The change of basis matrix can be written so that it is a triangular matrix with non zero diagonal coefficients. This makes it obvious that $\mathcal{B}_{\alpha}$ is a basis. For that purpose we order elements in  $\mathcal{B}_{\alpha}$ and in $(e_{i}\otimes e_{j}^{\ast})$ in a specific way. The change of basis matrix is written with the column labeled $a$ referring to $v_a$ and the row identified as $(b,c)$ corresponding to $e_{b}\otimes e_{c}^{\ast}$:
\newpage

\begin{landscape}
\thispagestyle{empty}
\centering
\small{

  \begin{blockarray}{@{}ccccccccccccccccccc@{}}
   & \matindex{1} & \matindex{8} & \matindex{9} & \matindex{16} &\matindex{2} & \matindex{5} &\matindex{15} & \matindex{12} &\matindex{3} & \matindex{4} &\matindex{14} & \matindex{13} &\matindex{6} & \matindex{11} &\matindex{10} & \matindex{7}& \\
    \begin{block}{(cccccccccccccccccc)c}
     & 1 & \begin{tabular}{@{}c@{}} $q^{\alpha }$ \\  $\times [\alpha]_{q} $\end{tabular} & \begin{tabular}{@{}c@{}} $q^{\alpha }$ \\  $\times [\alpha]_{q} $\end{tabular} & \begin{tabular}{@{}c@{}} $q^{2\alpha }[\alpha]_{q}$ \\  $\times [\alpha+1]_{q} $\end{tabular} & 0 & 0 & 0 & 0 & 0 & 0 & 0 & 0 & 0 & 0 & 0 & 0 & &\matindex{(1,1)} \\
    & 0 & \begin{tabular}{@{}c@{}} $q^{\alpha }$ \\  $\times [\alpha]_{q} $\end{tabular} & \begin{tabular}{@{}c@{}} -$q^{\alpha } [\alpha]_{q}$ \\  $\times (q^2-1) $\end{tabular} & \begin{tabular}{@{}c@{}} $q^{2\alpha }[\alpha]_{q}$ \\  $\times [\alpha+1]_{q} $\end{tabular} & 0 & 0 & 0 & 0 & 0 & 0 & 0 & 0 & 0 & 0 & 0 & 0 & &\matindex{(2,2)} \\
     & 0 & 0 & \begin{tabular}{@{}c@{}} $q^{\alpha }$ \\  $\times [\alpha]_{q} $\end{tabular} & \begin{tabular}{@{}c@{}} $q^{2\alpha }[\alpha]_{q}$ \\  $\times [\alpha+1]_{q} $\end{tabular} & 0 & 0 & 0 & 0 & 0 & 0 & 0 & 0 & 0 & 0 & 0 & 0 & &\matindex{(3,3)} \\
     & 0 & 0 & 0 & \begin{tabular}{@{}c@{}} $q^{2\alpha }[\alpha]_{q}$ \\  $\times [\alpha+1]_{q} $\end{tabular} & 0 & 0 & 0 & 0 & 0 & 0 & 0 & 0 & 0 & 0 & 0 & 0 & &\matindex{(4,4)} \\
     & 0 & 0 & 0 & 0 & $-q^{\alpha/2 }$ & 0 & \begin{tabular}{@{}c@{}} $-q^{2} q^{3\alpha/2 }$ \\  $\times [\alpha]_{q}$\end{tabular} & 0 & 0 & 0 & 0 & 0 & 0 & 0 & 0 & 0 & &\matindex{(1,2)} \\
     & 0 & 0 & 0 & 0 & 0 & \begin{tabular}{@{}c@{}} $q^{\alpha/2 }$ \\  $\times [\alpha]_{q} $\end{tabular} & 0 & \begin{tabular}{@{}c@{}} $q^{3\alpha/2 } [\alpha]_{q}$ \\  $\times [\alpha+1]_{q} $\end{tabular} & 0 & 0 & 0 & 0 & 0 & 0 & 0 & 0 & &\matindex{(2,1)} \\
     & 0 & 0 & 0 & 0 & 0 & 0 & \begin{tabular}{@{}c@{}} $-q^{1/2 } q^{3\alpha/2 }$ \\  $\times [\alpha]_{q} $\end{tabular} & 0 & 0 & 0 & 0 & 0 & 0 & 0 & 0 & 0 & &\matindex{(3,4)} \\
     & 0 & 0 & 0 & 0 & 0 & 0 & 0 & \begin{tabular}{@{}c@{}} $q^{-1/2 } q^{3\alpha/2 } $ \\  $\times [\alpha]_{q}[\alpha+1]_{q} $\end{tabular} & 0 & 0 & 0 & 0 & 0 & 0 & 0 & 0 & &\matindex{(4,3)} \\
     & 0 & 0 & 0 & 0 & 0 & 0 & 0 & 0 & \begin{tabular}{@{}c@{}} $-q^{-1}$ \\  $\times q^{\alpha/2 }$\end{tabular} & 0 & \begin{tabular}{@{}c@{}} $q^{3\alpha/2 }$ \\  $\times [\alpha]_{q}$\end{tabular} & 0 & 0 & 0 & 0 & 0 & &\matindex{(1,3)} \\
     & 0 & 0 & 0 & 0 & 0 & 0 & 0 & 0 & 0 & \begin{tabular}{@{}c@{}} $q q^{\alpha/2 }$\\  $\times [\alpha]_{q}$\end{tabular} & 0 & \begin{tabular}{@{}c@{}} $-q^{3\alpha/2 }[\alpha]_{q}$ \\  $\times [\alpha+1]_{q} $\end{tabular} & 0 & 0 & 0 & 0 & &\matindex{(3,1)} \\
     & 0 & 0 & 0 & 0 & 0 & 0 & 0 & 0 & 0 & 0 & \begin{tabular}{@{}c@{}} $-q^{-1/2 } q^{3\alpha/2 }$ \\  $\times [\alpha]_{q} $\end{tabular} & 0 & 0 & 0 & 0 & 0 & &\matindex{(2,4)} \\
     & 0 & 0 & 0 & 0 & 0 & 0 & 0 & 0 & 0 & 0 & 0 & \begin{tabular}{@{}c@{}} $q^{1/2 } q^{3\alpha/2 } $ \\  $\times [\alpha]_{q}[\alpha+1]_{q} $\end{tabular} & 0 & 0 & 0 & 0 & &\matindex{(4,2)} \\
     & 0 & 0 & 0 & 0 & 0 & 0 & 0 & 0 & 0 & 0 & 0 & 0 & $ -q^{-1} q^{\alpha}$ & 0 & 0 & 0 & &\matindex{(1,4)} \\
     & 0 & 0 & 0 & 0 & 0 & 0 & 0 & 0 & 0 & 0 & 0 & 0 & 0 & \begin{tabular}{@{}c@{}} $- q^{\alpha} [\alpha]_{q}$ \\  $\times [\alpha+1]_{q} $\end{tabular} & 0 & 0 & &\matindex{(4,1)} \\
     & 0 & 0 & 0 & 0 & 0 & 0 & 0 & 0 & 0 & 0 & 0 & 0 & 0 & 0 & \begin{tabular}{@{}c@{}} $q^{-1/2} q^{\alpha}$ \\  $\times [\alpha]_{q} $\end{tabular} & 0 & &\matindex{(2,3)} \\
     & 0 & 0 & 0 & 0 & 0 & 0 & 0 & 0 & 0 & 0 & 0 & 0 & 0 & 0 & 0 & \begin{tabular}{@{}c@{}} $q^{3/2} q^{\alpha}$ \\  $\times [\alpha]_{q} $\end{tabular} & &\matindex{(3,2)} \\
    \end{block}
  \end{blockarray}
}
\end{landscape}

which ends the proof since we are working over $\mathbb{C}(q,q^{\alpha})$.
\end{proof}

\begin{proof}[Proof of Theorem~\ref{thm:MAINRT}]
We need to prove that for any element $X \in U_{q}\mathfrak{gl}(2 \vert 1)^{\sigma}$ and any vector $w \in \mathcal{B}_{\alpha}$:
$$X.w = \displaystyle\sum_{x \in \mathcal{B}_{\alpha}}a_{x,X} x \text{ with } a_{x,X} \in \mathbb{C}(q) \text{ depending solely on } x \text{ and } X.$$
To achieve that goal, we only need to make sure that this is true for $X$ a \textit{generator} of $U_{q}\mathfrak{gl}(2 \vert 1)^{\sigma}$: $\sigma$, $q^{E_{1}^{1}}$, $q^{E_{2}^{2}}$, $q^{E_{3}^{3}}$, $E_{2}^{1}$, $E_{1}^{2}$, $E_{3}^{2}$ and $E_{2}^{3}$. Also, given the amount of computations required, we will not explicitly compute $X.w$ for all $w \in \mathcal{B}_{\alpha}$, but only for $w=v_1, v_2, v_3, v_4,$ and $ v_5$. One can make sure that similar calculations can be carried out for all other basis vectors.

For these computations, we refer the reader to the relations defining algebra $U_{q}\mathfrak{gl}(2 \vert 1)^{\sigma}$, that are the same as those defining superalgebra $U_{q}\mathfrak{gl}(2 \vert 1)$. They are all summed up nicely in \cite{DW}, Sections 4.2.3 and 4.2.4, and will be the main ingredients that make things work. 

\begin{itemize}
    \item \underline{X = $\sigma$}. 
    Given that $\sigma$ acts diagonally on $V_{\alpha} \otimes V_{\alpha}^{\ast}$ with eigenvalues $1$, $-1$ and that $\sigma$ commutes up to a sign with any element in $U_{q}\mathfrak{gl}(2 \vert 1)^{\sigma}$: $\sigma. w = \pm w$ for any basis element $w$.
    \item \underline{X = $q^{E_{i}^{i}}$}, $i=1, 2, 3$. \\
    $q^{E_{1}^{1}}.v = v = v_1$ ; \\ 
    $q^{E_{1}^{1}}E_{3}^{2}.v = E_{3}^{2}q^{E_{1}^{1}}.v = E_{3}^{2}.v = v_2$ ; \\
    $q^{E_{1}^{1}}E_{2}^{3}.v = E_{2}^{3}q^{E_{1}^{1}}.v = E_{2}^{3}.v = v_5$ ; \\
    $q^{E_{1}^{1}}E_{3}^{1}.v = q^{E_{1}^{1}}(E_{2}^{1} E_{3}^{2} - q^{-1} E_{3}^{2} E_{2}^{1}).v = (q E_{2}^{1} q^{E_{1}^{1}} E_{3}^{2} - q^{-1} E_{3}^{2} q^{E_{1}^{1}} E_{2}^{1}).v = q (E_{2}^{1} E_{3}^{2} - q^{-1} E_{3}^{2} E_{2}^{1}) q^{E_{1}^{1}}.v $  \\
    $\textcolor{white}{q^{E_{1}^{1}}E_{2}^{3}.v} = q E_{3}^{1} q^{E_{1}^{1}}.v = q E_{3}^{1}.v = q v_3$.\\
    The same goes for $q^{E_{1}^{1}}E_{1}^{3}.v$. And the same type of computations can be carried out for $q^{E_{2}^{2}}$ and $q^{E_{3}^{3}}$. 
    \item \underline{X = $E_{2}^{1}$}. \\
    $E_{2}^{1}.v = 0$ ; \\
    $E_{2}^{1}E_{3}^{2}.v = E_{3}^{1}.v + q^{-1} E_{3}^{2}E_{2}^{1}.v = E_{3}^{1}.v + q^{-1} E_{3}^{2}.0 = E_{3}^{1}.v = v_3$ ; \\
    $E_{2}^{1}E_{3}^{1}.v = q E_{3}^{1}E_{2}^{1}.v = q E_{3}^{1}.0 = 0$ ; \\
    $E_{2}^{1}E_{1}^{3}.v = (E_{2}^{1}E_{2}^{3}E_{1}^{2}-q E_{2}^{1}E_{1}^{2}E_{2}^{3}).v = 0 - q E_{2}^{1}E_{1}^{2}E_{2}^{3}.v = -q [E_{1}^{1} - E_{2}^{2}]_{q}E_{2}^{3}.v - q E_{1}^{2}E_{2}^{1}E_{2}^{3}.v$  \\
    $\textcolor{white}{E_{2}^{1}E_{1}^{3}.v} = -q [E_{1}^{1} - E_{2}^{2}]_{q}E_{2}^{3}.v - q E_{1}^{2}E_{2}^{3}E_{2}^{1}.v = -q [E_{1}^{1} - E_{2}^{2}]_{q}E_{2}^{3}.v - 0$\\
    $\textcolor{white}{E_{2}^{1}E_{1}^{3}.v} =-q \frac{q^{E_{1}^{1}- E_{2}^{2}}- q^{E_{2}^{2}- E_{1}^{1}}}{q-q^{-1}} E_{2}^{3}.v = \frac{q}{q^{-1}-q} q^{E_{1}^{1}- E_{2}^{2}} E_{2}^{3}.v + \frac{q}{q-q^{-1}} q^{E_{2}^{2}- E_{1}^{1}} E_{2}^{3}.v$\\
    $\textcolor{white}{E_{2}^{1}E_{1}^{3}.v}  = \frac{q}{q^{-1}-q} ( q E_{2}^{3}q^{E_{1}^{1}- E_{2}^{2}}).v + \frac{q}{q-q^{-1}} ( q^{-1} E_{2}^{3} q^{E_{2}^{2}- E_{1}^{1}}).v = \frac{q^2}{q^{-1}-q} E_{2}^{3}(q^{E_{1}^{1}- E_{2}^{2}}.v) + \frac{1}{q-q^{-1}} E_{2}^{3} (q^{E_{2}^{2}- E_{1}^{1}}.v)$\\
    $\textcolor{white}{E_{2}^{1}E_{1}^{3}.v}  = \frac{q^2}{q^{-1}-q} E_{2}^{3}.v + \frac{1}{q-q^{-1}} E_{2}^{3}.v = -q E_{2}^{3}.v = -q v_5$ ;\\
    $E_{2}^{1}E_{2}^{3}.v = E_{2}^{3}E_{2}^{1}.v = 0.$
    \item \underline{X = $E_{1}^{2}$}. \\
    $E_{1}^{2}.v = 0$ ; \\
    $E_{1}^{2} E_{3}^{2}.v = E_{3}^{2} E_{1}^{2}.v = 0$ ; \\
    $E_{1}^{2} E_{3}^{1}.v = E_{1}^{2}E_{2}^{1}E_{3}^{2}.v -q^{-1} E_{1}^{2}E_{3}^{2}E_{2}^{1}.v = E_{1}^{2}E_{2}^{1}E_{3}^{2}.v - 0 = (E_{2}^{1}E_{1}^{2} - [E_{1}^{1} - E_{2}^{2}]_{q})E_{3}^{2}.v$\\
    $\textcolor{white}{E_{1}^{2} E_{3}^{1}.v}  =E_{2}^{1}E_{1}^{2}E_{3}^{2}.v - [E_{1}^{1} - E_{2}^{2}]_{q}E_{3}^{2}.v = E_{2}^{1}E_{3}^{2}E_{1}^{2}.v - [E_{1}^{1} - E_{2}^{2}]_{q}E_{3}^{2}.v = 0 - [E_{1}^{1} - E_{2}^{2}]_{q}E_{3}^{2}.v $\\
    $\textcolor{white}{E_{1}^{2} E_{3}^{1}.v}  = \frac{q^{E_{2}^{2}- E_{1}^{1}}- q^{E_{1}^{1}- E_{2}^{2}}}{q-q^{-1}} E_{3}^{2}.v = \frac{1}{q-q^{-1}} E_{3}^{2}(q q^{E_{2}^{2}- E_{1}^{1}} - q^{-1} q^{E_{1}^{1}- E_{2}^{2}}).v = \frac{1}{q-q^{-1}} E_{3}^{2}(q.v - q^{-1}.v) = E_{3}^{2}.v = v_2$ ;\\
    $E_{1}^{2} E_{1}^{3}.v = q E_{1}^{3} E_{1}^{2}.v = 0$ ;\\
    $E_{1}^{2} E_{2}^{3}.v = q^{-1} E_{2}^{3} E_{1}^{2}.v - q^{-1} E_{1}^{3}.v = 0 - q^{-1} E_{1}^{3}.v = - q^{-1}v_4$.
    \item \underline{X = $E_{3}^{2}$}.\\
    $E_{3}^{2}.v = v_2$ ; \\
    $E_{3}^{2}E_{3}^{2}.v = 0.v = 0$ ; \\
    $E_{3}^{2}E_{3}^{1}.v = v_6$ ; \\
    $E_{3}^{2}E_{1}^{3}.v = v_7$ ; \\
    $E_{3}^{2}E_{2}^{3}.v = v_8$.
    \item \underline{X = $E_{2}^{3}$}.\\
    $E_{2}^{3}.v = v_5$ ; \\
    $E_{2}^{3}E_{3}^{2}.v = \frac{q^{E_{2}^{2}+ E_{3}^{3}}- q^{-E_{2}^{2}- E_{3}^{3}}}{q-q^{-1}}.v - E_{3}^{2}E_{2}^{3}.v = \frac{1}{q-q^{-1}}(v-v) - v_8 = -v_8$ ; \\
    $E_{2}^{3}E_{3}^{1}.v = E_{2}^{3}E_{2}^{1}E_{3}^{2}.v - q^{-1} E_{2}^{3}E_{3}^{2}E_{2}^{1}.v = (E_{2}^{3}E_{2}^{1})E_{3}^{2}.v - q^{-1} E_{2}^{3}E_{3}^{2}.0$  \\
    $\textcolor{white}{E_{2}^{3}E_{3}^{1}.v} = (E_{2}^{1}E_{2}^{3})E_{3}^{2}.v = E_{2}^{1}(E_{2}^{3}E_{3}^{2}).v = E_{2}^{1}(-E_{3}^{2}E_{2}^{3}).v = -E_{2}^{1}E_{3}^{2}E_{2}^{3}.v$ \\
    $\textcolor{white}{E_{2}^{3}E_{3}^{1}.v} = (-E_{3}^{1} - q^{-1} E_{3}^{2}E_{2}^{1})E_{2}^{3}.v = -E_{3}^{1}E_{2}^{3}.v - q^{-1} E_{3}^{2}(E_{2}^{1}E_{2}^{3}).v = -v_{10} - q^{-1} E_{3}^{2}(E_{2}^{3}E_{2}^{1}).v $ \\
    $\textcolor{white}{E_{2}^{3}E_{3}^{1}.v} = -v_{10} - q^{-1} E_{3}^{2}E_{2}^{3}.0 = -v_{10}$ ; \\
    $E_{2}^{3}E_{2}^{3}.v = 0.v = 0$ ;\\
    $E_{2}^{3}E_{1}^{3}.v = E_{2}^{3}(E_{2}^{3}E_{1}^{2} - q E_{1}^{2}E_{2}^{3}).v  = E_{2}^{3}E_{2}^{3}E_{1}^{2}.v - q E_{2}^{3}E_{1}^{2}E_{2}^{3}.v  = 0 E_{1}^{2}.v - q (E_{2}^{3}E_{1}^{2})E_{2}^{3}.v$\\
    $\textcolor{white}{E_{2}^{3}E_{1}^{3}.v} = -q (E_{1}^{3} + q E_{1}^{2}E_{2}^{3})E_{2}^{3}.v = -q E_{1}^{3} E_{2}^{3}.v - q^2 E_{1}^{2}E_{2}^{3}E_{2}^{3}.v = -q E_{1}^{3} E_{2}^{3}.v - q^2 E_{1}^{2}0.v $ \\
    $\textcolor{white}{E_{2}^{3}E_{1}^{3}.v} = -q E_{1}^{3} E_{2}^{3}.v  = -q v_{11}$.
\end{itemize}

\end{proof}

\section{Generalizing the genus bound to certain classes of links}\label{sec:Appendix}

We wish to use the techniques developed here to extend our results on the degree of the $LG$ polynomial to links. In particular, we would like to understand how we can generalize the picture shown in Figure~\ref{fig:FIG3} to links. The natural generalization is a link $L$ with an oriented Seifert surface $\Sigma$ with $k$ disks, $l$ strips and $m$ boundary components (the components of the link) that is embedded in $S^3$ like in Figure~\ref{fig:FIG9}.

\begin{figure}
\adjustbox{trim=0cm 2cm 0cm 1cm}{
  \includegraphics[scale=.4]{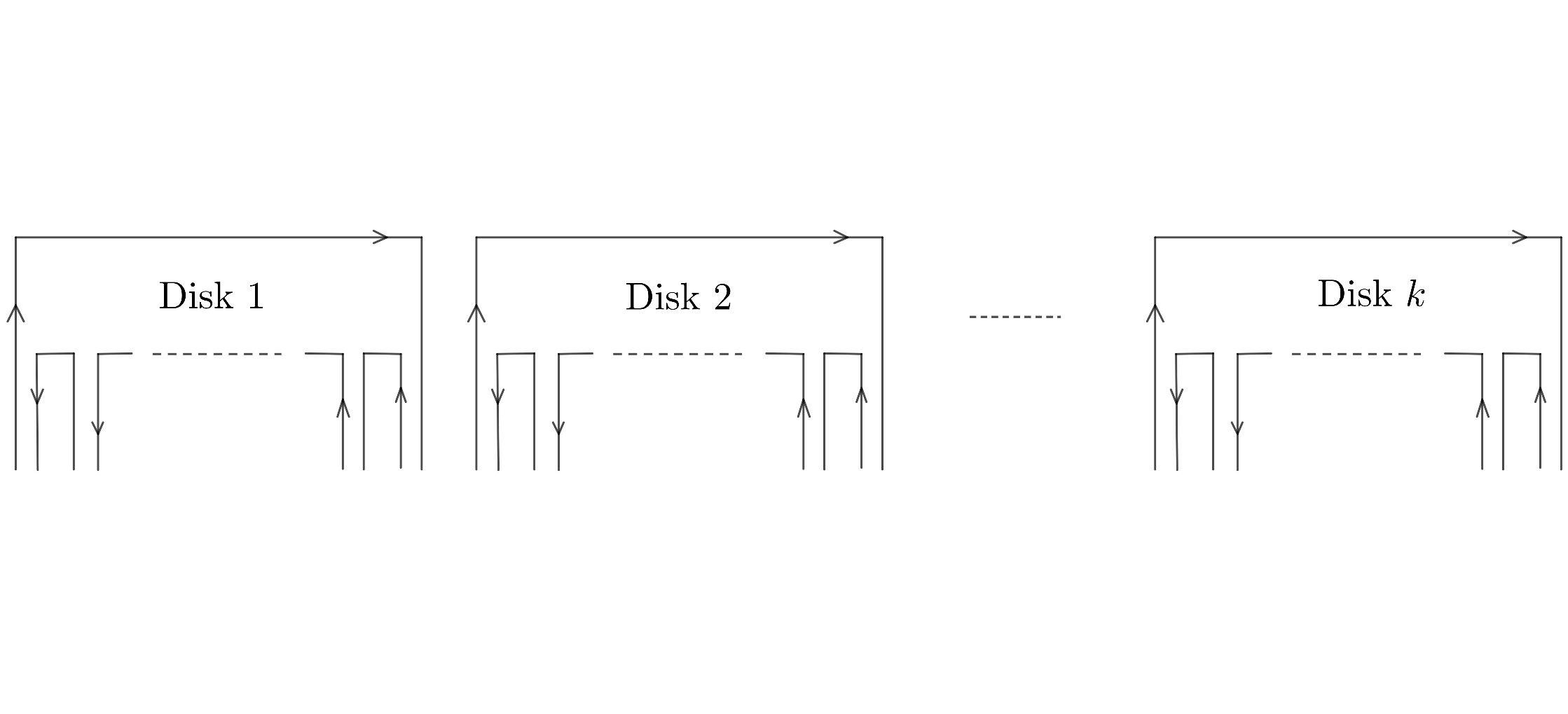}
}
\caption{An admissible diagram for a link $L$ embedded in $S^3$ has an obvious oriented Seifert surface, here with $k$ disks.}\label{fig:FIG9}
\end{figure}

\begin{defn}
    An \textit{admissible} diagram $D$ for a link $L$ is a diagram that presents $L$ as the boundary of an oriented surface obtained by attaching bands to disks, such that all disks lie in one half-plane while the bands connecting them lie in the other half-plane. Such a diagram is illustrated in Figure~\ref{fig:FIG9}.
\end{defn}

In the specific case where each strip connects a disk to itself in a way that each connected component of the surface has only one boundary component (see Figure~\ref{fig:FIG10}), allowing however that strips from different disks be entangled, the link $L$ that bounds the surface is a boundary link in the sense of Definition~~\ref{def:boundary}.

\begin{figure}
\begin{center}
\adjustbox{trim=0cm 1cm 0cm .5cm}{
  \includegraphics[scale=.4]{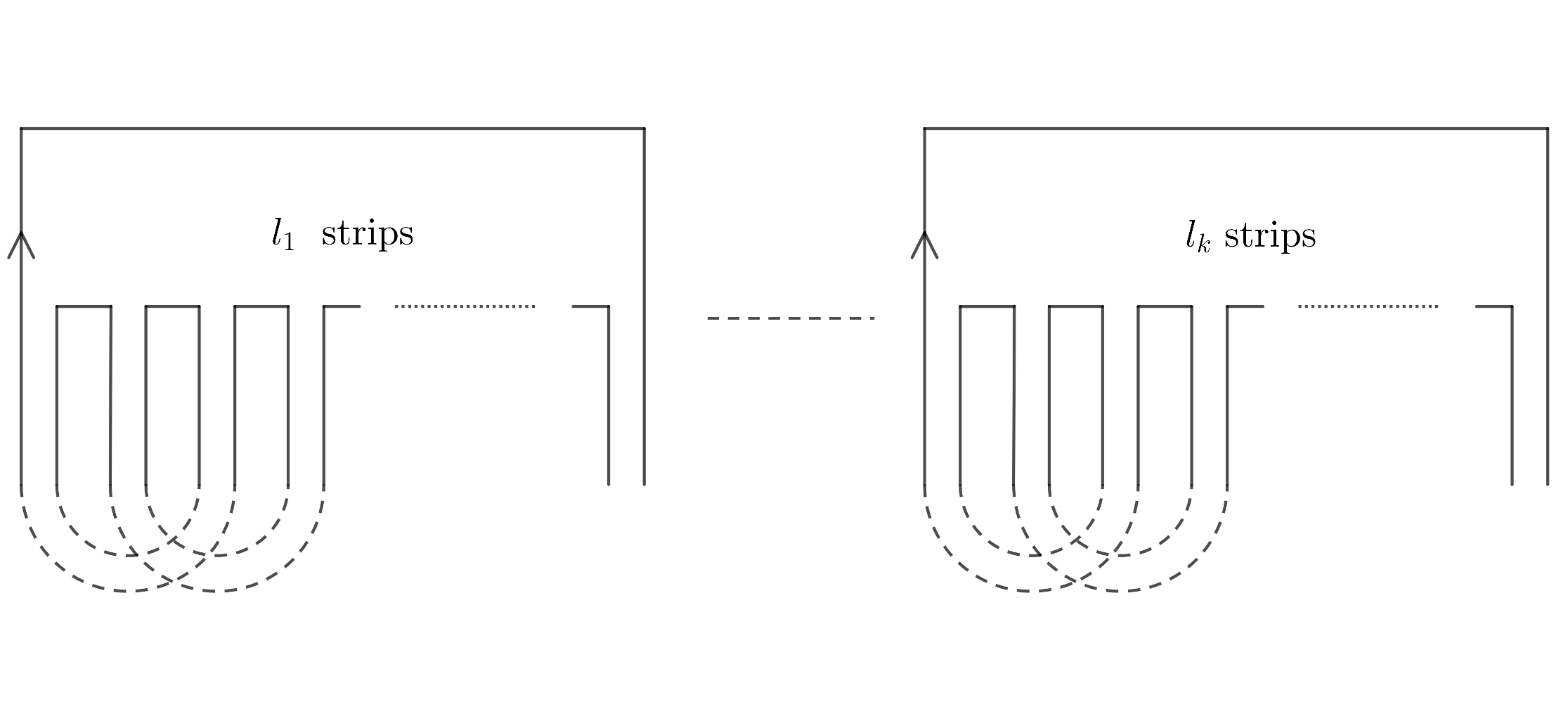}
}
\caption{This link drawing a Seifert surface with $k$ disks, $l = l_1 + l_2 + \ldots + l_k$ strips and $m = k$ boundary components is a boundary link.}\label{fig:FIG10}
\end{center}
\end{figure} 

\begin{defn}[Boundary link]\label{def:boundary}
    Set $L = K_1 \sqcup K_2 \sqcup \ldots \sqcup K_n \subset S^3$ a link with $n$ components. It is called a \textit{boundary link} if there exists an orientable surface $S = S_1 \sqcup S_2 \sqcup \ldots \sqcup S_n$ embedded in $S^3$ such that for all $i$, $\partial S_i = K_i$. In other words, a boundary link is a link that can be realized as the boundary of an orientable surface embedded in $S^3$, where each component of the link corresponds to the boundary of a connected component of the surface.
\end{defn}

Moreover, any link $L$ can be embedded in $S^3$ like in Figure~\ref{fig:FIG9}. In other words: 

\begin{prop}\label{existence}
    Any link has an admissible diagram.
\end{prop}

A proof for the case of a knot can be found in \cite{DW}, Chapter 8, Section B. It also works for links.

\begin{proof}[Proof of Proposition \ref{existence}]
   Any compact oriented surface with non-empty boundary has a handle decomposition consisting of $0$-handles and $1$-handles only. It can be assumed that each connected component has exactly one zero handle. An oriented spanning surface of a link is an embedding of a compact surface whose boundary is the link. The embedding transfers the handle decomposition of the compact surface to $S^3$. Contract each surface to the cores of the $1$-handles. Now you have a several $n$-leaf roses embedded in $\mathbb{R}^3$, where $n$ varies for each rose. For the same reason that every link has a planar projection (i.e. a local homeomorphism into some plane), so also do the union of $n$-leaf roses. Thickening this a bit, placing all the $0$-handles on a straight line, and twisting the bands according to necessary framing results in a picture similar to Figure~\ref{fig:FIG9}. 
\end{proof}

\begin{rmk}
    In particular, the number of disks can be chosen to be the number of components of the spanning surface.
\end{rmk}

\begin{prop}\label{links}
    Let $L \subset S^3$ be a link, and let $l$ be the number of strips in an admissible diagram of $L$. Then:
    $$
    span_{(t_0, t_1)}(LG(L, t_0, t_1)) = span_{q^{2 \alpha}}(LG(L,q,q^{\alpha})) \leq 2 \times l.
    $$
\end{prop}

\begin{proof}
    The diagram of a (1-1)-tangle associated to the admissible diagram for $L$ is represented in Figure~\ref{fig:FIG11}, where $\Omega^+$ is the positive cap, $\Omega^-$ is the negative cap, $\mho^+$ is the positive cup and $\mho^-$ is the negative cup. We have $\Omega^{+} = \mho^{+} = I_{4}$ while

\begin{figure}
\includegraphics[scale=.7]{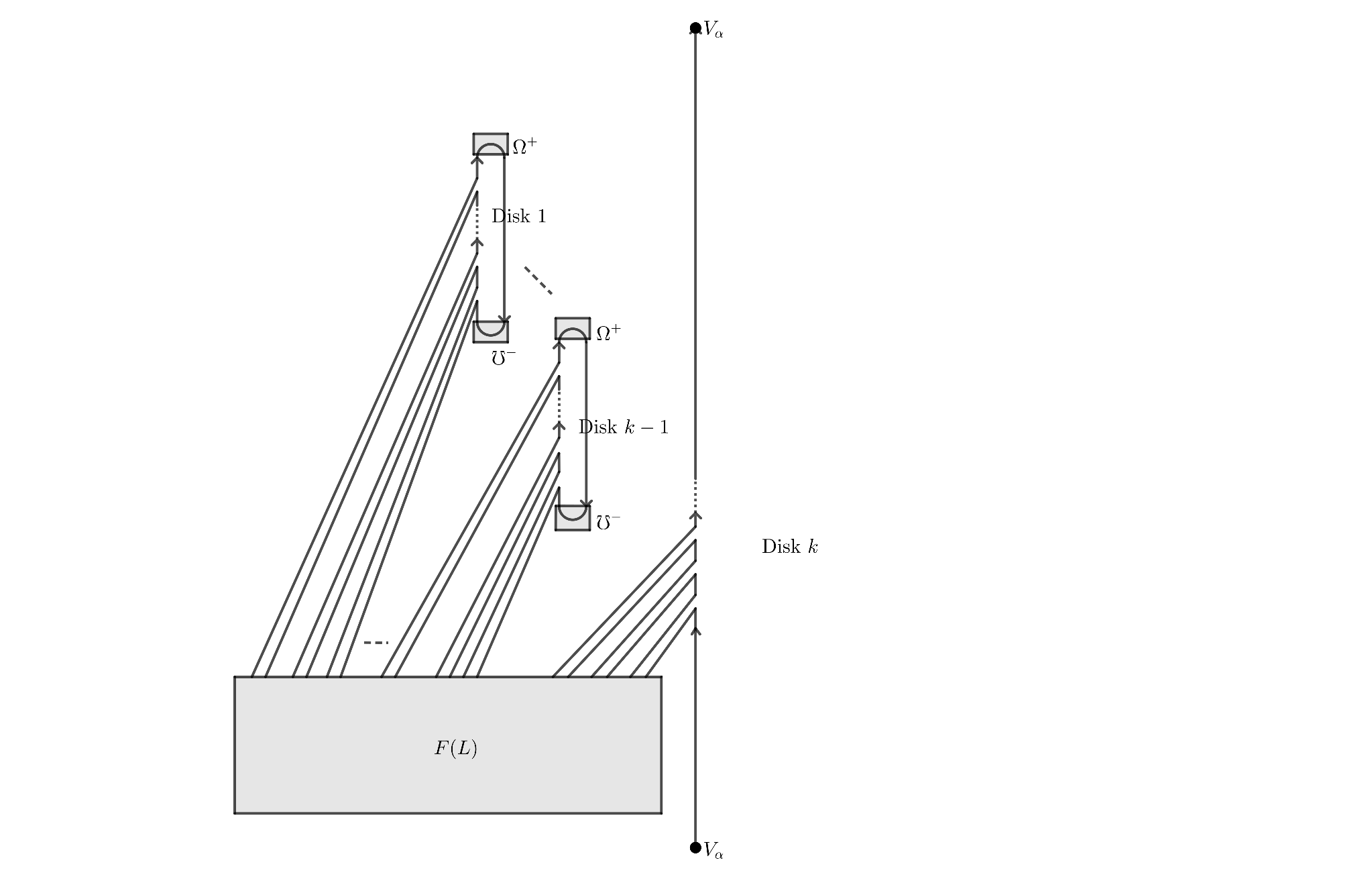}
\caption{A (1-1)-tangle associated to an admissible diagram for a link $L$ with $k$ disks and $l$ strips.}\label{fig:FIG11}
\end{figure} 

$$
\Omega^-=
\begin{pmatrix}
q^{2\alpha} & 0 & 0 & 0 \\
0 & -q^{2\alpha} & 0 & 0 \\
0 & 0 & -q^{2(\alpha+1)} & 0 \\
0 & 0 & 0 & q^{2(\alpha+1)}
\end{pmatrix} \text{   is replaced by  }
\begin{pmatrix}
1 & 0 & 0 & 0 \\
0 & -1 & 0 & 0 \\
0 & 0 & -q^{2} & 0 \\
0 & 0 & 0 & q^{2}
\end{pmatrix}$$
$$
\text{  and   }
\mho^-=
\begin{pmatrix}
q^{-2\alpha} & 0 & 0 & 0 \\
0 & -q^{-2\alpha} & 0 & 0 \\
0 & 0 & -q^{-2(\alpha+1)} & 0 \\
0 & 0 & 0 & q^{-2(\alpha+1)}
\end{pmatrix} \text{   is replaced by  }
\begin{pmatrix}
1 & 0 & 0 & 0 \\
0 & -1 & 0 & 0 \\
0 & 0 & -q^{-2} & 0 \\
0 & 0 & 0 & q^{-2}
\end{pmatrix}.
$$
\newline

Recall that following Definition~\ref{degree}, if $z =q^{\alpha}$ and $t = q^{-\alpha}$, then for $P(q,q^{\alpha}) \in \mathbb{C}(q,q^{\alpha})$, we defined $deg_z(P(q,q^{\alpha}))$ as the degree of the dominant coefficient in the expansion of $P(q,q^{\alpha})$ as a Laurent series in variable $z$, if such an expansion exists. Likewise we also introduced $deg_t(P(q,q^{\alpha}))$. 

Figure~\ref{fig:FIG11} allows us to compute the corresponding operator invariant that is represented by the diagram on Figure~\ref{fig:FIG12} using the ideas we developed in the previous section. 
\newline

\underline{Step 1:} Let us prove that 
$
deg_{z}(LG(L,q,q^{\alpha})) \leq 4 \times l,
$
where $l$ is the number of strips on the admissible diagram. Indeed, a computation similar to the one we carried out in the proof of Proposition~\ref{uperbound} shows that 
$$
deg_{z}(LG(L,q,q^{\alpha})) \leq 4 \times \text{the number of $b$ boxes in Figure~\ref{fig:FIG12}}.
$$
But there is exactly one box $b$ for each strip, so there are $l$ boxes $b$ in Figure~\ref{fig:FIG12}, which proves the inequality we announced for the $z$-degree.
\newline

\begin{figure}
\adjustbox{trim=4cm 0cm 0cm 0cm}{
  \includegraphics[scale=.7]{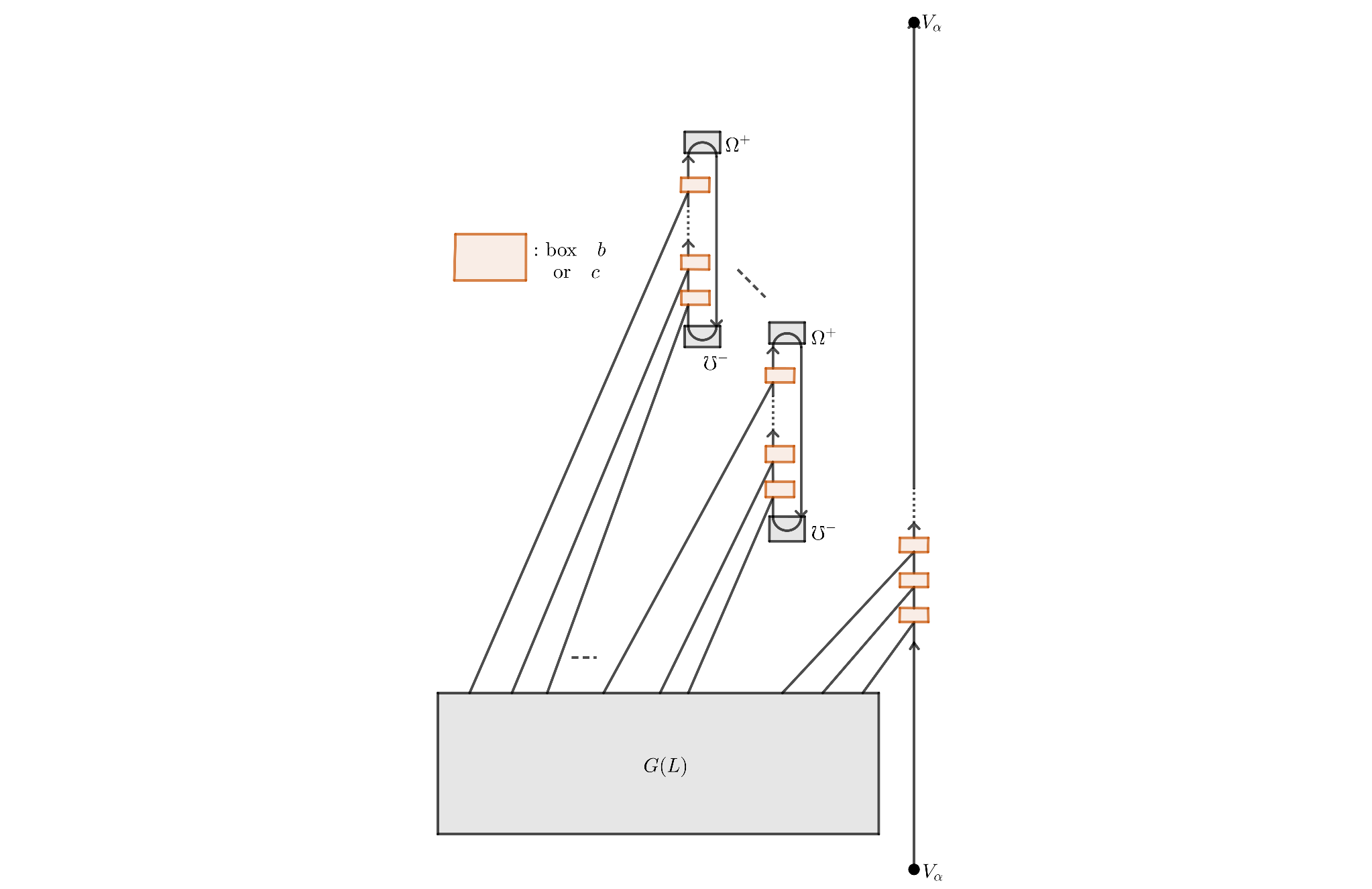}
}
\caption{The operator invariant diagram represented here is associated to the (1-1)-tangle represented in Figure~\ref{fig:FIG11} and has the same number of boxes $b$ and of boxes $c$.}\label{fig:FIG12}
\end{figure} 

\underline{Step 2:} Now we show that 
$
deg_{t}(LG(L,q,q^{\alpha})) \leq 0.
$ To prove this, we show that if $W^{(*) \otimes j}$ stands for a tensor product of $j$ vector spaces, each of them being $W$ or $W^*$, and if we define $\Phi$ graphically following Figure~\ref{fig:FIG13}, then for any basis vector $v_{i_{1}}^{(\ast)} \otimes v_{i_{2}}^{(\ast)} \otimes \ldots \otimes v_{i_{j}}^{(\ast)} \in \mathcal{B}_{\alpha}^{(*) \otimes j}$ we have:
$$
\Phi(v_{i_{1}}^{(\ast)} \otimes v_{i_{2}}^{(\ast)} \otimes \ldots \otimes v_{i_{j}}^{(\ast)}) \in \mathbb{C}(q,q^{\alpha}) \text{ and } deg_{t}(\Phi(v_{i_{1}}^{(\ast)} \otimes v_{i_{2}}^{(\ast)} \otimes \ldots \otimes v_{i_{j}}^{(\ast)})) \leq 0. 
$$
Indeed, still referring to Figure~\ref{fig:FIG13} we write $\Phi = \Omega^+ \circ \Psi \circ (id_{W^{(*)}} \otimes \ldots \otimes id_{W^{(*)}}  \otimes \mho^-)$. So:
\begin{figure}
\adjustbox{trim=3.5cm 0cm 0cm 0cm}{
  \includegraphics[scale=.7]{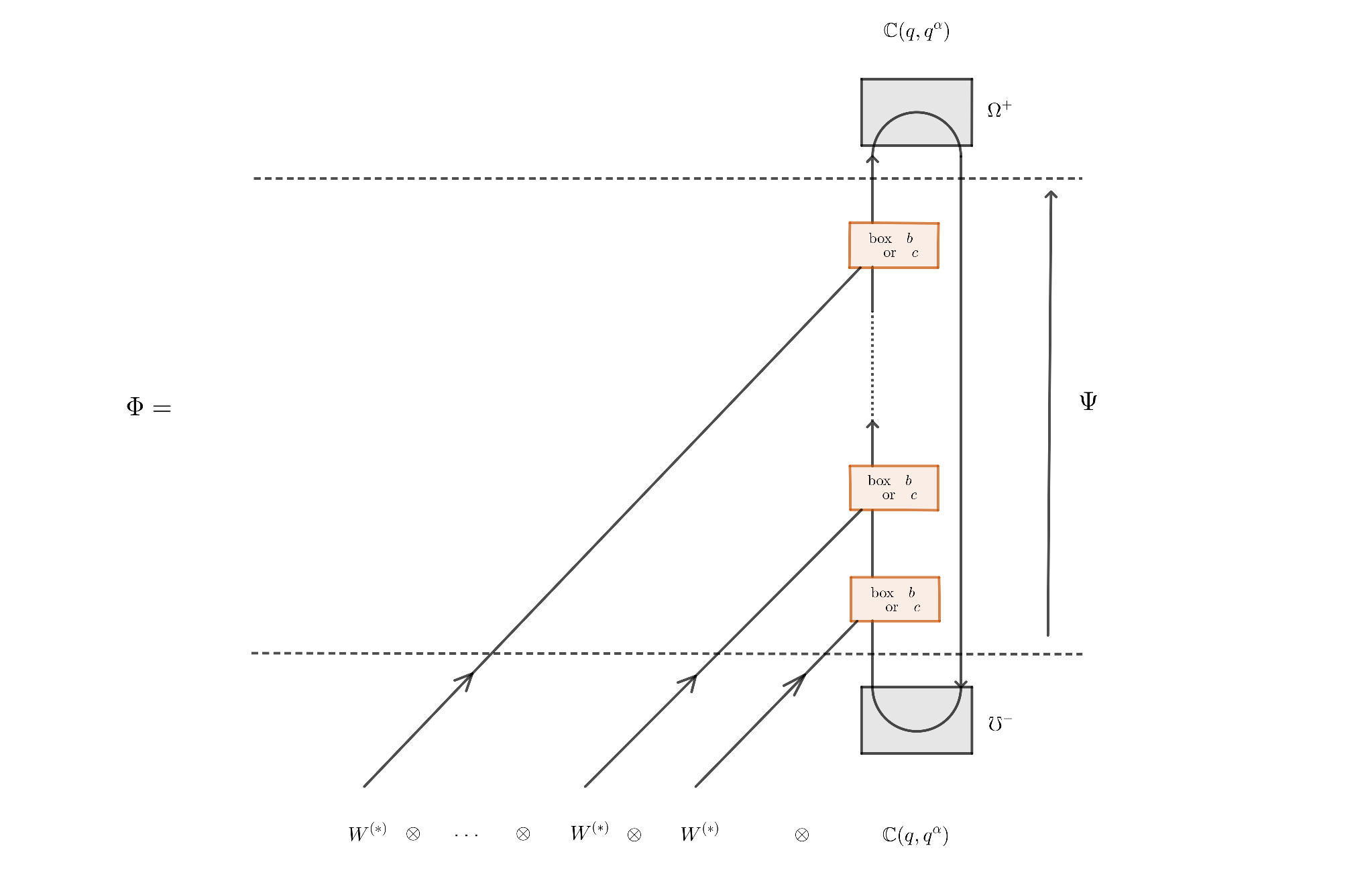}
}
\caption{Map $\Phi$.}\label{fig:FIG13}
\end{figure}

\begin{align*}
   &\Phi(v_{i_{1}}^{(\ast)} \otimes v_{i_{2}}^{(\ast)} \otimes \ldots \otimes v_{i_{j}}^{(\ast)}) = \Omega^+ \circ \Psi \circ (id_{W^{(*)}} \otimes \ldots \otimes id_{W^{(*)}}  \otimes \mho^-)(v_{i_{1}}^{(\ast)} \otimes v_{i_{2}}^{(\ast)} \otimes \ldots \otimes v_{i_{j}}^{(\ast)}\otimes 1) \\ 
   &= \Omega^+ \circ \Psi (v_{i_{1}}^{(\ast)} \otimes v_{i_{2}}^{(\ast)} \otimes \ldots \otimes v_{i_{j}}^{(\ast)}\otimes [ e_1 \otimes e_1^{\ast} - e_2 \otimes e_2^{\ast} - q^{-2} e_3 \otimes e_3^{\ast} + q^{-2} e_4 \otimes e_4^{\ast}]).
\end{align*}

Considering only $\Omega^+ \circ \Psi (v_{i_{1}}^{(\ast)} \otimes v_{i_{2}}^{(\ast)} \otimes \ldots \otimes v_{i_{j}}^{(\ast)}\otimes e_1 \otimes e_1^{\ast})$ - the other three parts of the computation can be taken care of similarly - we get:

\begin{align*}
   &\Omega^+ \circ \Psi (v_{i_{1}}^{(\ast)} \otimes v_{i_{2}}^{(\ast)} \otimes \ldots \otimes v_{i_{j}}^{(\ast)}\otimes e_1 \otimes e_1^{\ast}) \\ 
   &= \Omega^+ ([a_1(t) e_1 + a_2(t) e_2 + a_3(t) e_3 + a_4(t) e_4] \otimes e_1^{\ast}),
\end{align*}
where $deg_t(a_1(t)) = 0$, $deg_t(a_2(t)) = \frac{1}{2}$, $deg_t(a_3(t)) = \frac{1}{2}$ and $deg_t(a_4(t)) = 1$ because the degrees of these functions are functions \textit{only} of the vector $e_1$, $e_2$, $e_3$ or $e_4$ they are placed in front of, as explained in the proof of Proposition~\ref{lowerbound}. Thus:

\begin{align*}
   &\Omega^+ \circ \Psi (v_{i_{1}}^{(\ast)} \otimes v_{i_{2}}^{(\ast)} \otimes \ldots \otimes v_{i_{j}}^{(\ast)}\otimes e_1 \otimes e_1^{\ast}) \\ 
   &= a_1(t)~\Omega^+(e_1 \otimes e_1^{\ast}) + a_2(t)~\Omega^+(e_2 \otimes e_2^{\ast}) + a_3(t)~\Omega^+(e_3 \otimes e_3^{\ast}) + a_4(t)~\Omega^+(e_4 \otimes e_4^{\ast}) \\
   &= a_1(t)\times 1 + a_2(t)\times 0 + a_3(t)\times 0 + a_4(t) \times 0 \\
   &= a_1(t)~ \text{(which is of degree $0$ in $t$).}
\end{align*}

Collecting all terms of $\Phi(v_{i_{1}}^{(\ast)} \otimes v_{i_{2}}^{(\ast)} \otimes \ldots \otimes v_{i_{j}}^{(\ast)})$ we find that $deg_{t}(\Phi(v_{i_{1}}^{(\ast)} \otimes v_{i_{2}}^{(\ast)} \otimes \ldots \otimes v_{i_{j}}^{(\ast)})) \leq 0$. 
\newline

This proves that the computation of $deg_{t}(LG(L,q,q^{\alpha}))$ using the diagram represented in Figure~\ref{fig:FIG12} works the same way as for a knot using Figure~\ref{fig:FIG7}, which we did in detail in Section~\ref{sec:Topology}. This means that
$$
deg_{t}(LG(L,q,q^{\alpha})) \leq 0.
$$

Combining the two previous steps nicely provides the result we wished to prove here.
\end{proof}

\begin{cor}[Genus bound for $(n,0)$-cables]
    Set $K$ an oriented knot and $D(K)$ its \textit{preferred diagram} with $2 \times g(K)$ strips represented in Figure~\ref{fig:FIG3}. We refer to the $(n,0)$-cable of $K$ drawn on $D(K)$ through \textit{blackboard framing} as $K^{\# n}$, see Figure~\ref{fig:FIG14}. Link $K^{\# n}$ is a boundary link that has an admissible diagram with $2 \times n \times g(K)$ strips, and we have
    $$
    span_{(t_0, t_1)}(LG(K^{\# n}, t_0, t_1)) = span_{q^{2 \alpha}}(LG(K^{\# n},q,q^{\alpha})) \leq 4 \times n \times g(K)
    .$$
\end{cor}

\begin{figure}
\adjustbox{trim=0cm 4cm 0cm 1cm}{
  \includegraphics[scale=.7, angle =270 ]{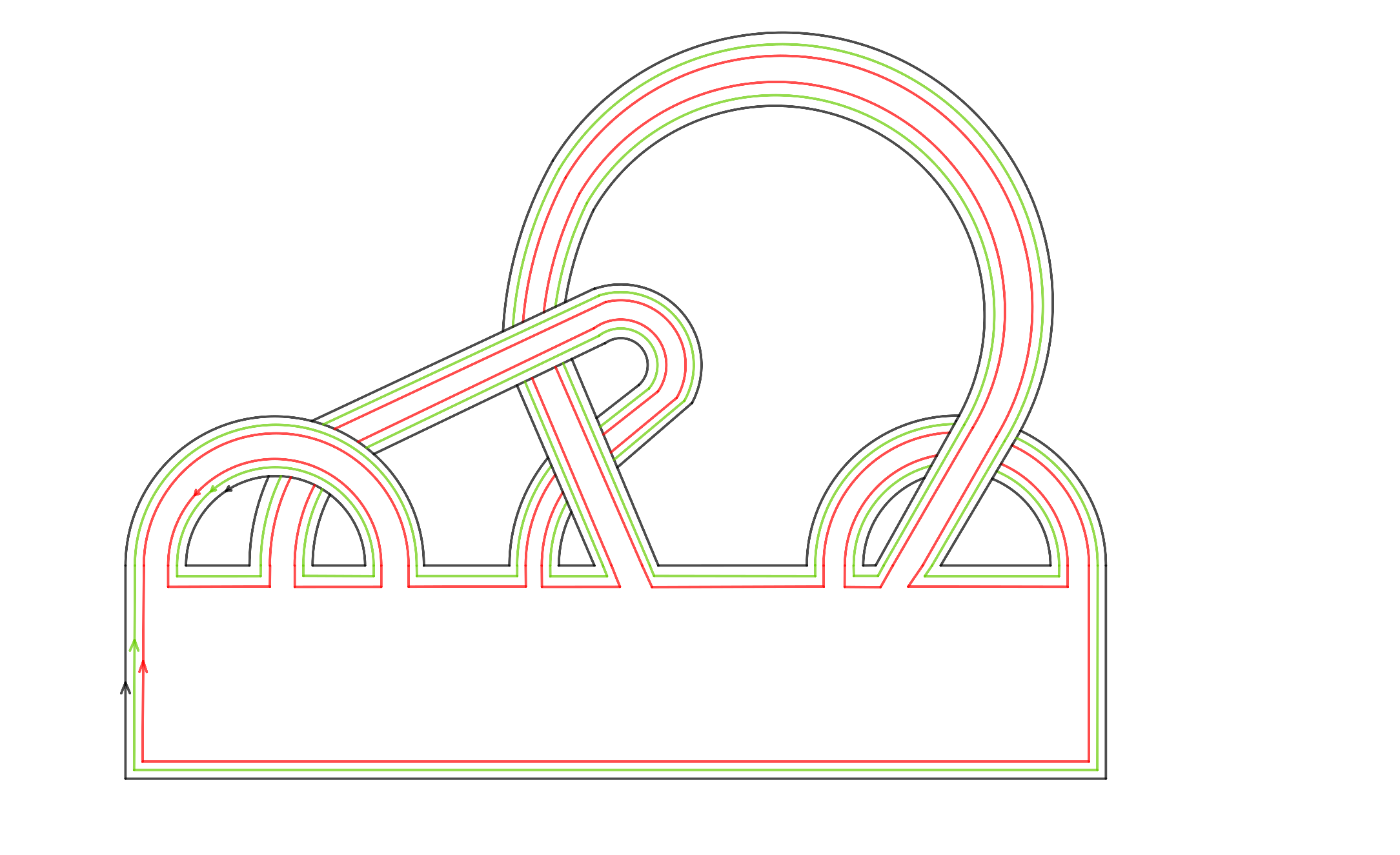}
  }
\caption{Step 1 of the isotopy: $K^{\# 3}$ is a boundary link.}\label{fig:FIG14}
\end{figure}

\begin{figure}
\adjustbox{trim=0cm 2.5cm 0cm 1.5cm}{
  \includegraphics[scale=.7, angle =270 ]{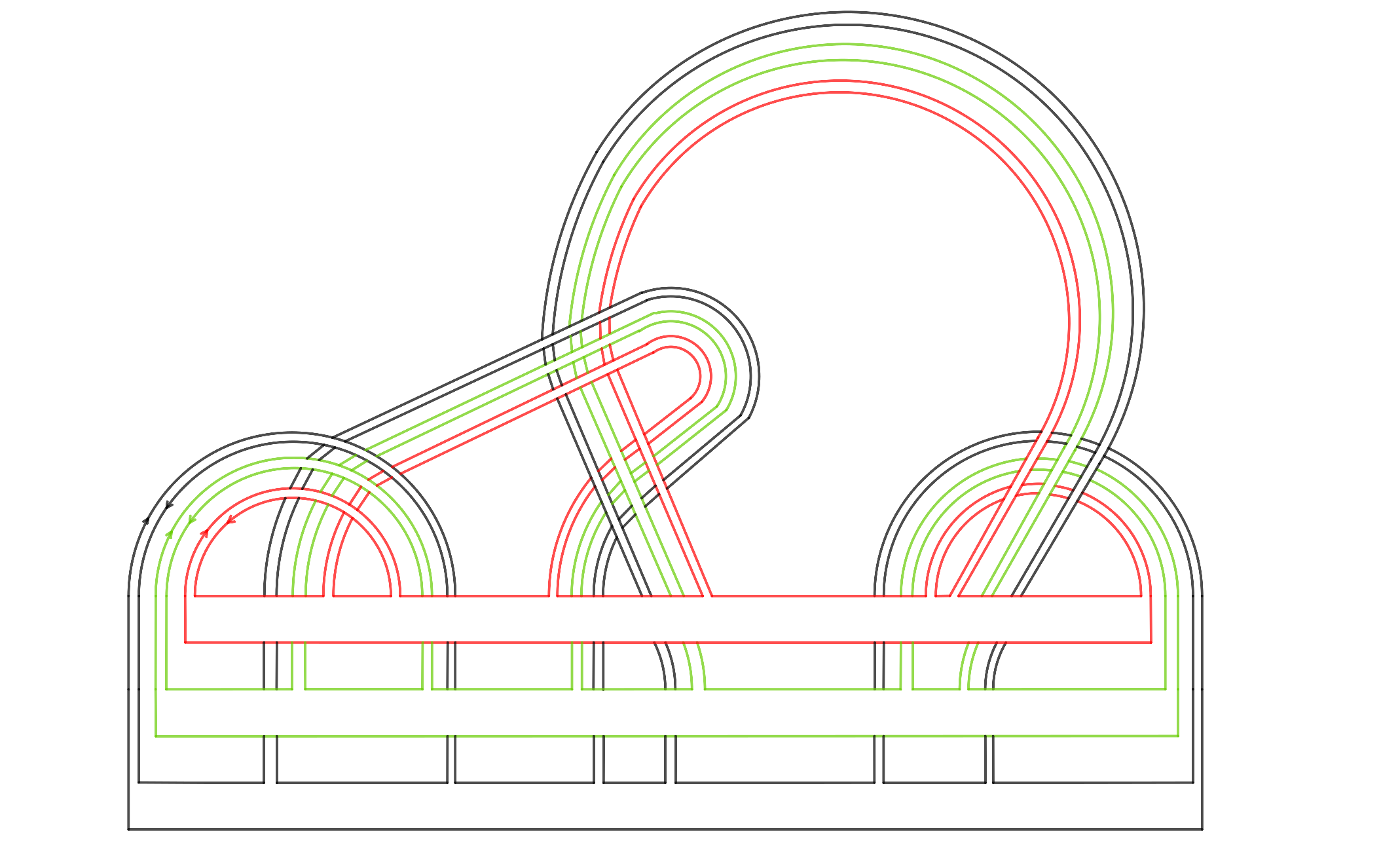}
  }
\caption{Step 2 of the isotopy: this surface is isotopic to the one depicted in Figure~\ref{fig:FIG14}.}\label{fig:FIG15}
\end{figure}

\begin{figure}
\adjustbox{trim=1.5cm 4.5cm 0cm 2.5cm}{
  \includegraphics[scale=.8, angle =270 ]{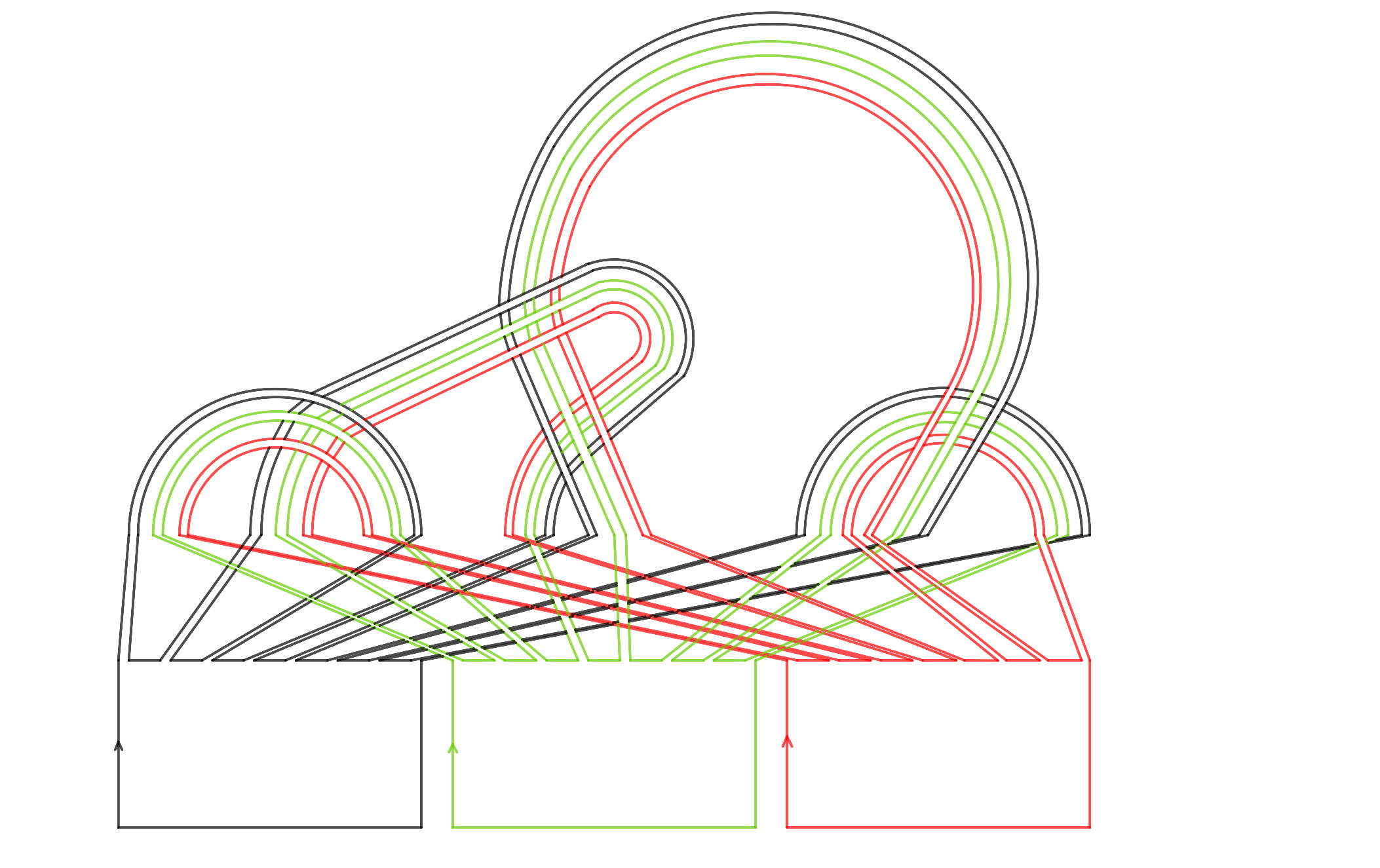}
  }
\caption{Step 3 of the isotopy: an admissible diagram for $K^{\# 3}$.}\label{fig:FIG16}
\end{figure}

\begin{proof}
To prove that $K^{\# n}$ is a boundary link that has an \textit{admissible diagram} with $2 \times n \times g(K)$ strips, we refer the reader to Figures~\ref{fig:FIG14}, \ref{fig:FIG15}, and \ref{fig:FIG16}, where we illustrate this by a series of isotopic surfaces in the case where $n = 3$ and the preferred diagram $D(K)$ has $4$ strips. We are confident that this nicely illustrates the general case that is identical.

Now let us use Proposition~\ref{links} to prove that the span of $LG(K^{\# n},q,q^{\alpha})$ is bounded from above by a multiple of the genus of $K$. Indeed, $K^{\# n}$ has an admissible diagram with $n$ times more strips than $D(K)$. But $D(K)$ has $2 \times g(K)$ strips, which means that
$$
span_{q^{2 \alpha}}(LG(K^{\# n},q,q^{\alpha})) \leq 2 \times \text{number of strips} = 4 \times n \times g(K)
.$$
\end{proof}




%
\end{document}